\documentclass[10pt,a4paper,openany,oneside]{article}
\usepackage[a4paper,left=3cm,right=3cm, top=3cm, bottom=3cm]{geometry}
\usepackage{times}

\usepackage[latin1]{inputenc}
\usepackage[english]{babel}
\usepackage{mathrsfs}
\usepackage{amssymb,amsmath,amsthm}
\usepackage{hyperref}
\usepackage[final]{showkeys}
\usepackage{graphicx}
\usepackage{subcaption}
\usepackage{eurosym}
\usepackage{color}
\usepackage{enumerate}
\usepackage{multirow}

\theoremstyle{plain}
\newtheorem{thm}{Theorem}[section]
\newtheorem{cor}[thm]{Corollary}
\newtheorem{lem}[thm]{Lemma}
\newtheorem{lemma}[thm]{Lemma}
\newtheorem{prop}[thm]{Proposition}
\newtheorem{proposition}[thm]{Proposition}

\newtheorem{example}{Example}
\theoremstyle{definition}

\theoremstyle{remark}
\newtheorem{remark}{Remark}

\newcommand{\R}{\mathbb{{R}}}
\newcommand{\N}{\mathbb{{N}}}
\newcommand{\alphastar}{\alpha_m^*}
\newcommand{\Qstar}{Q_*}

\begin{document}
	\title{Differences between fundamental solutions of general higher order elliptic operators and of products of second order operators}
	\author{Hans-Christoph Grunau\\
		{\normalsize Fakult\"{a}t f\"{u}r Mathematik}\\
		{\normalsize Otto-von-Guericke-Universit\"{a}t}\\
		{\normalsize Postfach 4120, 39016 Magdeburg, Germany}\\
		\texttt{\small hans-christoph.grunau@ovgu.de}
		\and
		Giulio Romani\\
		{\normalsize Institut f\"{u}r Mathematik}\\
		{\normalsize Martin-Luther-Universit\"{a}t Halle-Wittenberg}\\
		{\normalsize 06099 Halle (Saale), Germany}\\
		\texttt{\small giulio.romani@mathematik.uni-halle.de}
		\and
		Guido Sweers\\
		{\normalsize Department Mathematik/Informatik}\\
		{\normalsize Universit\"at zu K\"oln}\\
		{\normalsize Weyertal 86-90, 50931 K\"oln, Germany}\\
		\texttt{\small gsweers@math.uni-koeln.de}
	}
	
\maketitle

	\begin{abstract}
		We study fundamental solutions of elliptic operators of order $2m\geq4$ with constant coefficients
		in large dimensions $n\ge 2m$, where their singularities become unbounded. For compositions of second order
		operators these can be chosen as convolution products of positive singular functions, which are positive themselves.
		As soon as $n\geq3$, the polyharmonic operator $(-\Delta)^m$ may no longer serve as a prototype for the general elliptic operator. It is known from examples of Maz'ya-Nazarov
		\cite{MazyaNazarov} and Davies \cite{Davies} that in dimensions $n\ge 2m+3$ fundamental solutions of specific operators of order $2m\geq4$
		may change sign near their singularities: there
		are ``positive'' as well as ``negative'' directions along which the fundamental solution tends to $+\infty$ and $-\infty$ respectively, when approaching its pole.
		In order to understand this phenomenon systematically we first show that existence of a ``positive'' direction directly follows 
		from the ellipticity of the operator. We establish an inductive argument by space dimension which shows that sign change in some dimension implies sign change in any larger dimension for suitably constructed operators. Moreover, we
		deduce for $n=2m$, $n=2m+2$ and for all odd dimensions an explicit closed expression for the fundamental solution in terms of its symbol. {From such formulae it becomes clear that the sign of the fundamental solution for such operators depends on the dimension. Indeed, we}  show that we have
		even sign change for a suitable operator of order $2m$ in dimension $n=2m+2$. On the other hand we show  that in the dimensions $n=2m$ and $n=2m+1$ the fundamental solution of any such elliptic operator is always positive around its singularity.
	\end{abstract}
	\section{Introduction and  main results }

	\paragraph{General constant coefficients elliptic operators.}

	We focus our attention to uniformly elliptic operators of order $2m$ with constant coefficients
	which involve only the highest order derivatives, namely
	\begin{equation}\label{eq:ell_op}
	L =(-1)^m Q\bigg(\frac{\partial}{\partial x_1},\cdots,\frac{\partial}{\partial x_n}\bigg)
	=(-1)^m \sum_{i_1,\ldots,i_{2m}\\=1,\ldots ,n}A_{i_1,\ldots, i_{2m}}\,\frac{\partial}{\partial x_{i_1}}\cdots\frac{\partial}{\partial x_{i_{2m}}},
	\end{equation}
	where the $2m$-homogeneous characteristic polynomial
	\begin{equation*}
	Q(\xi)=\sum_{i_1,\ldots,i_{2m}\\=1,\ldots, n}A_{i_1,\ldots, i_{2m}}\,\xi_{i_1}\cdots\xi_{i_{2m}}.
	\end{equation*}
	is called (possibly up to a sign) the \emph{symbol}   of the operator.

	Uniform ellipticity means then that $Q$ is  strictly positive on the unit sphere, i.e. there exists a constant $\lambda >0$
	such that
	$$
	 \forall \xi\in\mathbb{R}^n:\quad Q(\xi)\ge \lambda |\xi|^{2m}.
	$$

	\paragraph{Fundamental solutions.}

	In order to construct and to understand solutions $u$ to the differential equation $L u=f$ for a given right-hand side
	$f$, one introduces the concept of a fundamental solution $K_{L }(x,\,.\,)$ for any ``pole'' $x\in \mathbb{R}^n$ which is defined
	as a solution to the equations $L ^* K_{L }(x,\,.\,)=\delta_x$ and $L K_{L}(\,.\,,x)=\delta_x$
	in the distributional sense where $\delta_x$ is the $\delta$-distribution
	located at $x$. This means that for any test function $\psi \in C^\infty_0(\mathbb{R}^n)$ one has
	$$
	\int_{\mathbb{R}^n} K_{L}(x,y) L  \psi (y) \, dy =\psi(x),\qquad \int_{\mathbb{R}^n} K_{L }(y,x) L^* \psi (y)\, dy =\psi(x)
	$$
	with
	$$
	L^* =(-1)^m \sum_{i_1,\cdots,i_{2m}\\=1,\cdots n}\,\frac{\partial}{\partial x_{i_1}}\cdots\frac{\partial}{\partial x_{i_{2m}}}\,A_{i_1,\cdots i_{2m}}
	$$
	being the adjoint operator of $L $. Because $L$ has only constant coefficients and only of the highest even
	order $2m$, we have that $L =L^*$. Moreover, we
	may achieve that
	\begin{equation}\label{pole0}
	K_{L}(x,y) =K_{L}(0,x-y)=K_{L}(0,y-x)=K_{L}(y,x).
	\end{equation}
	For given $f\in C^\infty_0(\mathbb{R}^n)$, any fundamental solution yields a
	solution to the differential equation $L u=f$ in $\R^n$ by putting
	$$
	u(x):=\int_{\mathbb{R}^n} K_{L} (x,y) f(y)\, dy.
	$$
	One should also notice that, if a fundamental solution exists, it is not unique: one may add any smooth solution of $L v=0$, namely
	$\tilde{K}_{L}(x,y) =K_{L}(x,y) + v(x-y)$ yields another fundamental solution.

	\paragraph{Green functions.}

	When the problem $L u=f$ is considered in a sufficiently smooth  bounded  domain,  one may still obtain solution and even
	representation formulae by means of suitable fundamental solutions.
	Indeed, let $\Omega\subset\mathbb R^n$ be a bounded smooth domain and consider the problem
	\begin{equation}\label{sys}
	\begin{cases}
	L u=f\quad&\mbox{in }\Omega,\\
	B(u)=0\quad&\mbox{on }\partial\Omega,
	\end{cases}
	\end{equation}
	where $f\in C^{0,\gamma}(\overline{\Omega})$ and the boundary conditions verify a complementing condition, see \cite{ADN}.
	As a typical and most frequently studied prototype one may think of Dirichlet boundary conditions
	$$
	B_D(u):=(u,\partial_\nu, \ldots,\partial_\nu^{m-1}\nu)=0\quad\mbox{on }\partial\Omega,
	$$
	with $\nu$ the exterior unit normal at $\partial\Omega$.
	If there exists a unique solution $h_{L,\Omega,B}(x,\cdot)$ of the boundary value problem (recall that $L^*=L$)
	\begin{equation*}
	\begin{cases}
	L h_{L,\Omega,B}(x,\cdot)=0\quad&\mbox{in }\Omega,\\
	B(h_{L,\Omega,B}(x,\cdot))=-B(K_{L}(|x-\cdot|))\quad&\mbox{on }\partial\Omega,
	\end{cases}
	\end{equation*}
	one can define the so called \emph{Green function} for problem \eqref{sys}, given by
	\begin{equation*}
	G_{L ,\Omega,B}(x,y)=K_{L}(|x-y|)+h_{L,\Omega,B}(x,y).
	\end{equation*}
	Then  the unique solution of \eqref{sys} is given by
	\begin{equation*}
	u(x)=\int_{\Omega}G_{L,\Omega,B}(x,y)~f(y)~dy.
	\end{equation*}
	Notice that in general it is not straightforward to infer the existence of such $h_{L,\Omega,B}$.
	However, exploiting the general elliptic theory of Agmon, Douglis, and Nirenberg \cite{ADN} this is always possible in our special case when the operator $L$ has only constant coefficients
	of highest order, if Dirichlet boundary conditions  $B=B_D$ are imposed and the domain is $C^{2m,\gamma}$-smooth.

	In this case, one also infers by standard estimates that the function $h_{L,\Omega,B}$ is regular in $\overline\Omega$.
	Since in large dimensions fundamental solutions have a singularity near the pole, it becomes clear that, in order to
	understand $G_{L ,\Omega,B}$, we need first to investigate the behaviour of fundamental solutions.

	\paragraph{Positivity questions.}

	Positivity properties for $G_{L ,\Omega,B}$ concern the question whether a positive right-hand side
	yields a positive solution: if $u$ is a solution of (\ref{sys}), does it hold that $f\ge 0 \Rightarrow u\ge 0$ ?
	One often expects such a behaviour for physical or geometrical reasons. However, for equations of order
	at least 4, such a positivity preserving property will fail in general, see \cite{GGS} for  historical
	remarks and detailed references. This question concerns a \emph{nonlocal}
	behaviour of the full boundary value problem and often the influence of boundary conditions spoils
	the expected positivity. However, physically, one would hope that when applying an extremely concentrated
	right-hand side -- a $\delta$-distribution -- then close to this point the solution should respond in
	the same direction. This leads to the related but relaxed \emph{local} question: Is a suitable fundamental solution
	to the differential equation positive, at least close to its pole? This question is reasonable only for large
	dimensions $n\ge 2m$ because only here, fundamental solutions become unbounded and they are unique only up
	to locally bounded regular solutions of the homogeneous equation. If $n > 2m$ one may achieve uniqueness
	of the fundamental solution by imposing zero (Dirichlet) boundary conditions at infinity. In this case $K_L$ may be considered
	as the Green function $G_{L,\mathbb{R}^n,B_D}$ in the whole space. This means that one considers
	just the behaviour of the differential equation and disregards the influence of possible boundary conditions
	(being infinitely far apart).
	\vskip0.4truecm

	\paragraph{Previous results.}
	In the context of second order equations ($m=1$), both local and nonlocal behaviours are well established.
	Indeed, within the class of constant coefficients operators, the Laplacian $-\Delta$ is, up to a change of coordinates, the only such operator.
	Its fundamental solutions are known  explicitly and in particular they are positive
	(if $n=2$, at least close to the pole). Moreover, the maximum principle holds for such operators, so positive data yield positive
	solutions (see \cite{GilbargTrudinger}). In other words, the Green function is always positive.

	\vskip0.2truecm
	When one moves to the higher order setting ($m\geq2$), several differences arise, even for $(-\Delta)^m$ or, equivalently, for powers
	of second order operators with constant coefficients.

	Indeed, if one investigates the positivity preserving property in bounded domains, then the answer is largely affected by the choice
	of boundary conditions. As an example, on the one hand, with Navier boundary conditions ($u=\Delta u=\cdots=\Delta ^{m-1}u=0$) one
	may rewrite the problem as a second order system and thus the maximum principle implies positivity. On the other hand, this tool is in
	general not available when dealing with Dirichlet boundary conditions ($u=\partial_\nu u=\cdots=\partial_\nu^{m-1}u=0$)
	and one cannot expect positivity, in general not even in \emph{convex} bounded smooth domains (see \cite{Garabedian}).
	Nevertheless, positivity holds in balls and their
	small smooth deformations (see \cite{Boggio,GrunauRobert}). We refer to \cite{GGS} for an extensive survey of the topic.

	However, within that class of powers of second order operators,
	if one restricts to a ``local" question, meaning the positivity of Green functions
	under Dirichlet boundary conditions
	near the pole, the answer is still affirmative. Indeed, a uniform local positivity can be proved, namely the existence of  constants $r_{m,\Omega} > 0,\delta_{m,\Omega}>0$ such that
	$G_{(-\Delta)^m,\Omega, B_D}(x,y)>\delta_{m,\Omega} >0$ for all $x,y\in\Omega$ with $|x-y|<r_{m,\Omega}$. This means that the negative
	part and the singularity of the Green function are uniformly apart.

	A consequence of that result is that the size of negative part of the Green function, if present at all, is small compared
	to its positive part. Indeed, concerning Dirichlet problems, positivity for a rank-1-correction of the polyharmonic
	Green function is retrieved, namely
	$$G_{(-\Delta)^m,\Omega, B_D}(x,y)+c_{m,\Omega}d_\Omega(x)^md_\Omega(y)^m\geq 0,$$
	where $d_\Omega$ denotes the distance to the boundary and $c_{m,\Omega}$ is a sufficiently large positive constant, see \cite{GrunauRobert, GRS}.

	These results have been extended later on  by Pulst in his PhD-dissertation \cite{Pulst} for formally selfadjoint
	positive definite operators of order $2m$ with the polyharmonic 
	operator $(-\Delta)^m$
	or an $m$-th power of a second order elliptic operator  with constant coefficients as  the leading term. Lower order terms are permitted
    provided they  can be written in divergence form and have sufficiently smooth  and uniformly bounded coefficients.
	\vskip0.2truecm
	In two dimensions, i.e. $n= 2$, the symbol $Q$ with real coefficients can be split into $2m$ linear terms. Combining mutually conjugate pairs $(\xi_1+a_k \xi_2)$ and $(\xi_1+\overline{a_k} \xi_2)$ of these linear terms  with nonreal $a_k $
	we see that
	$$
	Q(\xi) =c\prod^m_{k=1}\left(\xi_1^2+\left(a_k+\overline{a_k}\right)\xi_1\xi_2+|a_k|^2 \xi_2^2\right)
	$$
	is a product  of second order symbols.

	However, in dimensions $n> 2$ powers of second order operators are \textit{not} the prototype of a general
	operator $L $ of order $2m$, not even in the case of constant coefficients.
	Moreover,  it is in general not possible to rewrite $L $ as an $m$-fold composition of (possibly different)
	second order operators.
	Indeed, let us simply consider the case of a homogeneous fourth order operator with a symbol of the kind
	$$Q(x,y,z)=x^4+y^4+z^4+\sum_{\substack{i+j+k=4\\0\leq i,j,k\leq3}} c_{i,j,k}x^iy^jz^k$$
	and suppose that it is the product of two second order polynomials $q_1,q_2$. One may assume that both polynomials have their
	coefficients in front of $x^2$ equal to $1$ and then, they would  necessarily be of the kind
	$$q_1(x,y,z)=x^2+cy^2+dz^2+a_1xy+a_2xz+a_3yz$$
	$$q_2(x,y,z)=x^2+\frac1cy^2+\frac1dz^2+b_1xy+b_2xz+b_3yz.$$
	The smooth map from $(0,\infty)^2\times \mathbb{R}^6$ into the 12-dimensional vector space of such symbols $Q$ which maps
	$$(c,d,a_1,a_2,a_3,b_1,b_2,b_3)\mapsto q_1\cdot q_2$$
	is \emph{not surjective}.

	Concerning explicit formulae and (local) positivity properties of fundamental solutions of such
	general elliptic operators only  little is known.
	Existence of fundamental solution is shown in \cite{John} in a very general framework,
	and rather involved formulae are obtained. In the particular case of a \textit{$2m$-homogeneous higher order uniformly elliptic
	operator with constant coefficients}, different implicit expressions have been found according to the parity of the dimension $n$. In what follows
	we always assume that
	$$
	n\ge 2m.
	$$
	For \textit{odd} $n$, the general formula for a fundamental solution \cite[(3.44)]{John} simplifies as
	\begin{equation}\label{formulaJohn}
	K_{L}(x,y)=-\dfrac{1}{4(2\pi )^{n-1}}(-\Delta_y)^{\frac{n+1-2m}2}\int_{\substack{|\xi|=1}}\dfrac{|(x-y)\cdot \xi|}{Q(\xi)}d\mathcal{H}^{n-1}(\xi)
	\end{equation}
	(from \cite[(3.54)]{John}), while for \textit{even} $n$ one has  (see \cite[(3.62)]{John})
	\begin{equation}\label{formulaJohnEVEN}
	K_{L}(x,y)=- \frac1{(2\pi )^n} (-\Delta_y)^{\frac{n-2m}2}\int_{|\xi|=1} \frac{\log |(x-y)\cdot \xi|}{Q(\xi)} \,d\mathcal{H}^{n-1}(\xi).
	\end{equation}
	We recall that $Q$ denotes the symbol (possibly up to a sign) of the operator $L $. On the other hand, motivated by questions in potential and Schr\"odinger semigroup theory, respectively, and without referring to (\ref{formulaJohn}), (\ref{formulaJohnEVEN}) or even \cite{John},
	Maz'ya-Nazarov \cite{MazyaNazarov} and Davies \cite{Davies} found examples of elliptic operators of order $2m\ge 4$ in dimensions $n\ge 2m+3$ with sign changing fundamental solutions.
	The precise range of dimensions where this phenomenon may be observed
	remained open as well as a systematic study, see \cite[p. 85]{Davies}:
	``It seems to be difficult to find a useful characterization of the symbols
	of those constant coefficient elliptic operators with this property.''

	\paragraph{Aim and results.}
	The aim of this paper is a systematic investigation of the behaviour of fundamental solutions - and in particular whether or not they are
	positive close to the pole -
	for this class of uniformly elliptic operators of order $2m$ with constant coefficients.
	
	We find the above mentioned examples of sign changing  fundamental solutions  somehow unexpected because this means that even
	when applying  a right hand side, which is concentrated at some point and points into one direction, the response of any solution to
	the differential equation
	will be sign changing and so - in some regions arbitrarily close to this point -
	in opposite direction to the right hand side. Indeed,
	we show in Theorem~\ref{thm:pos_dir} that ``positivity'' is somehow the
	expected  behaviour  related to ellipticity.
	
	In Section~\ref{subsec:inductive_argument} we establish an inductive argument
	by space dimension. Roughly speaking this says that for understanding
	in any dimension whether one finds operators with sign changing
	fundamental solutions it suffices to understand the behaviour in ``small''
	dimensions.
	
	In Section~\ref{sec:Guido} we calculate in almost closed form fundamental
	solutions for some specific fourth order elliptic operators in any dimension 
	$n\ge 5$. For a specific direction we even find a very simple closed expression. {From} $n=6$ on, we observe sign change. While for 
	$n=6,7$ we need a nonconvex symbol, for $n\ge 8$ even convex symbols are admissible.
	The examples presented here use similar symbols as in \cite{Davies} and \cite{MazyaNazarov}, 	but they are constructed
	with the help of a  different method based on the Fourier transform and the residue theorem. For $n=6$ even the observation is new -- to the best of our knowledge.
	
	These examples  yield the first important result.
	
	\begin{thm}\label{thm:negative_directions_any_dimension}
	 For $m=2$ and any $n\ge 6$ there exists a uniformly elliptic (fourth order) operator $L$ with constant coefficients such that the corresponding fundamental solution $K_L$
	 is sign changing for $|x-y|\to 0$.
	\end{thm}
     For the proof, see Theorem~\ref{thm:pos_dir} and Proposition~\ref{prop:4thorder_signchange}.
	
	\vskip0.2truecm
	
	In order to answer the question asked by Davies for a systematic understanding of this phenomenon, we find in Section~\ref{sec:explicit_calc_FS} explicit formulae for
	fundamental solutions from which it becomes clear which kind of elliptic symbols yield
	positive and sign changing fundamental solutions, respectively.

	\vskip0.2truecm

    In \textit{odd dimensions} we find the following general elegant formula.
	\begin{thm}\label{TheoremODD}
		Let $n\geq 2m+1$ be odd. Then, the fundamental solution $K_{L}$ is given by
		\begin{equation*}
		K_{L}(x,y)=\frac{(-1)^{\frac{n-2m-1}2}}{2^n\pi^{n-1}|x-y|^{n-2m}}\int_{\substack{|\xi|=1\\(x-y)\cdot\xi\,=\,0}}
		\nabla^{n-2m-1}\frac1{Q(\xi)}\left(\frac {x-y}{|x-y|}^{\otimes\,n-2m-1}\right)d\mathcal{H}^{n-2}(\xi).
		\end{equation*}
	\end{thm}
	\noindent Here, if $T$ is a $j$-multilinear form and $v$ is a vector, we  use the compact tensorial notation
	\begin{equation*}
	T\left(v^{\otimes j}\right):=T(\underbrace{v, v, \cdots, v}_{j-\text{times}}),
	\end{equation*}
	so, in particular,
	\begin{equation*}
	\nabla^{j}f(\xi)\left(v^{\otimes j}\right)=\sum_{h_1,\cdots,h_{j}\,=\,1}^{n}\frac{\partial^{j}f}{\partial_{h_1}
	\cdots\partial_{h_j}}(\xi)\,v_{h_1}\cdots v_{h_j}.
	\end{equation*}
	To avoid redundant parenthesis we  write $\dfrac y{|y|}^{\otimes2j}: =\bigg(\dfrac y{|y|}\bigg)^{\otimes2j}$.
	Theorem \ref{TheoremODD} is proved in Section~\ref{subsec:explicit_calc_FS_odd} and it follows directly from Theorem~\ref{conjecture_j}.

	In \textit{even dimensions}, due to the presence of the logarithm in \eqref{formulaJohnEVEN},
	we are not able to achieve a comparable compact result, computations being much more involved. However, we show a related formula for the first
	``critical'' dimension $n=2m+2$.

	\begin{thm}\label{TheoremEVEN}
		Let $n=2m+2$. Then, the fundamental solution $K_{L}$ is given by
		\begin{equation*}
		\begin{split}
		K_{L}(x,y)&=\frac{1}{2^{2m+1}\pi^{2m+2}|x-y|^2}\,\,\bigg\{\,\frac{n-2}2 \int_{|\xi|=1}\frac{1}{Q(\xi)}\,d\mathcal{H}^{n-1}(\xi)\\
		&\quad+\int_{|\xi|=1 \atop (x-y)\cdot \xi>0} \log\left(\xi\cdot \frac{x-y}{|x-y|}\right)\bigg[(4m+4-n) \bigg(\xi\cdot \frac{x-y}{|x-y|}\bigg)\nabla \frac{1}{Q(\xi)}\left( \frac{x-y}{|x-y|}\right)\\
		&\quad+2m\frac{1}{Q(\xi)} + \nabla^2\frac1{Q(\xi)}\left(\frac {x-y}{|x-y|}^{\otimes2}\right)
		+\left(\xi\cdot \frac{x-y}{|x-y|}\right)^2\Delta \frac{1}{Q(\xi)}
		\bigg]\,d\mathcal H^{n-1}(\xi)\bigg\}.
		\end{split}
		\end{equation*}
	\end{thm}
	The proof is given in Section \ref{Section_EVEN_proof}.
	
	\vskip0.2truecm
	
	The difference between even and odd dimensions here reminds us somehow of the same dinstinction for the wave equation. In Theorem \ref{TheoremEVEN} (even dimensional) the integration is carried out over a one-codimensional surface with a weight function, which becomes infinite at its boundary. On the other hand, in Theorem \ref{TheoremODD} (odd dimensional) the integration is carried out over the boundary of this surface, i.e. a 2-codimensional surface. The method of descent, as outlined in Section~\ref{subsec:inductive_argument}, gives further support to this observation.
	
	\vskip0.2truecm
	
	Theorem~\ref{TheoremODD} and Formula (\ref{formulaJohnEVEN}) allow for a first
	interesting result concerning \emph{positivity} of fundamental solutions.
	
	\begin{cor}\label{cor:pos_fund_sol} {}\hspace*{4cm}
	 \begin{enumerate}[i)]
			\item If $n=2m$, the fundamental solution $K_{L}$ is positive for $|x-y|\in B_{r_L}(0)$ with some $r_L>0$.
			\item If $n=2m+1$, the fundamental solution $K_{L}$ is always positive.
		\end{enumerate}	
	\end{cor}

	Theorem~\ref{TheoremODD} shows further which kind of symbol $Q$ will e.g.
	in space dimension $n=2m+3$ yield $K(0,e_1)<0$. For this one needs
	$\partial_1^2 \frac{1}{Q} (0,\xi')>0$ for $|\xi'|=1$. This would follow
	e.g. from $\partial_1 {Q} (0,\xi')=0$ and $\partial_1^2 {Q} (0,\xi')<0$,
	i.e. (only) from a nonconvex shape of the level set of $Q$ in these points.

	For an example of this kind see 
	\cite[p. 100]{Davies} and pp. 20-23 of a preliminary preprint version
	of this article which can be found at \href{https://arxiv.org/abs/1902.06503v1}{arXiv:1902.06503v1}.
	
	The situation is similar but not this clear in the even dimension $n=2m+2$,
	due to the relatively higher dimensional domain of integration and  to the presence of further terms. However, thanks to the logarithmic singularity
	one may expect that also here, a nonconvex symbol may work. Indeed we prove
	in Section~\ref{Section_EVEN} the following result.

	\vskip0.2truecm

	\begin{thm}[Examples of sign changing fundamental solutions, the even dimensional case]\label{TheoremEVEN_Examples}
	
			For any $m\ge 2$ and $n=2m+2$ there exists an elliptic symbol $Q$ of order $2m$ such that the fundamental solution of the associated operator
			$L_Q$ is sign changing for $|x-y|\to 0$.
    \end{thm}

	In view of Section~\ref{subsec:inductive_argument} and Theorem~\ref{thm:pos_dir} we may immediately conclude the following
	general result.
	
	\begin{cor}\label{cor:sign_changing_fund_sol}
            For any $m\ge 2$ and $n\ge 2m+2$ there exists an elliptic symbol $Q$ such that the fundamental solution of the associated operator
			$L_Q$ is sign changing for $|x-y|\to 0$.
	\end{cor}
	
	Together with Corollary~\ref{cor:pos_fund_sol} we have so obtained a complete picture of positivity and change of sign, respectively, in all ``singular''
	dimensions $n\ge 2m$.

	\paragraph{Notation.}
	We denote the partial derivative as $D^\alpha$ or $\partial_\alpha$ or $\frac{\partial}{\partial\alpha}$, where $\alpha$ is a multi-index, with the convention that if $D^0u=u$ for any function. Moreover, if $j\in\N$, $\nabla^ju$ stands for the tensor of the $j$-th derivatives. Finally, we denote by $\mathcal H^k$ the $k$-th dimensional Hausdorff measure.

	
	\section{Basic observations}\label{sec:basic_observations}

	In this section we consider only the case $n>2m$ of
    large dimensions.
	In what follows $L$ always denotes a uniformly elliptic operator
    as in (\ref{eq:ell_op}) with constant coefficients and only of
    highest order $2m$. $K_L$ denotes John's  fundamental solution as it is given
    in (\ref{formulaJohn}) and (\ref{formulaJohnEVEN}), respectively.
    By \eqref{pole0}, without loss of generality, we may consider $0$ as the pole of $K_L$.

	\subsection{Homogeneity, decay and uniqueness of John's fundamental solution}
    
    \begin{lem}\label{lemma:A1.1}
     For $\sigma\in\mathbb{R}\setminus \{0\}$ and
     $x\in\mathbb{R}^n\setminus \{0\}$ we have
     \begin{equation}\label{eq:app:1.1}
     K_L(0, \sigma x)=|\sigma |^{2m-n}K_L(0,  x).
     \end{equation}
     In particular this yields for all $x\in\mathbb{R}^n\setminus \{0\}$
     and all multi-indices $\alpha\in\mathbb{N}^n_0$
     \begin{equation}\label{eq:app:1.2}
     | D^\alpha K_L(0,  x)| \le C_\alpha |x|^{2m-n-|\alpha|}.
     \end{equation}
    \end{lem}

    \begin{proof}
     In case of odd $n>2m$, (\ref{formulaJohn}) shows that $K_L$
     can be obtained as $(-\Delta)^{(n+1-2m)/2}$ of a $1$-homogeneous function.
     In case of even $n\ge 2m+2$, the proof of Theorem~\ref{TheoremEVEN}
     in Section~\ref{Section_EVEN_proof} shows that $K_L$ can be obtained
     $(-\Delta)^{(n+1-2m)/2}$ of a $(-2)$-homogeneous function.
     $K_L(0,  x)=K_L(0,  -x)$ follows from the corresponding property of the symbol.
    \end{proof}


     \begin{prop}\label{prop:app:2.1}
     Let  $K_L$ and $\tilde{K}_L$ be two fundamental solutions for $L$
    which both obey (\ref{eq:app:1.1}). Then
    $$
    K_L (0,x) \equiv \tilde{K}_L (0,x).
    $$
    \end{prop}

    \begin{proof}
     Defining $u(x):=K_L(0, x)- \tilde{K}_L(0,x)$, we find a solution of $Lu=0$
    in $\mathbb{R}^{n}$ which thanks to elliptic regularity theory satisfies
    $u\in C^\infty (\mathbb{R}^{n}) $. To see this one may combine local elliptic
    $L^p$-estimates (see \cite[Theorem 15.1'']{ADN}), the difference quotient method as outlined in \cite[Section 7.11]{GilbargTrudinger} and a bootstrapping argument.
    Since both  $K_L$ and $\tilde{K}_L$
    satisfy (\ref{eq:app:1.1}) we find that for any $x\in\mathbb{R}^n\setminus\{0\}$ and any $\sigma\in \mathbb{R}\setminus\{0\}$:
    $$
    u(x)=|\sigma|^{n-2m}u(\sigma x).
    $$
    Since $n>2m$ we conclude by continuity of $u$ in $0$ from letting $\sigma\to 0$ that
    $
    u(x)\equiv 0.
    $
    \end{proof}

	\subsection{Ellipticity and positive directions}\label{subsec:ell_pos_dir}

	We  prove the existence of ``positive'' directions
	(observe our sign convention for ellipticity) for the fundamental
	solutions which is somehow the simpler case and which one expects from the
	notion of ellipticity.

	\begin{thm}\label{thm:pos_dir}
	Assume that $n>2m$, $L$ is a uniformly  elliptic operator
	with constant coefficients of order $2m$ as introduced in (\ref{eq:ell_op})   and consider the
	$(2m-n)$-homogeneous fundamental solution $K_L$ according to
	(\ref{formulaJohn}) and (\ref{formulaJohnEVEN}), respectively. Then there exists $y\in\mathbb{R}^n\setminus\{ 0\}$
	such that
	$$
	K_L(0,y)>0.
	$$
	\end{thm}

	\begin{proof}
	We assume by contradiction that
 	$$
 	\forall y\in\mathbb{R}^n\setminus\{ 0\}:\quad K_L(0,y)\le 0.
	 $$
 	Certainly, $K_L(0,y)\not\equiv 0$. By continuity there exists a nonempty open set $\Omega\subset\mathbb{S}^{n-1}$ such that we
 	have
 	$$
	 \forall y\in {\mathcal C}_\Omega:\quad K_L(0,y) < 0
	 $$
	 on the corresponding cone
	 $$
	 {\mathcal C}_\Omega:= \{ r\eta:\ \eta\in \Omega,\ r \in\mathbb{R}\setminus\{ 0\} \}.
	 $$
	 We consider a fixed radially symmetric $\varphi\in C^\infty_0 (\mathbb{R}^n)$ with
	 $$
	 0\le \varphi (x)\le 1;\qquad \varphi (x)> 0 \Leftrightarrow |x| \le \frac{1}{2}.
	 $$
	 We introduce a corresponding solution (defined in the whole space) of the differential equation by
	 $$
 		U(x):= \int_{\mathbb{R}^n} K_L(x,y)\varphi (y)\, dy
 		    =\int_{\mathbb{R}^n} K_L(0,y)\varphi (x-y)\, dy.
	 $$
	 Since for $x\in \overline{B_{1/2} (0)}$ the intersection $(x+{\mathcal C}_\Omega )\cap B_{1/2} (0)$
	 is nonempty, $U$ is strictly negative there. By compactness we find a constant $C_0>0$ such that
	 $$
 	\forall x\in \overline{B_{1/2} (0)}:\quad U(x)\le - C_0 <0.
 	$$

	 Next, we introduce a scaling parameter
	 $
	 \sigma \in (0,1]
	 $
	 and consider for
	 $$
 	\varphi_\sigma(x):=  \varphi (x/\sigma)
 	$$
	 the solution of the corresponding Dirichlet problem in $\overline{B_1(0)}$
 	$$
	 \overline{B_1(0)}\ni x\mapsto u_\sigma (x):=\int_{B_1 (0)} G (x,y) \varphi_\sigma(y)\, dy.
	 $$
	 Here,
	 $$
	 G (x,y):= G_{L,B_1 (0)} (x,y) =:K_L(x,y) + h (x,y)=K_L(x,y) + h_{L,B_1 (0)} (x,y)
	 $$
	 denotes the corresponding Green function and its decomposition into fundamental solution and
	 regular part. By continuous dependence on parameters and general elliptic theory (see \cite{ADN}) we find that
 	$$
	 \forall (x,y) \in \overline{B_{1/2} (0) \times B_{1/2} (0) }:\quad | h (x,y) |\le C_1
	 $$
	 with a suitable constant $C_1$. In what follows we consider only $x\in \overline{B_{1/2} (0)}$.
	 By the $(2m-n)$-homogeneity of the fundamental solution we obtain:
 	\begin{align*}
 	 u_\sigma (\sigma x) &= \sigma^n \int_{B_{1}(0)}\left(  K_L(\sigma x,\sigma y)  +h(\sigma x,\sigma y) \right) \varphi(y)\, dy \\
 	 &= \sigma^{2m} U(x) + \sigma^n \int_{B_{1/2}(0)} h(\sigma x,\sigma y)  \varphi(y)\, dy \\
	  &\le - C_0 \sigma^{2m} + \sigma^n C_1 |B_{1/2}(0)|
 	 \le - \frac{C_0}{2} \sigma^{2m} ,
	 \end{align*}
 	provided that $\sigma \in (0,1]$ is chosen small enough. We fix such a suitable parameter and keep the
 	corresponding $u_\sigma$ and $ \varphi_\sigma$ fixed. We recall that we have shown:
 	$$
 	\varphi_\sigma (x) > 0 \mbox{\ in\ } B_{\sigma/2}(0),\quad
 	\varphi_\sigma (x) = 0 \mbox{\ outside\ } B_{\sigma/2}(0),\quad
 	u_\sigma (x) < 0 \mbox{\ in\ } B_{\sigma/2}(0).
 	$$
	This yields (we recall that $\lambda $ denotes the ellipticity constant of $L$)
	\begin{align*}
	0&> \int_{B_{\sigma/2} (0)} u_\sigma (x) \varphi_\sigma (x)\, dx
	=\int_{B_1 (0)} u_\sigma (x) \varphi_\sigma (x) \, dx
	=\int_{B_1 (0)} u_\sigma (x) \left( L u_\sigma (x) \right)\, dx\\
	&\ge \lambda \| u_\sigma \|^2_{H^m_0 (B_1 (0))}>0,
	\end{align*}
	a contradiction. In the last step we used the elementary form of
	G\aa{}rding's inequality (see \cite{Garding}) for operators,
	which have only constant coefficients and only  of highest order, which follows
	from the ellipticity condition by employing the Fourier transform.
	\end{proof}
    An alternative proof would follow from Corollary~\ref{cor:pos_fund_sol}
    and the inductive argument of Proposition~\ref{prop:app:4.1} below.

	\subsection{An inductive argument}
	\label{subsec:inductive_argument}

     For simplicity we write in the remainder of this section
    $$
    K_L(x):=K_L(0,x)=K_L(x,0).
    $$
    In what follows we always assume that $n-1>2m$, i.e. that
    $$
    n >2m+1.
    $$

   \subsubsection{A method of descent with respect to space dimension}\label{subsec:app:3}

    Here we use  the notation
    $$
    x=(x',x_n)\in \mathbb{R}^{n-1} \times \mathbb{R}=\mathbb{R}^{n}.
    $$
    Let $L_n $ be an elliptic operator in $\mathbb{R}^n$ as in (\ref{eq:ell_op})
    \begin{equation*}
	L_n =(-1)^m Q_n\bigg(\frac{\partial}{\partial x_1},\cdots,\frac{\partial}{\partial x_n}\bigg)
	=(-1)^m \sum_{i_1,\ldots,i_{2m}\\=1,\ldots, n}A_{i_1,\ldots, i_{2m}}\,\frac{\partial}{\partial x_{i_1}}\cdots\frac{\partial}{\partial x_{i_{2m}}},
	\end{equation*}
	with the symbol
	\begin{equation*}
	Q_n(\xi)=\sum_{i_1,\ldots,i_{2m}\\=1,\ldots, n}A_{i_1,\ldots, i_{2m}}\,\xi_{i_1}\cdots\xi_{i_{2m}}.
	\end{equation*}
	From this we obtain an operator in $\mathbb{R}^{n-1}$ by simply ``forgetting''
	the $x_n$-coordinate or by considering only functions, which do not depend
	on $x_n$:
    \begin{equation}\label{L_n-1}
	L_{n-1} =(-1)^m Q_{n-1}\bigg(\frac{\partial}{\partial x_1},\cdots,\frac{\partial}{\partial x_{n-1}}\bigg)
	=(-1)^m \sum_{i_1,\ldots,i_{2m}\\=1,\ldots, (n-1)}A_{i_1,\ldots, i_{2m}}\,\frac{\partial}{\partial x_{i_1}}\cdots\frac{\partial}{\partial x_{i_{2m}}},
	\end{equation}
	with the the corresponding symbol
	\begin{equation*}
	Q_{n-1}(\xi')=Q_{n}(\xi',0)=
	\sum_{i_1,\ldots,i_{2m}\\=1,\ldots,(n-1)}A_{i_1,\ldots, i_{2m}}\,\xi_{i_1}\cdots\xi_{i_{2m}}.
	\end{equation*}
	Since
	$$
	Q_{n-1}(\xi')=Q_{n}(\xi',0)\ge \lambda |(\xi',0)|^{2m}=\lambda |\xi'|^{2m},
	$$
	$L_{n-1}$ is an elliptic operator. Next we define
    \begin{equation}\label{eq:app:3.1}
    K_{n-1}(x'):=\int_\mathbb{R}K_{n} (x',\xi_n)\, d\xi_n
    \end{equation}
    and aim at showing that this is John's (unique) fundamental solution for $L_{n-1}$.

    We prove first that we have the expected homogeneity and hence also the expected decay  at $\infty$.
    \begin{lem}\label{lem:app:3.1}
     For $\sigma\in\mathbb{R}\setminus \{0\}$ and
     $x'\in\mathbb{R}^{n-1}\setminus \{0\}$ we have:
     \begin{equation*}
     K_{n-1}(\sigma x')=|\sigma |^{2m-n+1}K_{n-1}( x').
     \end{equation*}
    \end{lem}
    \begin{proof}
     \begin{align*}
      K_{n-1}(\sigma x') &= \int_\mathbb{R} K_{n} (\sigma x',\xi_n)\, d\xi_n=|\sigma |\int_\mathbb{R} K_{n} (\sigma x',\sigma\xi_n)\, d\xi_n\\
      &=|\sigma |^{1+2m-n} \int_\mathbb{R} K_{n} (x',\xi_n)\, d\xi_n
      = |\sigma |^{1+2m-n}K_{n-1}( x').
     \end{align*}
    \end{proof}

    Next we prove that $K_{n-1}$ is in fact a fundamental solution for $L_{n-1}$.
    \begin{lem}\label{lem:app:3.2}
     For all $\varphi\in C^\infty_0 (\mathbb{R}^{n-1})$ we have that
     $$
     \int_{\mathbb{R}^{n-1}} \left( L_{n-1}\varphi (x') \right) K_{n-1}(x') \, dx'= \varphi(0).
     $$
    \end{lem}

    \begin{proof}
     For  $\varphi \in C^\infty_0 (\mathbb{R}^{n-1})$ we define
     $$
     \tilde{\varphi} (x',x_n):= \varphi (x')
     $$
     and find that
     \begin{align}
     \int_{\mathbb{R}^{n-1}} \left( L_{n-1}\varphi (x')\right)\,  K_{n-1}(x') \, dx'&=
     \int_{\mathbb{R}^{n-1}} \left( L_{n-1}\varphi (x')\right)\, \int_\mathbb{R} K_{n}(x',x_n) \, dx'dx_n\nonumber \\
     &=
     \int_{\mathbb{R}^{n}} \left( L_{n}\tilde{\varphi} (x)\right)\,  K_{n}(x) \, dx.\label{eq:app:3.3}
     \end{align}
     In order to proceed we need to overcome the difficulty that
     $ \tilde{\varphi}\not\in C^\infty_0 (\mathbb{R}^{n})$ by a suitable
     approximation. To this end we choose
     $$
     \chi\in C^\infty_0 (\mathbb{R},[0,1]), \qquad
     \chi=\left\{ \begin{array}{ll}
                   1\quad \mbox{\ in\ } [-1,1],\\
                   0 \quad \mbox{\ outside\ } [-2,2],
                  \end{array}
          \right.
     $$
     and define
     $$
     \tilde{\varphi}_k (x):=\tilde{\varphi} (x)\chi(x_n/k)={\varphi} (x')\chi(x_n/k).
     $$
     We find
     \begin{align*}
     \lefteqn{
      \int_{\mathbb{R}^{n}} \left| L_{n}\tilde{\varphi} (x)-L_{n}\tilde{\varphi}_k (x)\right|\, \left| K_{n}(x)\right| \, dx
      \le \int_{\operatorname{supp}({\varphi})\times \mathbb{R}}(1-\chi(x_n/k)) \left| L_{n-1} {\varphi} (x')\right|\, \left| K_{n}(x)\right| \, dx'\, dx_n } &&\\
      & + \sum_{j=1}^{2m}C_j k^{1-j} \int_{\operatorname{supp}({\varphi})\times \left([-2,-1]\cup [1,2] \right)}\left| D^{2m-j} {\varphi} (x')\right|
      \, \left| K_{n}(x',kx_n)\right| \, dx'\, dx_n\\
      \le& C\int_{\operatorname{supp}({\varphi})\times \left((-\infty,k)\cup(k,\infty) \right)}\left| K_{n}(x)\right| \, dx
      + \sum_{j=1}^{2m}C_j k^{1-j}k^{2m-n}
      \le \sum_{j=0}^{2m}C_j k^{2m-n+1-j}\hspace*{1cm}\ \\
      \le & C k^{2m-n+1}\to 0 \quad \mbox{\ as\ }k\to\infty,
     \end{align*}
     because we assume that $n>2m+1$. With this we conclude from (\ref{eq:app:3.3})
   $$
     \int_{\mathbb{R}^{n-1}} \left( L_{n-1}\varphi (x')\right)\,  K_{n-1}(x') \, dx'=\lim_{k\to \infty } \int_{\mathbb{R}^{n}} \left( L_{n}\tilde{\varphi}_k (x)\right)\,  K_{n}(x) \, dx    
     = \lim_{k\to \infty }\tilde{\varphi}_k (0)=\varphi(0)
     $$
      as claimed.
    \end{proof}
     Combining Lemmas~\ref{lem:app:3.1} and \ref{lem:app:3.2} with the uniqueness
     result of fundamental solutions with suitable degree of homogeneity
     from Proposition~\ref{prop:app:2.1} we conclude:
     \begin{prop}\label{prop:app:3.1}
      $K_{n-1}$ as defined in (\ref{eq:app:3.1}) is John's fundamental solution for $L_{n-1}$ as it is given
    in (\ref{formulaJohn}) and (\ref{formulaJohnEVEN}), respectively.
     \end{prop}

    \subsubsection{Understanding ``small'' dimensions is sufficient}\label{subsec:app:4}

     \begin{prop}\label{prop:app:4.1}
      Assume that $n>2m+1$ and that for \emph{all} elliptic operators $L_{n-1}$ of the form (\ref{L_n-1})
      in $\mathbb{R}^{n-1}$ with  John's corresponding fundamental solution $K_{n-1}$ there exists a
      vector $x'\in \mathbb{R}^{n-1}\setminus \{0\}$ such that $K_{n-1}(x')>0$.
      Then for \emph{all} elliptic operators $L_{n}$ of the form (\ref{eq:ell_op})
      in $\mathbb{R}^{n}$ with  John's corresponding   fundamental solution $K_{n}$ there exists a
      vector $x=(x',x_n)\in \mathbb{R}^{n}\setminus \{0\}$ such that $K_{n}(x')>0$.
     \end{prop}

     \begin{proof}
      Let  $L_{n}$ be an arbitrary  elliptic operator of the form (\ref{eq:ell_op}) in $\mathbb{R}^{n}$
      with  corresponding  John's fundamental solution $K_{n}$. We define $L_{n-1}$ and $K_{n-1}$
      as in Subsection~\ref{subsec:app:3}. Then Prop.~\ref{prop:app:3.1} shows that $K_{n-1}$
      is John's corresponding   fundamental solution. Making use of (\ref{eq:app:3.1}), the assumption yields
      the existence of a vector $x'\in \mathbb{R}^{n-1}\setminus \{0\}$ such that
      $$
      0< K_{n-1}(x') = \int_\mathbb{R} K_{n} (x',\xi_n)\, d\xi_n.
      $$
      This shows that there exists a point $x_n\in  \mathbb{R}$ which satisfies
      $
      K_{n} (x',x_n)>0.
      $
     \end{proof}

     \begin{remark}\label{rem:app:4.1}
      Theorem~\ref{TheoremODD}(i) and Proposition~\ref{prop:app:4.1} yield a different
      proof of Theorem~\ref{thm:pos_dir} by means of the inductive procedure.
     \end{remark}

     \begin{prop}\label{prop:app:4.2}
      Assume that $n>2m+1$ and that there exists \emph{one} elliptic operator $L_{n-1}$ of the form (\ref{L_n-1})
      in $\mathbb{R}^{n-1}$ with   John's corresponding   fundamental solution $K_{n-1}$
      for which one finds a
      vector $x'\in \mathbb{R}^{n-1}\setminus \{0\}$ such that $K_{n-1}(x')<0$.
      Then there exists \emph{one} elliptic operator $L_{n}$ of the form (\ref{eq:ell_op})
      in $\mathbb{R}^{n}$ with  John's corresponding  fundamental solution $K_{n}$
      for which one finds a
      vector $x=(x',x_n)\in \mathbb{R}^{n}\setminus \{0\}$ such that $K_{n}(x')<0$.
     \end{prop}

     \begin{proof}
      Let  $L_{n-1}$ be an  elliptic operator of the form (\ref{L_n-1}) in $\mathbb{R}^{n-1}$ with symbol $Q_{n-1}$ and
      corresponding John's fundamental solution $K_{n-1}$ for which  one finds a
      vector $x'\in \mathbb{R}^{n-1}\setminus \{0\}$ such that $K_{n-1}(x')<0$. We define
      $$
      L_n:= L_{n-1}+\partial_n^{2m}
      $$
      which is an operator of the form (\ref{eq:ell_op}) in $\mathbb{R}^{n}$
      with elliptic symbol
      $$
      Q_n (\xi',\xi_n)=Q_{n-1} (\xi') + \xi_n^{2m} \ge \lambda_{n-1} |\xi'|^{2m}+ \xi_n^{2m}
      \ge \frac{\min\{\lambda_{n-1},1\}}{2m} |(\xi',\xi_n)|^{2m}.
      $$
      The operator $L_n$ is connected to $L_{n-1}$ by the procedure described in Subsection~\ref{subsec:app:3}.
      In particular John's   fundamental solution $K_{n-1}$ corresponding to $L_{n-1}$
      is given by (\ref{eq:app:3.1}).  The assumption yields
      the existence of a vector $x'\in \mathbb{R}^{n-1}\setminus \{0\}$ such that
      $$
      0> K_{n-1}(x') = \int_\mathbb{R} K_{n} (x',\xi_n)\, d\xi_n.
      $$
      This shows that there exists a point $x_n\in  \mathbb{R}$ which satisfies
      $$
      K_{n} (x',x_n)<0,
      $$
      which completes the proof.
     \end{proof}

	
	\section{Sign changing fundamental solutions for $\boldsymbol{m=2}$}
	\label{sec:Guido}
	
On $\mathcal{S}\left( \mathbb{R}^{n}\right) $, the space of rapidly
decreasing functions, one may define the Fourier-transformation $\mathcal{F}:%
\mathcal{S}\left( \mathbb{R}^{n}\right) \rightarrow \mathcal{S}\left(
\mathbb{R}^{n}\right) $ by
\begin{equation}
\left( \mathcal{F}u\right) (\xi )=\frac{1}{\left( 2\pi \right) ^{n/2}}\int_{%
\mathbb{R}^{n}}e^{-ix\cdot \xi }u\left( x\right) dx.  \label{Fou}
\end{equation}%
The inverse on $\mathcal{S}\left( \mathbb{R}^{n}\right) $ is $\left(
\mathcal{F}^{-1}v\right) (x)=\left( \mathcal{F}v\right) (-x)$. The
definition of $\mathcal{F}$\ in (\ref{Fou}) can be directly extended to $%
u\in L^{1}\left( \mathbb{R}^{n}\right) $. For $u\in W^{1,1}\left( \mathbb{R}%
^{n}\right) $ one finds $\left( \mathcal{F}\left( \frac{\partial }{\partial
x_{j}}u\right) \right) (\xi )=i\xi _{j}\left( \mathcal{F}u\right) (\xi )$
and for $u\in W^{2m,1}\left( \mathbb{R}^{n}\right) $ the differential
equation $Lu=f$, with $L$ as in (\ref{eq:ell_op}) having symbol $Q$, turns into $Q(\xi )\left(
\mathcal{F}u\right) (\xi )=\left( \mathcal{F}f\right) (\xi )$. If $\mathcal{F%
}^{-1}$ is defined, then one would obtain a solution of $Lu=f$ by
\begin{equation*}
u=\mathcal{F}^{-1}Q\left( \xi \right) ^{-1}\mathcal{F}f.
\end{equation*}%
So formally one would obtain the following expression for the corresponding
fundamental solution:
\begin{equation}
F(x)=\left( \mathcal{F}^{-1}Q\left( \xi \right) ^{-1}\mathcal{F}\delta
_{0}\right) (x)=\frac{1}{\left( 2\pi \right) ^{n}}\int_{\mathbb{R}%
^{n}}e^{ix\cdot \xi }\frac{1}{Q\left( \xi \right) }d\xi ,  \label{FFF}
\end{equation}%
with $\delta _{0}$ the delta-distribution in $0$ and $\mathcal{F}\delta
_{0}=(2\pi)^{-n/2}$, cf. (\ref{Foud}) below. Due to the homogeneity of $Q$ the integral in (\ref{FFF}) is
however not defined in $L^{1}\left( \mathbb{R}^{n}\right) $ but at most as
an oscillatory integral.\medskip

The Malgrange-Ehrenpreis Theorem, see \cite[Theorem IX.23]{ReedSimonII},
states that a distributional solution $F$ exists for $LF=\delta _{0}$,
whenever $L$ is a differential operator with constant coefficients. For
elliptic operators the zero sets of $Q$ are small in $\mathbb{R}^{n}$, which
may allow one to give a classical meaning to (\ref{FFF}) and gives a route
to the fundamental solution. In some special cases this formula even allows
one to derive an (almost) explicit fundamental solution. One such case is
the following class of fourth order elliptic operators:%
\begin{equation}
L=\left( \Delta ^{\prime }\right) ^{2}+\alpha \Delta ^{\prime }\left( \tfrac{%
\partial }{\partial x_{n}}\right) ^{2}+\left( \tfrac{\partial }{\partial
x_{n}}\right) ^{4},  \label{El}
\end{equation}%
where $x^{\prime }=\left( x_{1},\dots ,x_{n-1}\right) $ and $\Delta ^{\prime
}=\sum_{i=1}^{n-1}\left( \tfrac{\partial }{\partial x_{i}}\right) ^{2}$%
.\medskip

Although $L$ is only interesting in the present setting whenever $n\geq 5$,
allow us to classify $L$ for all dimensions.

\begin{lemma}
For $L$ in (\ref{El}) one finds:

\begin{enumerate}
\item $L$ is elliptic, if and only if $\alpha >-2$.

\item If $\alpha \geq 2$, the operator $L$ can be written as a product of
two real second order elliptic operators.

\item If $\alpha \in \left( -2,2\right) $, the operator $L$ can be written
as a product of two real second order elliptic operators only for $n=2$.
\end{enumerate}
\end{lemma}

Notice that the level hypersurfaces of the symbol for $L$ are convex, if and
only if $\alpha \geq 0$. For $\alpha =2$ one recovers $L=\Delta ^{2}$%
.\medskip

\begin{proof}
To prove that ellipticity holds if and only if $\alpha >-2$, is elementary.
For $\left\vert \alpha \right\vert \geq 2$ one may split the symbol $Q$ for $%
L$ in (\ref{El}) into real quadratic polynomials by:%
\begin{equation*}
Q\left( \xi ^{\prime },\xi _{n}\right) =\left( \left\vert \xi ^{\prime
}\right\vert ^{2}+\tfrac{\alpha +\sqrt{\alpha ^{2}-4}}{2}\xi _{n}^{2}\right)
\left( \left\vert \xi ^{\prime }\right\vert ^{2}+\tfrac{\alpha -\sqrt{\alpha
^{2}-4}}{2}\xi _{n}^{2}\right) .
\end{equation*}%
Whenever $n=2$ and $\alpha \in \left( -2,2\right] $ the operator $L$ can be
split into a product of two real second order elliptic operators following:%
\begin{equation*}
Q\left( \xi _{1},\xi _{2}\right) =\left( \xi _{1}^{2}-\sqrt{2-\alpha }\,\xi
_{1}\xi _{2}+\xi _{2}^{2}\right) \left( \xi _{1}^{2}+\sqrt{2-\alpha }\,\xi
_{1}\xi _{n}+\xi _{2}^{2}\right) .
\end{equation*}%
This last splitting in dimensions $n\geq 3$ with $\alpha \in \left(
-2,2\right) $, that is, replacing $\xi _{1}$ by $\left\vert \xi ^{\prime
}\right\vert $, would lead to a Fourier multiplier operator of order $2$
with \emph{nonsmooth} symbol, i.e. not even to a pseudodifferential operator.
\end{proof}

The interesting case is hence $\alpha \in \left( -2,2\right) $ and then it
is convenient to use $\alpha =2\cos \gamma $ with $\gamma \in (0,\pi )$. So
we write
\begin{equation}
L_{\gamma }=\left( \Delta ^{\prime }\right) ^{2}+2\cos \gamma \ \Delta
^{\prime }\left( \tfrac{\partial }{\partial x_{n}}\right) ^{2}+\left( \tfrac{%
\partial }{\partial x_{n}}\right) ^{4}  \label{Elga}
\end{equation}%
with corresponding symbol $Q\left( \xi ^{\prime },\xi _{n}\right)
=\left\vert \xi ^{\prime }\right\vert ^{4}+2\cos \gamma \ \left\vert \xi
^{\prime }\right\vert ^{2}\xi _{n}^{2}+\xi _{n}^{4}$.
The fundamental solution for (\ref{Elga}) is a regular distribution, so a function, which is $C^\infty$ on $\R^n\setminus\{0\}$ and moreover, homogeneous of degree $4-n$. We will recall that fact as the first step,
when we prove the following result.

\begin{proposition}
\label{calclem}Let $n\geq 5$ and $\gamma \in \left( 0,\pi \right) $. The fundamental solution $F_{n,\gamma }$ for (\ref{Elga}) satisfies:

\begin{itemize}
\item when $\left\vert x^{\prime }\right\vert \not=0$ and $x_{n}\neq 0$ one has
\begin{equation*}
\hspace{-1cm}F_{n,\gamma }\left( x\right) =\frac{1}{2^{n-2}\Gamma \left( \frac{n-2}{2}%
\right) \pi ^{n/2}\left\vert x^{\prime }\right\vert ^{n-4}}\int_{0}^{\infty
}C_{n}\left( s\right) \tfrac{\exp \left( -\frac{\left\vert x_{n}\right\vert
}{\left\vert x^{\prime }\right\vert }s\cos \left( \frac{1}{2}\gamma \right)
\right) \sin \left( \frac{\left\vert x_{n}\right\vert }{\left\vert x^{\prime
}\right\vert }s\sin \left( \frac{1}{2}\gamma \right) +\frac{1}{2}\gamma
\right) }{2\sin \left( \gamma \right) }s^{n-5}ds
\end{equation*}%
with $C_{n}\left( s\right) :=\int_{0}^{\pi }\cos \left( s\cos \varphi
\right) \ \left( \sin \varphi \right) ^{n-3}d\varphi $;

\item when $\left\vert x^{\prime }\right\vert =0$ and $x_{n}\neq 0$ one
obtains%
\begin{equation}
F_{n,\gamma }\left( x\right) =\frac{\Gamma \left( n-4\right) }{2^{n-1}\Gamma
\left( \frac{n-1}{2}\right) \pi ^{(n-1)/2}\left\vert x_{n}\right\vert
^{n-4}\sin \left( \gamma \right) }\sin \left( \frac{n-3}{2}\gamma \right) .
\label{critfo}
\end{equation}
\end{itemize}
\end{proposition}

\begin{proof}
The proof is divided into 4 steps.\medskip

\noindent\emph{i) Fundamental solution as distribution through an inverse
Fourier transform.} Let us first discuss the extensions of the Fourier
transform $\mathcal{F}$ in (\ref{Fou}). Both $\mathcal{F}$ and its inverse
are well defined on $\mathcal{S}\left( \mathbb{R}^{n}\right) $. 
By definition, a sequence $\left(\varphi_\ell\right)_{\ell\in\mathbb{N}}\subset \mathcal{S}\left( \mathbb{R}^{n}\right) $ converges to $\varphi\in\mathcal{S}\left( \mathbb{R}^{n}\right) $ iff for all $k\in \mathbb{N}$ 
and for all multi-indices $\alpha\in\mathbb{N}_0^n$ one has 
$\sup_{x\in\mathbb{R}^n} \left( (1+|x|)^k \left|D^\alpha (\varphi_\ell-\varphi)(x)\right|\right)\to 0$ as $\ell\to\infty$.
The natural
extension to the space of tempered distributions $\mathcal{S}\left( \mathbb{R}^{n}\right) ^{\prime }$ is then \cite[Definition 7.1.9]{Hormander-I} as
follows:%
\begin{equation}
\left\langle \mathcal{F}\Psi ,\varphi \right\rangle :=\left\langle \Psi ,%
\mathcal{F}\varphi \right\rangle \text{ for }\Psi \in \mathcal{S}\left(
\mathbb{R}^{n}\right) ^{\prime }\text{ and }\varphi \in \mathcal{S}\left(
\mathbb{R}^{n}\right) ,  \label{Foud}
\end{equation}%
with a similar version for the inverse; $\left\langle \cdot ,\cdot
\right\rangle $ denotes the duality between distribution and test function.

For $u\in L^{1}\left( \mathbb{R}^{n}\right) $ the $\mathcal{F}$ in (\ref{Fou}%
) is well defined and one finds $\mathcal{F}u\in L^{\infty }\left( \mathbb{R}%
^{n}\right) $ and even the estimate 
$\left\Vert \mathcal{F}u\right\Vert _{\infty }\leq
\frac{1}{\left( 2\pi \right) ^{n/2}}\left\Vert u\right\Vert _{1}$, but
generically $\mathcal{F}u\not\in L^{1}\left( \mathbb{R}^{n}\right) $. In
general $\mathcal{F}^{-1}$ is not directly well defined on $L^{\infty
}\left( \mathbb{R}^{n}\right) $. The Fourier-transformation can also be
extended to $L^{2}\left( \mathbb{R}^{n}\right) $ by Plancherel and hence
\cite[Theorem 7.1.13]{Hormander-I} for $L^{p}\left( \mathbb{R}^{n}\right) $
with $p\in \left[ 1,2\right] $. For those $p\in \left( 1,2\right] $ one
finds $\mathcal{F}L^{p}\left( \mathbb{R}^{n}\right) \subset L^{q}\left(
\mathbb{R}^{n}\right) $ with $q=\frac{p}{p-1}\geq 2$. So the formula in (\ref%
{FFF}) needs clarification.

With the definition of the (inverse) Fourier transform in (\ref{Foud}) one
finds by \cite[Theorem 7.1.20]{Hormander-I} for $n>4$, that%
\begin{equation}
F_{n,\gamma }:=\mathcal{F}^{-1}\left( \left( \frac{1}{Q}\right) \mathcal{F}\delta_{0}\right) 
=(2\pi)^{-n/2}\mathcal{F}^{-1}\left( \left( \frac{1}{Q}\right) ^{\bullet
}\right)  \label{FFFpunt}
\end{equation}%
is defined in $\mathcal{S}\left( \mathbb{R}^{n}\right) ^{\prime }$ and, since
$L\varphi =L^{\ast }\varphi $, is such that
\begin{equation*}
\left\langle LF_{n,\gamma },\varphi \right\rangle :=\left\langle F_{n,\gamma
},L\varphi \right\rangle =\left\langle \delta _{0},\varphi \right\rangle
\text{ for all }\varphi \in \mathcal{S}\left( \mathbb{R}^{n}\right) .
\end{equation*}%
The dot in (\ref{FFFpunt}) is defined in \cite[Theorem 3.2.3]{Hormander-I}
as the unique homogeneous extension to $\mathcal{D}\left( \mathbb{R}%
^{n}\right) ^{\prime }$ of the same degree of homogeneity, namely $-4$, of $%
Q^{-1}\in \mathcal{D}\left( \mathbb{R}^{n}\setminus \left\{ 0\right\}
\right) ^{\prime }$, whenever this degree is not an integer below or equal $%
-n$. Here $\mathcal{D}\left( \mathbb{R}^{n}\right) ^{\prime }$ is the space
of Schwartz distributions. The distribution $u\in \mathcal{D}\left( X\right)
^{\prime }$ is homogeneous of degree $a$, when
\begin{equation*}
\left\langle u,\varphi \right\rangle =t^{a}\left\langle u,t^{n}\varphi
\left( t\cdot \right) \right\rangle \text{ for all }t>0\text{ and }\varphi
\in \mathcal{D}\left( X\right) .
\end{equation*}%
The distribution $Q^{-1}$ on $\mathcal{D}\left( \mathbb{R}^{n}\setminus
\left\{ 0\right\} \right) $, when extended to $\mathcal{D}\left( \mathbb{R}%
^{n}\right) $, can only add a combination of the $\delta _{0}$-distribution
and its distributional derivatives. Since in $\mathbb{R}^{n}$ each such a
distribution is homogeneous of degree $-n$ or less, one finds that the
extension is the regular distribution, that is, the function $\xi
\mapsto Q^{-1}\left( \xi \right) $ on $\mathbb{R}^{n}$ and we may skip
the dot. By \cite[Theorem 7.1.16]{Hormander-I} $F_{n,\gamma }$ with $n\geq 5$
is then homogeneous of degree $4-n$ and by \cite[Theorem 7.1.18]{Hormander-I}
one finds that $F_{n,\gamma }\in \mathcal{S}\left( \mathbb{R}^{n}\right)
^{\prime }$ and that $\left( F_{n,\gamma }\right) |_{\mathbb{R}^{n}\setminus
\left\{ 0\right\} }\in C^{\infty }\left( \mathbb{R}^{n}\setminus \left\{
0\right\} \right) $ is a function. Also here the extension in $0$ of this
function can only add a combination of the $\delta _{0}$-distribution and
its distributional derivatives and again, in $\mathbb{R}^{n}$ each such
distribution is homogeneous of degree $-n$ or less. So indeed, one finds
that also $F_{n,\gamma }$ is given by a function satisfying:%
\begin{equation*}
F_{n,\gamma }\left( x\right) =\left\vert x\right\vert ^{4-n}F_{n,\gamma
}\left( \frac{x}{\left\vert x\right\vert }\right) .
\end{equation*}%
With $M=\sup_{\left\vert \omega \right\vert =1}\left\vert F_{n,\gamma
}\left( \omega \right) \right\vert $ one finds from this
\begin{equation}
\left\vert F_{n,\gamma }\left( x\right) \right\vert \leq M\left\vert
x\right\vert ^{4-n}\text{ for }x\in \mathbb{R}^{n}\setminus \left\{
0\right\} . \medskip  \label{ub}
\end{equation}

\noindent \emph{ii) Approximation as distribution through a summability
kernel.} Since we have established that $F_{n,\gamma }$ is a function, we
will try next to derive a more explicit formula. Since $Q^{-1}$ is not an $%
L^{1}$-function the direct definition of the inverse Fourier-transform just
after (\ref{Fou}) is not applicable. We will use an approximation through a
special positive summability kernel $k_{\varepsilon }$, see \cite[Section
VI.1.9]{Katznelson}. A positive summability kernel on $\mathbb{R}^{n}$ is
defined as a family $\left( k_{\varepsilon }\right) _{\varepsilon \in
\left( 0,\varepsilon _{0}\right] }\subset C\left( \mathbb{R}^{n}\right) $
with $k_{\varepsilon }\geq 0$ satisfying:

\begin{enumerate}
\item for all $\varepsilon \in \left( 0,\varepsilon _{0}\right] $: $\int_{%
\mathbb{R}^{n}}k_{\varepsilon }\left( x\right) dx=1$;

\item for all $\delta >0$: $\lim_{\varepsilon \downarrow 0}\int_{\left\vert
x\right\vert >\delta }k_{\varepsilon }\left( x\right) dx=0$.
\end{enumerate}

The summability kernel that we use is a combination of a Gauss kernel in $%
x^{\prime }\in \mathbb{R}^{n-1}$ and $\frac{1}{2\varepsilon }e^{-\left\vert
x_{n}\right\vert /\varepsilon }$. We set%
\begin{equation}
k_{1}\left( x\right) :=\frac{1}{\sqrt{2\pi }^{n-1}}\exp \left( -\tfrac{1}{2}%
\left\vert x^{\prime }\right\vert ^{2}\right) \frac{1}{2}\exp \left(
-\left\vert x_{n}\right\vert \right)   \label{sumker}
\end{equation}%
and define%
\begin{equation}
k_{\varepsilon }\left( x\right) :=\varepsilon ^{-n}k_{1}\left( x/\varepsilon
\right) .  \label{sumkerscale}
\end{equation}%
Recall that%
\begin{equation*}
\int_{\mathbb{R}}e^{-ist-\frac{1}{2}t^{2}}dt=\sqrt{2\pi }e^{-\frac{1}{2}%
s^{2}}\;\,\text{ and }\;\,\int_{\mathbb{R}}e^{-ist}\tfrac{1}{2}e^{-\left\vert
t\right\vert }dt=\frac{1}{1+s^{2}}.
\end{equation*}%
So one finds $\left( \mathcal{F}k_{\varepsilon }\right) \left( \xi \right)
=\left( \mathcal{F}k_{1}\right) \left( \varepsilon \xi \right) $ and%
\begin{equation*}
\left( \mathcal{F}k_{1}\right) \left( \xi \right) =\frac{1}{\sqrt{2\pi }%
^{n-1}}\exp \left( -\tfrac{1}{2}\left\vert \xi ^{\prime }\right\vert
^{2}\right) \frac{1}{\sqrt{2\pi }\left( 1+\xi _{n}^{2}\right) }=\frac{1}{%
\sqrt{2\pi }^{n}}\frac{\exp \left( -\tfrac{1}{2}\left\vert \xi ^{\prime
}\right\vert ^{2}\right) }{1+\xi _{n}^{2}}.
\end{equation*}
One obtains for all $\varphi \in \mathcal{S}\left( \mathbb{R}^{n}\right) $,
exploiting $k_{\varepsilon }\ast \varphi \to \varphi $ in $%
\mathcal{S}\left( \mathbb{R}^{n}\right) $, that
\begin{equation}
\left\langle F_{n,\gamma },\varphi \right\rangle 
=\lim_{\varepsilon
\downarrow 0}\left\langle F_{n,\gamma },k_{\varepsilon }\ast \varphi
\right\rangle 
=(2\pi)^{-n/2}\lim_{\varepsilon \downarrow 0}\left\langle \frac{1}{Q},%
\mathcal{F}^{-1}\left( k_{\varepsilon }\ast \varphi \right) \right\rangle
=\lim_{\varepsilon \downarrow 0}\left\langle \frac{1}{Q},\left(\mathcal{F}%
^{-1}k_{\varepsilon }\right)\ \left( \mathcal{F}^{-1}\varphi\right) \right\rangle  \label{pref}
\end{equation}%
and from the properties of distributions:
\begin{equation}
\left\langle \frac{1}{Q},\left(\mathcal{F}^{-1}k_{\varepsilon }\right)\ 
\left(\mathcal{F}^{-1}\varphi \right)\right\rangle
=\left\langle \overline{\mathcal{F}^{-1}k_{\varepsilon }}\ \frac{1}{Q},\mathcal{F}^{-1}\varphi \right\rangle
=\left\langle \mathcal{F}k_{\varepsilon }~\frac{1}{Q},\mathcal{F}%
^{-1}\varphi \right\rangle =\left\langle \mathcal{F}^{-1}\left( \mathcal{F}%
k_{\varepsilon }~\frac{1}{Q}\right) ,\varphi \right\rangle .  \label{FminFq}
\end{equation}%
In other words $\mathcal{F}^{-1}\left( \mathcal{F}k_{\varepsilon }\frac{1}{Q}%
\right) \rightarrow F_{n,\gamma }$ for $\varepsilon \downarrow 0$ in the
sense of distributions.\medskip

\noindent \emph{iii) Approximation as a function through the summability
kernel.} Since $F_{n,\gamma }$ is a regular distribution and
\begin{equation*}
\left( x,y\right) \mapsto F_{n,\gamma }(x)~k_{\varepsilon }\left( x-y\right)
~\varphi (y)\in L^{1}\left( \mathbb{R}^{n}\times \mathbb{R}^{n}\right) ,
\end{equation*}%
we may also write%
\begin{gather}
\left\langle F_{n,\gamma },k_{\varepsilon }\ast \varphi \right\rangle =\int_{%
\mathbb{R}^{n}}F_{n,\gamma }\left( x\right) \int_{\mathbb{R}%
^{n}}k_{\varepsilon }\left( x-y\right) \ \varphi (y)~dydx\hspace{1cm}  \notag
\\
=\int_{\mathbb{R}^{n}}\int_{\mathbb{R}^{n}}F_{n,\gamma }(x)~k_{\varepsilon
}\left( x-y\right) \ \varphi (y)~dxdy=\left\langle F_{n,\gamma }\ast
k_{\varepsilon },\varphi \right\rangle .  \label{switch}
\end{gather}%
Setting%
\begin{equation}
f_{\varepsilon }(x):=\left( F_{n,\gamma }\ast k_{\varepsilon }\right) \left(
x\right)   \label{feppre}
\end{equation}%
we find from (\ref{ub}) and (\ref{sumkerscale}) that%
\begin{align}
\left\vert f_{\varepsilon }(x)\right\vert & \leq \int_{y\in \mathbb{R}%
^{n}}\left\vert F_{n,\gamma }\left( y\right) k_{\varepsilon }\left(
x-y\right) \right\vert dy\leq M\int_{y\in \mathbb{R}^{n}}\left\vert
y\right\vert ^{4-n}k_{\varepsilon }\left( x-y\right) dy  \notag \\
& =M\left( k_{\varepsilon }\ast \left\vert \cdot \right\vert ^{4-n}\right)
\left( x\right) =\varepsilon ^{4-n}M\left( k_{1}\ast \left\vert \cdot
\right\vert ^{4-n}\right) \left( \varepsilon ^{-1}x\right) .  \label{crt}
\end{align}%
Since $k_{1}\left( x\right) \leq ce^{-\left\vert x\right\vert /2}$ for some $%
c>0$, we may estimate the last expression in (\ref{crt}) by splitting the
corresponding integral in two parts: $\left\vert x-y\right\vert <\frac{1}{2}%
\left\vert x\right\vert $, implying $\left\vert y\right\vert \geq \frac{1}{2}%
\left\vert x\right\vert $, and $\left\vert x-y\right\vert >\frac{1}{2}%
\left\vert x\right\vert $. Indeed, one finds for some $C_{1}>0$ that%
\begin{align*}
\left( k_{1}\ast \left\vert \cdot \right\vert ^{4-n}\right) \left( x\right)
& \leq c\left( \int_{\left\vert x-y\right\vert <\frac{1}{2}\left\vert
x\right\vert }e^{-\left\vert y\right\vert /2}\left\vert x-y\right\vert
^{4-n}dy+\int_{\left\vert x-y\right\vert >\frac{1}{2}\left\vert x\right\vert
}e^{-\left\vert y\right\vert /2}\left\vert x-y\right\vert ^{4-n}dy\right)  \\
& \leq c\sigma _{n}\left( e^{-\frac{1}{4}\left\vert x\right\vert
}\int_{r=0}^{\frac{1}{2}\left\vert x\right\vert }r^{4-n+n-1}dr+\left( \tfrac{%
1}{2}\left\vert x\right\vert \right) ^{4-n}\int_{r=0}^{\infty
}e^{-r/2}r^{n-1}dr\right)  \\
& =c\sigma _{n}\left( \tfrac{1}{64}e^{-\frac{1}{4}\left\vert x\right\vert
}\left\vert x\right\vert ^{4}+2^{2n-4}\,\left\vert x\right\vert ^{4-n}\Gamma
\left( n\right) \right) \leq C_{1}\left\vert x\right\vert ^{4-n},
\end{align*}%
with $\sigma _{n}=\int_{\mathbb{S}^{n-1}}d\omega =\frac{2\,\pi ^{n/2}}{%
\Gamma \left( n/2\right) }$, the surface area of the unit sphere in $\mathbb{%
R}^{n}$. The result is that
\begin{equation*}
\left\vert f_{\varepsilon }(x)\right\vert \leq \varepsilon ^{4-n}M\left(
k_{1}\ast \left\vert \cdot \right\vert ^{4-n}\right) \left( \varepsilon
^{-1}x\right) \leq C_{1}\varepsilon ^{4-n}\left\vert \varepsilon
^{-1}x\right\vert ^{4-n}=C_{1}\left\vert x\right\vert ^{4-n}.
\end{equation*}%
This allows us to use Lebesgue's dominated convergence theorem to find with (%
\ref{pref}) and (\ref{switch}) that
\begin{equation}
\left\langle F_{n,\gamma },\varphi \right\rangle =\lim_{\varepsilon
\downarrow 0}\int_{\mathbb{R}^{n}}f_{\varepsilon }(x)\varphi \left( x\right)
dx=\int_{\mathbb{R}^{n}}\lim_{\varepsilon \downarrow 0}f_{\varepsilon
}(x)\varphi \left( x\right) dx\;\;\text{ for all }\varphi \in \mathcal{S}\left(
\mathbb{R}^{n}\right)   \label{hcmain}
\end{equation}%
with the last identity for any measurable $\varphi :\mathbb{R}%
^{n}\rightarrow \mathbb{R}$ such that $\int_{\mathbb{R}^{n}}\left\vert
\varphi \left( x\right) \right\vert \left\vert x\right\vert ^{4-n}dx<\infty $%
.\medskip

\noindent \emph{iv) An almost explicit formula by a contour integral.} Next
we will compute $f_{\varepsilon }$ using the formula \linebreak $\mathcal{F}%
^{-1}\left( \mathcal{F}k_{\varepsilon }~\frac{1}{Q}\right) $ from (\ref%
{FminFq}). The symbol $Q=Q_{\gamma }$ for (\ref{Elga}) satisfies%
\begin{equation*}
Q_{\gamma }\left( \xi ^{\prime },\xi _{n}\right) =\left( \xi
_{n}+ie^{i\gamma /2}\left\vert \xi ^{\prime }\right\vert \right) \left( \xi
_{n}-ie^{i\gamma /2}\left\vert \xi ^{\prime }\right\vert \right) \left( \xi
_{n}+ie^{-i\gamma /2}\left\vert \xi ^{\prime }\right\vert \right) \left( \xi
_{n}-ie^{-i\gamma /2}\left\vert \xi ^{\prime }\right\vert \right)
\end{equation*}%
and the approximation $f_{\varepsilon }$ of the fundamental solution $%
F_{n,\gamma }\left( x\right) $ for (\ref{Elga}) becomes%
\begin{equation}
f_{\varepsilon }(x)=\left( \mathcal{F}^{-1}\left( \mathcal{F}k_{\varepsilon
}\ \frac{1}{Q_{\gamma }}\right) \right) (x)=\frac{1}{\left( 2\pi \right) ^{n}%
}\int_{\mathbb{R}^{n}}e^{ix\cdot \xi }\frac{1}{Q_{\gamma }\left( \xi \right)
}\frac{\exp \left( -\tfrac{1}{2}\varepsilon ^{2}\left\vert \xi ^{\prime
}\right\vert ^{2}\right) }{1+\varepsilon ^{2}\xi _{n}^{2}}d\xi .
\label{feppe}
\end{equation}%
Note that the integral converges near $0$ for $n\geq 5$. Near $\infty $ the
integral converges for all $n$.

The integrand in (\ref{feppe}) contains an analytic function of $\xi _{n}\in
\mathbb{C}$ and we find for $\gamma \in \left( 0,\pi \right) $ and $x_{n}>0$
by a contour integral in $\mathbb{R}+i\left[ 0,\infty \right) \subset
\mathbb{C}$ that%
\begin{gather}
f_{\varepsilon }\left( x\right) =\frac{1}{\left( 2\pi \right)
^{n}\varepsilon ^{2}}\int_{\xi ^{\prime }\in \mathbb{R}^{n-1}}e^{ix^{\prime
}\cdot \xi ^{\prime }-\frac{1}{2}\varepsilon ^{2}\left\vert \xi ^{\prime
}\right\vert ^{2}}\int_{\xi _{n}\in \mathbb{R}}\dfrac{e^{ix_{n}\xi _{n}}}{%
Q\left( \xi ^{\prime },\xi _{n}\right) \left( \xi _{n}-\frac{i}{\varepsilon }%
\right) \left( \xi _{n}+\frac{i}{\varepsilon }\right) }d\xi _{n}d\xi
^{\prime }\medskip   \notag \\
=\frac{2\pi i}{\left( 2\pi \right) ^{n}\varepsilon ^{2}}\int_{\xi ^{\prime
}\in \mathbb{R}^{n-1}}e^{ix^{\prime }\cdot \xi ^{\prime }-\frac{1}{2}%
\varepsilon ^{2}\left\vert \xi ^{\prime }\right\vert ^{2}}\hspace{-0.25cm}%
\sum_{z\in \left\{ ie^{i\gamma /2}\left\vert \xi ^{\prime }\right\vert,\,
ie^{-i\gamma /2}\left\vert \xi ^{\prime }\right\vert,\, \frac{i}{\varepsilon }%
\right\} }\hspace{-0.25cm}\mathrm{Res}\left( \frac{\exp \left( ix_{n}\xi
_{n}\right) }{Q_{\gamma }\left( \xi ^{\prime },\xi _{n}\right) \left( \xi
_{n}-\frac{i}{\varepsilon }\right) \left( \xi _{n}+\frac{i}{\varepsilon }%
\right) }\right) _{\xi _{n}=z}\hspace{-3mm}d\xi ^{\prime }\medskip   \notag \\
=\frac{2\pi i}{\left( 2\pi \right) ^{n}}\int_{\xi ^{\prime }\in \mathbb{R}%
^{n-1}}e^{ix^{\prime }\cdot \xi ^{\prime }-\frac{1}{2}\varepsilon
^{2}\left\vert \xi ^{\prime }\right\vert ^{2}}\left( \frac{\exp \left(
-x_{n}e^{i\gamma /2}\left\vert \xi ^{\prime }\right\vert \right) }{%
4e^{i\gamma /2}\sin \left( \gamma \right) \left\vert \xi ^{\prime
}\right\vert ^{3}\left( 1-\varepsilon ^{2}e^{i\gamma }\left\vert \xi
^{\prime }\right\vert ^{2}\right) }\right. \hspace{25mm}  \notag \\
\hspace{5mm}+\left. \frac{-\exp \left( -x_{n}e^{-i\gamma /2}\left\vert \xi
^{\prime }\right\vert \right) }{4e^{-i\gamma /2}\sin \left( \gamma \right)
\left\vert \xi ^{\prime }\right\vert ^{3}\left( 1-\varepsilon
^{2}e^{-i\gamma }\left\vert \xi ^{\prime }\right\vert ^{2}\right) }-\dfrac{%
i\varepsilon ^{3}\exp \left( -x_{n}/\varepsilon \right) }{2\left(
1-\varepsilon ^{2}e^{i\gamma }\left\vert \xi ^{\prime }\right\vert
^{2}\right) \left( 1-\varepsilon ^{2}e^{-i\gamma }\left\vert \xi ^{\prime
}\right\vert ^{2}\right) }\right) d\xi ^{\prime }.  \label{Feps}
\end{gather}%
Since $\gamma \in \left( 0,\pi \right) $ holds and by the negative exponents
in the exponential, the integral in (\ref{Feps}) converges. The last term of
(\ref{Feps}) contains $\exp \left( -x_{n}/\varepsilon \right) $, which shows
that
\begin{equation*}
\left\vert \dfrac{\varepsilon ^{3}\exp \left( -x_{n}/\varepsilon \right) }{%
2iQ_{\gamma }\left( \varepsilon \xi ^{\prime },i\right) }\right\vert \leq
\frac{k!~\varepsilon ^{k+3}}{\left( \varepsilon ^{4}\left\vert \xi ^{\prime
}\right\vert ^{4}-2\cos \gamma \ \varepsilon ^{2}\left\vert \xi ^{\prime
}\right\vert ^{2}+1\right) x_{n}^{k}}.
\end{equation*}%
Hence by taking $k>n-4$ it follows, whenever $x_{n}>0$, for $\varepsilon
\downarrow 0$ that
\begin{equation*}
\left\vert \int_{\xi ^{\prime }\in \mathbb{R}^{n-1}}e^{ix^{\prime }\cdot \xi
^{\prime }-\frac{1}{2}\varepsilon ^{2}\left\vert \xi ^{\prime }\right\vert
^{2}}\dfrac{\varepsilon ^{3}\exp \left( -x_{n}/\varepsilon \right) }{%
2iQ_{\gamma }\left( \varepsilon \xi ^{\prime },i\right) }d\xi ^{\prime
}\right\vert \leq \int_{y^{\prime }\in \mathbb{R}^{n-1}}\frac{k!\varepsilon
^{k+4-n}e^{-\frac{1}{2}\left\vert y^{\prime }\right\vert ^{2}}}{\left(
\left\vert y^{\prime }\right\vert ^{4}-2\cos \gamma \ \left\vert y^{\prime
}\right\vert ^{2}+1\right) x_{n}^{k}}dy^{\prime }\rightarrow 0.
\end{equation*}%
It remains to consider
\begin{align}
& F_{n,\gamma }\left( x\right) =\displaystyle\lim_{\varepsilon \downarrow
0}f_{\varepsilon }\left( x\right)   \notag \\
& =\lim_{\varepsilon \downarrow 0}\frac{2\pi i}{\left( 2\pi \right) ^{n}}%
\int_{\xi ^{\prime }\in \mathbb{R}^{n-1}}\frac{e^{ix^{\prime }\cdot \xi
^{\prime }-\frac{1}{2}\varepsilon ^{2}\left\vert \xi ^{\prime }\right\vert
^{2}}}{4\sin \left( \gamma \right) \left\vert \xi ^{\prime }\right\vert ^{3}}%
\left( \tfrac{\exp \left( -x_{n}e^{i\gamma /2}\left\vert \xi ^{\prime
}\right\vert \right) }{e^{i\gamma /2}\left( 1-\varepsilon ^{2}e^{i\gamma
}\left\vert \xi ^{\prime }\right\vert ^{2}\right) }-\tfrac{\exp \left(
-x_{n}e^{-i\gamma /2}\left\vert \xi ^{\prime }\right\vert \right) }{%
e^{-i\gamma /2}\left( 1-\varepsilon ^{2}e^{-i\gamma }\left\vert \xi ^{\prime
}\right\vert ^{2}\right) }\right) d\xi ^{\prime }  \notag \\
& =\frac{2\pi i}{\left( 2\pi \right) ^{n}}\int_{\xi ^{\prime }\in \mathbb{R}%
^{n-1}}\frac{e^{ix^{\prime }\cdot \xi ^{\prime }}}{4\sin \left( \gamma
\right) \left\vert \xi ^{\prime }\right\vert ^{3}}\left( e^{-x_{n}e^{i\gamma
/2}\left\vert \xi ^{\prime }\right\vert -i\gamma /2}-e^{-x_{n}e^{-i\gamma
/2}\left\vert \xi ^{\prime }\right\vert +i\gamma /2}\right) d\xi ^{\prime }
\notag \\
& =\frac{1}{\left( 2\pi \right) ^{n-1}}\int_{\xi ^{\prime }\in \mathbb{R}%
^{n-1}}e^{ix^{\prime }\cdot \xi ^{\prime }}\frac{\exp \left(
-x_{n}\left\vert \xi ^{\prime }\right\vert \cos (\tfrac{1}{2}\gamma )\right)
\sin \left( x_{n}\left\vert \xi ^{\prime }\right\vert \sin (\frac{1}{2}%
\gamma )+\frac{1}{2}\gamma \right) }{2\left\vert \xi ^{\prime }\right\vert
^{3}\sin (\gamma )}d\xi ^{\prime },  \label{Feng}
\end{align}%
which is well-defined for $x_{n}>0$. For $x_{n}<0$ one replaces $x_{n}$ by $%
\left\vert x_{n}\right\vert $ in (\ref{Feng}).

For the remaining integral in (\ref{Feng}) we proceed by using for $\xi
^{\prime }\in \mathbb{R}^{n-1}$ with $n\geq 4$ the coordinates%
\begin{equation*}
\xi ^{\prime }=\left(
\begin{array}{c}
r\cos \varphi \\
r\omega ^{\prime \prime }\sin \varphi%
\end{array}%
\right) \text{ with }r\geq 0\text{, }\omega ^{\prime \prime }\in \mathbb{S}%
^{n-3}\text{ and }\varphi \in \left[ 0,\pi \right] ,
\end{equation*}%
where $\mathbb{S}^{n-3}=\left\{ v\in \mathbb{R}^{n-2};\text{ }\left\vert
v\right\vert =1\right\} $. By the rotational symmetry in $x^{\prime }\in
\mathbb{R}^{n-1}$ one may assume that $x^{\prime }=\left\vert x^{\prime
}\right\vert \boldsymbol{\vec{e}}_{1}$ and we find through $e^{ix^{\prime
}\cdot \xi ^{\prime }}=e^{i\left\vert x^{\prime }\right\vert \xi
_{1}}=e^{ir\left\vert x^{\prime }\right\vert \cos \left( \varphi \right) }$
and $d\xi ^{\prime }=\left( \sin \varphi \right) ^{n-3}r^{n-2}d\varphi
drd\omega ^{\prime \prime }$, that%
\begin{gather}
F_{n,\gamma }\left( x\right) =\frac{1}{\left( 2\pi \right) ^{n-1}}\int_{%
\mathbb{S}^{n-3}}\left( \int_{0}^{\infty }\frac{\exp \left( -\left\vert
x_{n}\right\vert r\cos \left( \frac{1}{2}\gamma \right) \right) \sin \left(
\left\vert x_{n}\right\vert r\sin \left( \frac{1}{2}\gamma \right) +\frac{1}{%
2}\gamma \right) }{2r^{3}\sin \left( \gamma \right) }\right.   \notag
\\
\hspace{5cm}\left.\cdot \left( \int_{0}^{\pi }e^{ir\left\vert x^{\prime
}\right\vert \cos \varphi }\left( \sin \varphi \right) ^{n-3}d\varphi
\right) r^{n-2}dr\rule{0mm}{6mm}\right) d\omega ^{\prime \prime }.
\label{conti2}
\end{gather}%
Notice that
\begin{equation}
\int_{0}^{\pi }e^{is\cos \varphi }\left( \sin \varphi \right) ^{n-3}d\varphi
=\int_{0}^{\pi }\cos \left( s\cos \varphi \right) \ \left( \sin \varphi
\right) ^{n-3}d\varphi =C_{n}\left( s\right)  \label{sC}
\end{equation}%
with $C_{n}$ as in Lemma \ref{calclem} and%
\begin{equation*}
C_{n}(0)=\int_{0}^{\pi }\left( \sin \varphi \right) ^{n-3}d\varphi =\tfrac{%
\sqrt{\pi }~\Gamma \left( \frac{n-2}{2}\right) }{\Gamma \left( \frac{n-1}{2}%
\right) }>0.
\end{equation*}%
We find for $\left\vert x^{\prime }\right\vert \not=0$ and $x_{n}\neq 0$
from (\ref{conti2}) that%
\begin{align*}
F_{n,\gamma }\left( x\right) &=\frac{\sigma _{n-2}}{\left( 2\pi \right)
^{n-1}}\int_{0}^{\infty }C_{n}\left( r\left\vert x^{\prime }\right\vert
\right) \tfrac{\exp \left( -\left\vert x_{n}\right\vert r\cos \left( \frac{1%
}{2}\gamma \right) \right) \sin \left( \left\vert x_{n}\right\vert r\sin
\left( \frac{1}{2}\gamma \right) +\frac{1}{2}\gamma \right) }{2\sin \left(
\gamma \right) }r^{n-5}dr \\
&=\frac{1}{2^{n-2}\Gamma \left( \frac{n-2}{2}\right) \pi ^{n/2}\left\vert
x^{\prime }\right\vert ^{n-4}}\int_{0}^{\infty }C_{n}\left( s\right) \tfrac{%
\exp \left( -\frac{\left\vert x_{n}\right\vert }{\left\vert x^{\prime
}\right\vert }s\cos \left( \frac{1}{2}\gamma \right) \right) \sin \left(
\frac{\left\vert x_{n}\right\vert }{\left\vert x^{\prime }\right\vert }s\sin
\left( \frac{1}{2}\gamma \right) +\frac{1}{2}\gamma \right) }{2\sin \left(
\gamma \right) }s^{n-5}ds.
\end{align*}%
For $\left\vert x^{\prime }\right\vert =0$ and $x_{n}\neq 0$ one obtains%
\begin{align*}
F_{n,\gamma }\left( x\right) &=\frac{\sigma _{n-2}}{\left( 2\pi \right)
^{n-1}}\tfrac{\sqrt{\pi }~\Gamma \left( \frac{n-2}{2}\right) }{\Gamma \left(
\frac{n-1}{2}\right) }\int_{0}^{\infty }\frac{\exp \left( -\left\vert
x_{n}\right\vert r\cos \left( \frac{1}{2}\gamma \right) \right) \sin \left(
\left\vert x_{n}\right\vert r\sin \left( \frac{1}{2}\gamma \right) +\frac{1}{%
2}\gamma \right) }{2\sin \left( \gamma \right) }r^{n-5}dr \\
&=\frac{\sqrt{\pi }}{2^{n-1}\Gamma \left( \frac{n-1}{2}\right) \pi
^{n/2}\left\vert x_{n}\right\vert ^{n-4}\sin \left( \gamma \right) }\mathrm{%
Re}\left( ie^{-i\gamma /2}\int_{0}^{\infty }\exp \left( -se^{i\gamma
/2}\right) s^{n-5}ds\right) \\
&=\frac{1}{2^{n-1}\Gamma \left( \frac{n-1}{2}\right) \pi
^{(n-1)/2}\left\vert x_{n}\right\vert ^{n-4}\sin \left( \gamma \right) }%
\mathrm{Re}\left( i\left( e^{i\gamma /2}\right) ^{3-n}\Gamma \left(
n-4\right) \right) \\
&=\frac{\Gamma \left( n-4\right) }{2^{n-1}\Gamma \left( \frac{n-1}{2}\right)
\pi ^{(n-1)/2}\left\vert x_{n}\right\vert ^{n-4}\sin \left( \gamma \right) }%
\sin \left( \frac{n-3}{2}\gamma \right) ,
\end{align*}%
which shows (\ref{critfo}). \end{proof}

Whenever $\frac{n-3}{2}\gamma $ reaches values above $\pi $, $F_{n,\gamma }$
in (\ref{critfo}) changes sign. This is the case when $n\geq 6$, so we may
conclude:

\begin{proposition}\label{prop:4thorder_signchange}
For all $n\geq 6$ there are $\gamma \in \left( 0,\pi \right) $ such that $%
F_{n,\gamma }\left( 0,\dots ,0,x_{n}\right) <0$ for $x_{n}\neq 0$.
\end{proposition}

Together with Theorem~\ref{thm:pos_dir}, this proposition yields the proof of Theorem~\ref{thm:negative_directions_any_dimension}.
Notice that whenever $n\geq 8$, the fundamental solution changes sign even
for $\gamma <\frac{1}{2}\pi $, where the level hypersurfaces of the symbol
are still convex.

\begin{example}
{\rm
For $n=8$ one finds that%
\[
C_{8}\left( s\right) =16\frac{\left( 3-s^{2}\right) \sin s-3s\cos s}{s^{5}}=%
\tfrac{16}{15}-\tfrac{8}{105}s^{2}+\mathcal{O}\left( s^{4}\right) .
\]%
We have $F_{8,\gamma }\left( 0,\dots ,0,x_{8}\right) =\frac{\sin \left(
\frac{5}{2}\gamma \right) }{40 \pi^4\left\vert x_{8}\right\vert ^{4} \sin \left(
\gamma \right) }$, which is negative for $\gamma \in \left( \frac{2}{5}\pi ,%
\frac{4}{5}\pi \right) $. To get an impression for which combinations of $\gamma $ and $%
\beta :=\arctan \left( \frac{\left\vert x_{8}\right\vert }{\left\vert
x^{\prime }\right\vert }\right) $ the fundamental solution is negative,
Figure \ref{nis8} contains a graph of
\begin{equation}
g\left( \beta ,\gamma \right)
=\left\vert x\right\vert ^{4}F_{8,\gamma }( x ) \text{ \ for } x=(x^\prime,x_8)\text{ with }\left\vert x_{8}\right\vert
=\left\vert x^{\prime }\right\vert \tan \beta ,\label{eqge}
\end{equation}
which indeed only depends on $%
\beta ,\gamma \in \left[ 0,\frac{1}{2}\pi \right] \times \left( 0,\pi
\right) $.
}
\end{example} 

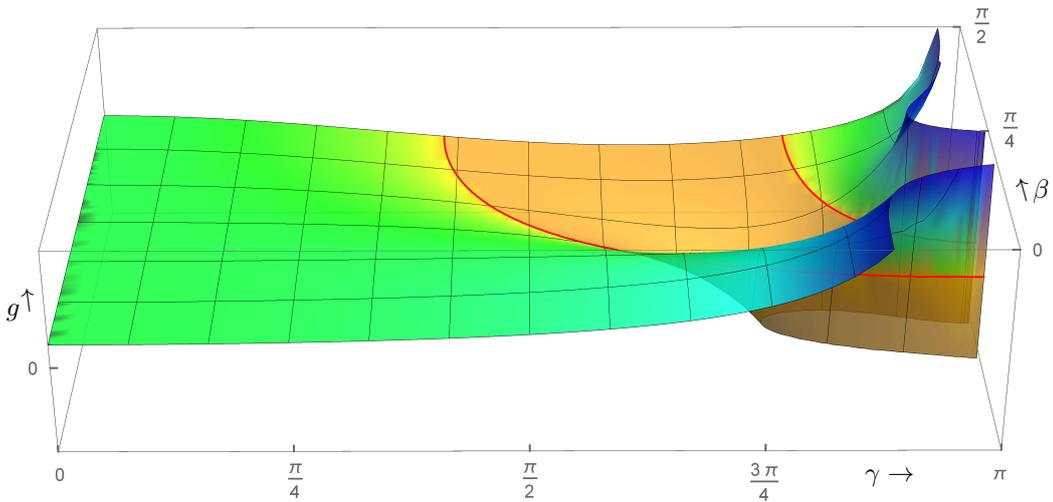
\begin{figure}[h]
  \centering\setlength{\unitlength}{.09\textwidth}
  \begin{picture}(10,5)(0,0)
  \put(.3,0){\includegraphics[width=.9\textwidth]{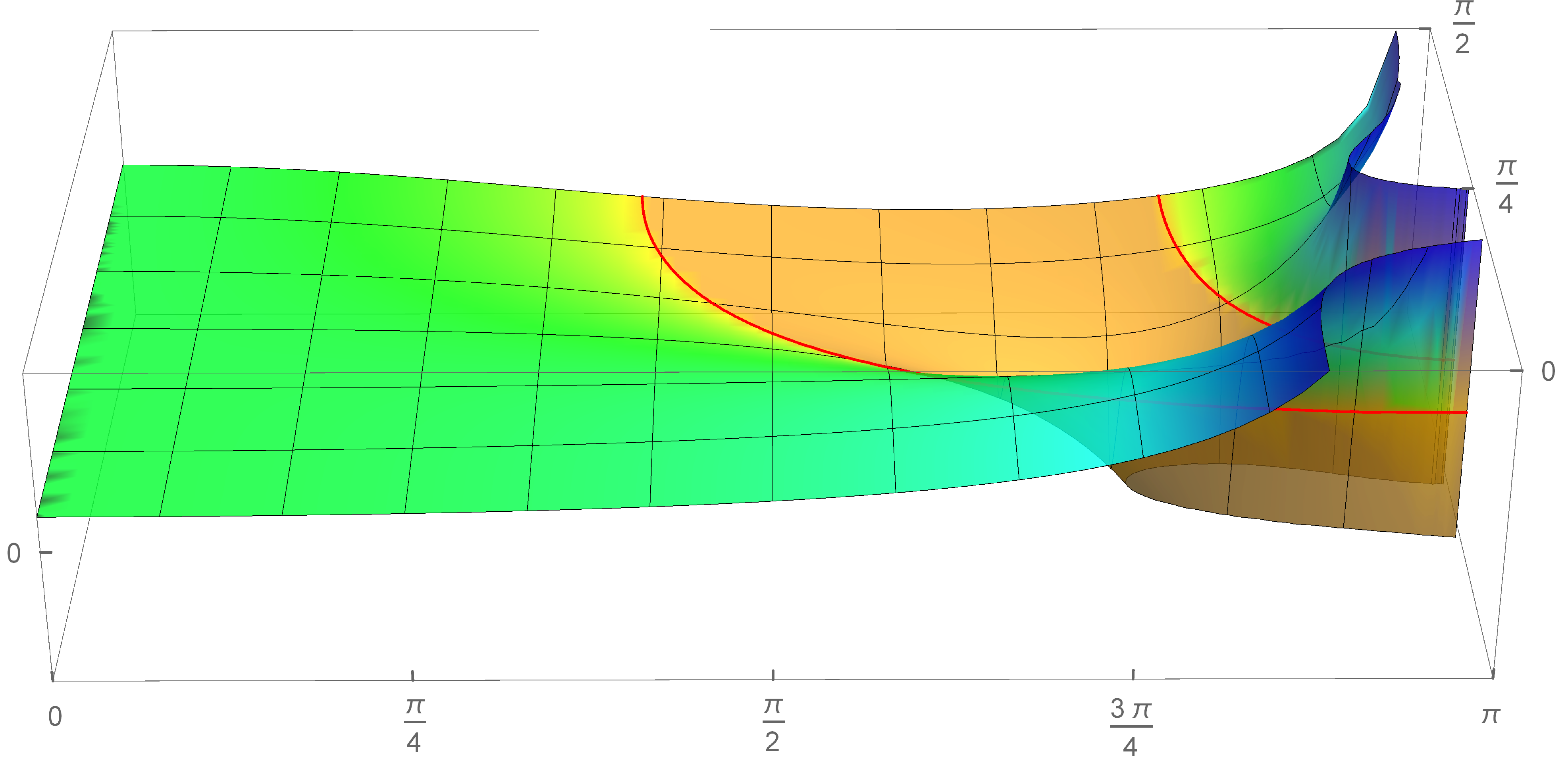}}
  \put(8.5,.22){$\gamma\to$}
  \put(9.95,3){\rotatebox{15}{$\uparrow$}$\beta$}
  \put(.13,1.85){$g$\raisebox{1mm}{\rotatebox{7}{$\uparrow$}}}
  \end{picture}

  \caption{Sign changing depending on $\gamma$ for $F_{8,\gamma}$ through a sketch of $g$ from (\ref{eqge}). The nodal lines appear in red (dark).}\label{nis8}
\end{figure}


	\section{Developing John's formulae}\label{sec:explicit_calc_FS}
	
	\subsection{Odd dimensions $\boldsymbol{n>2m}$: explicit fundamental solutions}
	\label{subsec:explicit_calc_FS_odd}

	In this section we prove Theorem \ref{TheoremODD} starting from John's formula \eqref{formulaJohn}. In other words, recalling that $K_L(x,y)=K_L(0,x-y)$, we shall compute the iterated Laplacian for
	\begin{equation}\label{F}
	F(y):=\int_{\substack{|\xi|=1\\y\cdot\xi\,>\,0}}\dfrac{y\cdot\xi}{Q(\xi)}d\mathcal{H}^{n-1}(\xi).
	\end{equation}

	\vskip0.2truecm
	\subsubsection{The first iteration: $\boldsymbol{\Delta F}$}
	\begin{prop}\label{Prop_1st_iteration}
		With the notation above, we have
		\begin{equation}\label{dF_y}
		\nabla F(y)=\int_{\substack{|\xi|=1\\y\cdot\xi\,>\,0}}\frac{\xi^T}{Q(\xi)}d\mathcal{H}^{n-1}(\xi),
		\end{equation}
		\begin{equation}\label{first_iteration}
		\Delta F(y)=\frac1{|y|}\int_{\substack{|\xi|=1\\y\cdot\xi\,=\,0}}\dfrac1{Q(\xi)}d\mathcal{H}^{n-2}(\xi).
		\end{equation}
	\end{prop}
	\noindent Before going into details of the proof, let us remark an important consequence of Proposition \ref{Prop_1st_iteration}.
	\begin{remark}\label{fund_sols_first_iteration}
		In the case $n=2m+1$, e.g. for a fourth order elliptic operator in $\R^5$, \eqref{formulaJohn} and \eqref{first_iteration} imply
		\begin{equation}\label{m2n5}
		K_L(x,y)=\frac1{32\pi^4|x-y|}\int_{\substack{|\xi|=1\\(x-y)\cdot\xi\,=\,0}}\dfrac1{Q(\xi)}d\mathcal{H}^{n-2}(\xi),
		\end{equation}
		which is thus a \textit{positive} fundamental solution having the \textit{expected order of singularity}.
		The only difference with the polyharmonic case is that an "angular dependent" positive factor appears.
		Notice that for the model polyharmonic case, as $Q_{\Delta^m}(\xi):=|\xi|^{2m}$, then such factor is identically 1
		and we of course retrieve its well-known fundamental solution.
	\end{remark}

	We first recall a classical result about integrations of differential forms (see for instance \cite[Satz 3]{Forster}).
	\begin{lem}\label{Lemma_Forster}
		Let $M$ be an oriented hypersurface with exterior normal vector field $\nu(x)$, which means that for any admissible parametrisation $\Phi$
		with $x=\Phi(t)$ we have $\operatorname{det} \left(\nu(x), \frac{\partial \Phi}{\partial t_1}(t),\ldots,
		\frac{\partial \Phi}{\partial t_{n-1}}(t) \right)>0$.
		Let further $A\subseteq M$ be a compact submanifold and $\boldsymbol{f}=(f_i)_{i=1}^n:M\to\R^n$ be a vector field. Then we have
		\begin{equation*}
		\int_A \boldsymbol f(x)\cdot\nu(x)dS(x)=\int_A \bigg(\sum_{i=1}^{n-1}(-1)^{i-1}f_i\,dx_1\wedge\cdots\wedge \widehat{dx_i}\wedge\cdots\wedge dx_n\bigg),
		\end{equation*}
		with the convention that $\widehat{dx_i}$ means that this factor is missing.
	\end{lem}

	\begin{proof}[Proof of Proposition \ref{fund_sols_first_iteration}]
		\textbf{Step 1.} Let us first compute $\nabla F$ on points $y=(y_1,0,\cdots,0)\not=\bf{0}$.
		Here and everywhere in what follows we assume $y_1>0$.
		Then, by means of a rigid motion in $\R^n$,
		we will extend the result to the general case.\\
		Writing $y=r\eta$, where $|\eta|=1$, we have
		\begin{equation*}
		F(y)=r\int_{\substack{|\xi|=1\\\eta\cdot\xi\,>\,0}}\dfrac{\eta\cdot\xi}{Q(\xi)}d\mathcal{H}^{n-1}(\xi).
		\end{equation*}
		Noticing that the derivative in the first direction is indeed a normal derivative,
		\begin{equation}\label{F_rad}
		\frac{\partial^k}{\partial y_1^k} F(y_1,0,\cdots,0)=\frac{d^k}{dr^k}F(r\eta),
		\end{equation}
		therefore we infer
		\begin{equation}\label{dF_d1}
		\frac{\partial}{\partial y_1} F(y_1,0,\cdots,0)=\frac1{|y|}\int_{\substack{|\xi|=1\\y\cdot\xi\,>\,0}}\dfrac{y\cdot\xi}{Q(\xi)}d\mathcal{H}^{n-1}(\xi)=\int_{\substack{|\xi|=1\\y\cdot\xi\,>\,0}}\dfrac{\xi_1}{Q(\xi)}d\mathcal{H}^{n-1}(\xi),
		\end{equation}
		\begin{equation*}\label{d^2F_d11}
		\frac{\partial^2}{\partial y_1^2} F(y_1,0,\cdots,0)=0.
		\end{equation*}
		Let us now compute all other first and second derivatives, which thus involve tangential directions. Without loss of generality, we may consider just $\partial_3F$ and $\partial^2_{33}F$, the general case being similar.  To this aim, we introduce the rotation matrix $B_\varphi$ which roughly speaking exchanges the first direction with the third:
		\begin{equation}\label{B_phi}
		B_\varphi=\left[\begin{array}{ccc|c}
		\cos\varphi & 0 & -\sin\varphi &  \\
		0 & 1 & 0 & \bf0 \\
		\sin\varphi & 0 & \cos\varphi &  \\ \hline
		& \bf{0} &  & \bf I
		\end{array}\right]
		\end{equation}
		Moreover, define $H(\varphi):=F(B_\varphi y)$. Recalling that $y_i=0$ for all $i>1$ and differentiating $H$ in $\varphi=0$, we have
		\begin{equation}\label{dF_d3-dH_dphi}
		\begin{split}
		\frac{\partial H}{\partial\varphi}\bigg|_{\varphi=0}&=\partial_1 F{|_{(y_1\cos\varphi,0,y_1\sin\varphi,0,\cdots,0)}}\bigg|_{\varphi=0}(-\sin\varphi y_1)|_{\varphi=0}+\partial_3F{|_{(y_1\cos\varphi,0,y_1\sin\varphi,0,\cdots,0)}}\bigg|_{\varphi=0}(\cos\varphi y_1)|_{\varphi=0}\\
		&=y_1\partial_3F(y).
		\end{split}
		\end{equation}
		On the other hand, by definition
		\begin{equation*}
		H(\varphi)=\int_{\substack{|\xi|=1\\B_\varphi y\cdot\xi\,>\,0}}\dfrac{B_\varphi y\cdot\xi}{Q(\xi)}d\mathcal{H}^{n-1}(\xi)
		\stackrel{\xi=B_\varphi\xi'}{=}\int_{\substack{|\xi'|=1\\y\cdot\xi'\,>\,0}}\frac{y\cdot\xi'}{Q(B_\varphi\xi')}d\mathcal{H}^{n-1}(\xi').
		\end{equation*}
		Therefore, using Lemma \ref{Lemma_Forster},
		\begin{equation}\label{dH_dphi}
		\begin{split}
		\frac{\partial H}{\partial\varphi}\bigg|_{\varphi=0}&
		=\int_{\substack{|\xi'|=1\\y\cdot\xi'\,>\,0}}(y\cdot\xi')\frac{\partial}{\partial\varphi}\bigg|_{\varphi=0}\frac1{Q(B_\varphi\xi')}d\mathcal{H}^{n-1}(\xi')\\
		&=\int_{\substack{|\xi'|=1\\\xi'_1\,>\,0}}(y_1\xi'_1)\bigg(-\xi'_3\partial_1\frac1{Q(\xi')}+\xi'_1\partial_3\frac1{Q(\xi')}\bigg)d\mathcal{H}^{n-1}(\xi')\\
		&=\int_{\substack{|\xi'|=1\\\xi'_1\,>\,0}}(y_1\xi'_1)\begin{pmatrix}\partial_3\frac1{Q(\xi')}\\0\\-\partial_1\frac1{Q(\xi')}\\0\\\cdots\\0\end{pmatrix}
		\cdot\xi' d\mathcal{H}^{n-1}(\xi')\\
		&=y_1\int_{\substack{|\xi'|=1\\\xi'_1\,>\,0}}\bigg(\xi'_1\partial_3\frac1{Q(\xi')}\,d\xi'_2\wedge\cdots\wedge d\xi'_n\;-\;\xi'_1\partial_1\frac1{Q(\xi')}
		\,d\xi'_1\wedge d\xi'_2\wedge d\xi'_4\wedge\cdots\wedge d\xi'_n\bigg).
		\end{split}
		\end{equation}
		We observe that for the $(n-2)$\,-\,form
		\begin{equation*}
		\omega:=-\frac{\xi'_1}{Q(\xi')}d\xi'_2\wedge d\xi'_4\wedge\cdots\wedge d\xi'_n
		\end{equation*}
		we have
		\begin{equation}\label{d_omega}
		\begin{split}
		d\omega&=-\left(d\frac{\xi'_1}{Q(\xi')}\right)\wedge d\xi'_2\wedge d\xi'_4\wedge\cdots\wedge d\xi'_n\\
		&=-\partial_1\left(\frac{\xi'_1}{Q(\xi')}\right)d\xi'_1\wedge d\xi'_2\wedge d\xi'_4\wedge\cdots\wedge d\xi'_n-\partial_3\left(\frac{\xi'_1}{Q(\xi')} \right)d\xi'_3\wedge d\xi'_2\wedge d\xi'_4\wedge\cdots\wedge d\xi'_n\\
		&=-\xi'_1\left(\partial_1\frac1{Q(\xi')}\right)d\xi'_1\wedge d\xi'_2\wedge d\xi'_4\wedge\cdots\wedge d\xi'_n\,-\,\frac1{Q(\xi')}d\xi'_1\wedge d\xi'_2\wedge d\xi'_4\wedge\cdots\wedge d\xi'_n\\
		&\quad +\xi'_1\partial_3\frac1{Q(\xi')}d\xi'_2\wedge\cdots\wedge d\xi'_n.
		\end{split}
		\end{equation}
		Hence, by \eqref{dF_d3-dH_dphi}-\eqref{d_omega} and Stokes' Theorem, noticing that $\omega=0$ on $\{\xi'_1=0\}$, we infer
		\begin{equation*}
		\begin{split}
		\frac{\partial F}{\partial y_3}(y_1,0,\cdots,0)&=\frac1{y_1}\frac{\partial H}{\partial\varphi}\bigg|_{\varphi=0}=\int_{\substack{|\xi'|=1\\\xi'_1\,>\,0}}\frac1{Q(\xi')}d\xi'_1\wedge d\xi'_2\wedge d\xi'_4\wedge\cdots\wedge d\xi'_n+\int_{\substack{|\xi'|=1\\\xi'_1\,>\,0}}d\omega\\
		&=\int_{\substack{|\xi'|=1\\\xi'_1\,>\,0}}\frac1{Q(\xi')}\nu_3(\xi')d\mathcal{H}^{n-1}(\xi')\\
		&=\int_{\substack{|\xi'|=1\\\xi'_1\,>\,0}}\frac{\xi'_3}{Q(\xi')}d\mathcal{H}^{n-1}(\xi').
		\end{split}
		\end{equation*}
		Analogously one may compute $\partial_k F(y_1,0,\cdots,0)$ for $k\not=1$ and then obtain
		\begin{equation}\label{dF_dk}
		\frac{\partial F}{\partial y_k}(y_1,0,\cdots,0)=\int_{\substack{|\xi'|=1\\y\cdot\xi'\,>\,0}}\frac{\xi'_k}{Q(\xi')}d\mathcal H^{n-1}(\xi').
		\end{equation}
		Indeed, it is sufficient to consider instead of $B_\varphi$ a similar matrix corresponding to a rotation in the plane $\langle y_1,y_k\rangle$. Hence, \eqref{dF_d1} and \eqref{dF_dk} yield
		\begin{equation*}
		\nabla F(y_1,0,\cdots,0)=\int_{\substack{|\xi'|=1\\y\cdot\xi'\,>\,0}}\frac{(\xi')^T}{Q(\xi')}d\mathcal H^{n-1}(\xi').
		\end{equation*}
		\vskip0.2truecm

		\textbf{Step 2.} Now we want to to extend this identity to a generic point $y\in\R^n\setminus\{0\}$. To this aim, let us write $y=|y|b_1$, where $|b_1|=1$, and complete this unit vector to a matrix $B=\big(b_1\,|\,\cdots\,|\,b_n\big)\in SO(n)$. Notice that
		$$y=|y|b_1=B\cdot\big(|y|,0,\cdots,0\big)^T\qquad\mbox{and}\qquad(|y|,0,\cdots,0\big)^T=B^Ty.$$
		Moreover, define $\tilde F:=F\circ B$. From \eqref{F} we infer for all $z\in\R^n\setminus\{0\}$ that
		\begin{equation*}
		\tilde F(z)=F(Bz)= \int_{\substack{|\xi|=1\\Bz\cdot\xi\,>\,0}}\frac{Bz\cdot\xi}{Q(\xi)}d\mathcal H^{n-1}(\xi)\stackrel{\xi=B\xi'}{=}\int_{\substack{|\xi'|=1\\z\cdot\xi'\,>\,0}}\frac{z\cdot\xi'}{Q(B\xi')}d\mathcal{H}^{n-1}(\xi').
		\end{equation*}
		Therefore,
		\begin{equation*}
		\begin{split}
		\nabla F(y)&=\nabla F(B(|y|,0,\cdots,0)^T)=\nabla\tilde F(|y|,0,\cdots,0)B^T\\
		&=\int_{\substack{|\xi'|=1\\\xi'_1\,>\,0}}\frac{\xi'^T}{Q(B\xi')}d\mathcal{H}^{n-1}(\xi')\cdot B^T =\int_{\substack{|\xi'|=1\\\xi'_1\,>\,0}}\frac{(B\xi')^T}{Q(B\xi')}d\mathcal{H}^{n-1}(\xi').
		\end{split}
		\end{equation*}
		The change of variable $\xi=B\xi'$ yields finally \eqref{dF_y}.

		\vskip0.2truecm

		\textbf{Step 3.} Now it is the turn of \textit{second derivatives}. We compute them with the same method we applied so far, so first we consider the
		easier case $y=(y_1,0,\cdots,0)^T$ with $y_1>0$ and then we extend this to a general $y\in\R^n\setminus\{0\}$.
		\vskip0.2truecm
		Once again, we may for simplicity consider just $\partial_3F$, the other cases being similar as already mentioned. Let $B_\varphi$ be as in \eqref{B_phi} and for $y=(y_1,0,\cdots,0)^T$ define
		\begin{equation*}
		\tilde H(\varphi):=\partial_3F(B_\varphi y)=\int_{\substack{|\xi|=1\\B_\varphi y\cdot\xi\,>\,0}}\frac{\xi_3}{Q(\xi)}d\mathcal{H}^{n-1}(\xi)\stackrel{\xi=B_\varphi\xi'}{=}\int_{\substack{|\xi'|=1\\\xi'_1\,>\,0}}\frac{(B_\varphi\xi')_3}{Q(B_\varphi\xi')}d\mathcal{H}^{n-1}(\xi').
		\end{equation*}
		Exactly as for $F$ in \eqref{dF_d3-dH_dphi}, we have
		\begin{equation*}
		\frac{\partial^2 F}{\partial y_3^2}(y_1,0,\cdots,0)=\frac1{y_1}\frac{\partial\tilde H}{\partial\varphi}\bigg|_{\varphi=0}.
		\end{equation*}
		We denote $R(z):=\frac{z_3}{Q(z)}$ for all $z\in\R^n$. Then we have
		\begin{equation*}
		\begin{split}
		\frac{\partial\tilde H}{\partial\varphi}\bigg|_{\varphi=0} &=\frac{\partial}{\partial\varphi}\bigg|_{\varphi=0}
		\int_{\substack{|\xi'|=1\\\xi'_1\,>\,0}}R(B_\varphi\xi')\,d\mathcal{H}^{n-1}(\xi')\\
		&=\int_{\substack{|\xi'|=1\\\xi'_1\,>\,0}}\big(-\xi'_3\partial_1R(\xi')+\xi'_1\partial_3R(\xi')\big)\,d\mathcal{H}^{n-1}(\xi')\\
		&=\int_{\substack{|\xi'|=1\\\xi'_1\,>\,0}}\begin{pmatrix}\partial_3R(\xi')\\0\\-\partial_1R(\xi')\\0\\\cdots\\0\end{pmatrix}
		\cdot\nu(\xi')\,d\mathcal{H}^{n-1}(\xi')\\
		&=\int_{\substack{|\xi'|=1\\\xi'_1\,>\,0}}\big(\partial_3R(\xi')\,d\xi'_2\wedge\cdots\wedge d\xi'_n-\partial_1R(\xi')\,d\xi'_1
		\wedge d\xi'_2\wedge d\xi'_4\wedge\cdots\wedge d\xi'_n\big)\\
		&=\int_{\substack{|\xi'|=1\\\xi'_1\,>\,0}}d\xi'_2\wedge\big(\partial_1R(\xi')d\xi'_1+\partial_3R(\xi')d\xi'_3\big)\wedge d\xi'_4\wedge\cdots\wedge d\xi'_n\\
		&=-\int_{\substack{|\xi'|=1\\\xi'_1\,>\,0}}d\,\big(R(\xi') d\xi'_2\wedge d\xi'_4\wedge\cdots\wedge d\xi'_n\big)\\
		&=-\int_{\substack{|\xi'|=1\\\xi'_1\,=\,0}}R(\xi')\, d\xi'_2\wedge d\xi'_4\wedge\cdots\wedge d\xi'_n,
		\end{split}
		\end{equation*}
		having applied Stokes' Theorem. Let $\tilde\nu$ denote the exterior normal  of the half-sphere on its boundary, which gives its induced  orientation.
		Hence, using Lemma \ref{Lemma_Forster} and the definition of $R$,
		\begin{equation*}
		\frac{\partial\tilde H}{\partial\varphi}\bigg|_{\varphi=0}=\int_{\substack{|\xi'|=1\\\xi'_1\,=\,0}}\tilde\nu(\xi')\cdot
		\begin{pmatrix}0\\R(\xi')\\0\\\cdots\\0\end{pmatrix}d\mathcal{H}^{n-2}(\xi')=\int_{\substack{|\xi'|=1\\\xi'_1\,=\,0}}\frac{{\xi'_3}^2}{Q(\xi')}\,
		d\mathcal{H}^{n-2}(\xi').
		\end{equation*}

		\noindent Therefore, we may conclude that for all $k\in\{2,\cdots,n\}$ we have
		\begin{equation*}
		\frac{\partial^2 F}{\partial y_k^2}(y_1,0,\cdots,0)=\frac1{|y|}\int_{\substack{|\xi'|=1\\\xi'_1\,=\,0}} \frac{{\xi'_k}^2}{Q(\xi')}d\mathcal{H}^{n-2}(\xi').
		\end{equation*}
		Recalling that $\partial^2_{11}F(y_1,0,\cdots,0)=0$, we thus infer
		\begin{equation}\label{DeltaF}
		\Delta F(y_1,0,\cdots,0)=\frac1{|y|}\int_{\substack{|\xi'|=1\\\xi'_1\,=\,0}}\frac1{Q(\xi')}d\mathcal H^{n-2}(\xi').
		\end{equation}
		\vskip0.2truecm

		\textbf{Step 4.} Let now $y\in\R^n\setminus\{0\}$ which we write as $y=|y|b_1$, with $|b_1|=1$, and let us complete this
		unit vector to a matrix $B:=\big(b_1\,|\,\cdots\,|\,b_n\big)\in SO(n)$. Moreover, recall $\tilde F:=F\circ B$. Then, one has
		\begin{equation*}
		\Delta F(z)=Tr(\nabla^2 F(z))=Tr(B\nabla^2\tilde F{|_{B^Tz}}B^T)=Tr(\nabla^2\tilde F{|_{B^Tz}})=\Delta\tilde F(B^Tz),
		\end{equation*}
		and therefore by \eqref{DeltaF},
		\begin{equation*}
		\begin{split}
		\Delta F(y)&=\Delta\tilde F(|y|,0,\cdots,0)=\frac1{|y|}\int_{\substack{|\xi'|=1\\B^Ty\cdot\xi'\,=\,0}}\frac1{Q(B\xi')}d\mathcal H^{n-2}(\xi')
		\stackrel{\xi=B\xi'}{=}\frac1{|y|}\int_{\substack{|\xi|=1\\y\cdot\xi\,=\,0}}\frac1{Q(\xi)}d\mathcal H^{n-2}(\xi).
		\end{split}
		\end{equation*}
	\end{proof}

	\subsubsection{The k-th iteration $\boldsymbol{\Delta^k F}$ and the proof of Theorem \ref{TheoremODD}}\label{Section3}

	The following theorem provides a general formula for the iterated Laplacian of $F$. We will see that tangential derivatives of the symbol play a fundamental role in the formula.

	\begin{thm}\label{conjecture_j}
		\begin{equation}\label{k_iteration}
		\Delta^k F(y)=\frac1{|y|^{2k-1}}\sum_{j=1}^{k}d_{k,j}(n,m)
		\int_{\substack{|\xi|=1\\y\cdot\xi\,=\,0}}\nabla^{2(k-j)}\frac1{Q(\xi)}\left(\frac y{|y|}^{\otimes2(k-j)}\right)d\mathcal{H}^{n-2}(\xi),
		\end{equation}
		where
		$$d_{k,1}(n,m)=1$$
		and for $j=2\,\ldots,k$:
		\begin{equation}\label{d_kj}
		d_{k,j}(n,m)=(-1)^{j-1} c_{k,j}\prod_{\ell=2}^j (n-2m-(2k-2\ell+3))
		\end{equation}
		with (using the convention that the product is $1$ whenever $j=1$)
		\begin{equation}\label{c_kj}
		c_{k,j}=\prod_{\ell=2}^j \frac{(k-\ell+1)(2k-2\ell+1)}{\ell-1}.
		\end{equation}
	\end{thm}
	\noindent Note that Theorem \ref{TheoremODD} is an immediate consequence of Theorem \ref{conjecture_j}. This follows from putting $k=\tfrac{n-2m+1}2$, where $d_{k,j}=0$ for all $j\geq2$.

	\vskip0.2truecm

	The rest of the subsection is devoted to proving Theorem \ref{conjecture_j}, the strategy being the following. Firstly, we show that each term of the sum, namely
	\begin{align*}
	\int_{\substack{|\xi|=1\\y\cdot\xi\,=\,0}}\nabla^{2j}\frac1Q(\xi)\left(\frac y{|y|}^{\otimes2j}\right)d\mathcal H^{n-2}(\xi),
	\end{align*}
	once the Laplacian is applied, produces only terms of the same kind (so only even derivatives of $\frac1Q$ are involved) with order at most $2j+2$, each of them multiplied by the same suitable power of $\frac1{|y|}$. This is achieved in Proposition \ref{Prop0.3}. As a consequence, we obtain some recurrence formulae for the coefficients $d_{k,j}$ in the proof of Theorem~\ref{conjecture_j}. These relations will be important to finally prove the theorem by induction.

	\vskip0.2truecm

	\noindent Let us fix $k\in\N$ and $j\in\{0,\cdots,k-1\}$ and define
	\begin{equation*}
	J_{k,j}(y):=\frac1{|y|^{2k-1}}\int_{\substack{|\xi|=1\\y\cdot\xi\,=\,0}}\nabla^{2j}\frac1Q(\xi)\left(\frac y{|y|}^{\otimes2j}\right)d\mathcal H^{n-2}(\xi).
	\end{equation*}

	\begin{prop}\label{Prop0.3}
		\begin{equation}\label{dJ_y}
		\begin{split}
		\nabla J_{k,j}(y)=\frac1{|y|^{2k}}\int_{\substack{|\xi|=1\\y\cdot\xi\,=\,0}}\bigg[&(2j)\nabla^{2j}\frac1Q(\xi)\left(\frac y{|y|}^{\otimes2j-1},\cdot\right)-\xi^T\cdot\nabla^{2j+1}\frac1Q(\xi)\left(\frac y{|y|}^{\otimes2j+1}\right)\\
		&-(2k-1+2j)\nabla^{2j}\frac1Q(\xi)\left(\frac y{|y|}^{\otimes2j}\right)\frac {y^T}{|y|}\bigg]d\mathcal{H}^{n-2}(\xi),
		\end{split}
		\end{equation}
		\begin{equation}\label{J_iteration_1}
		\Delta J_{k,j}(y_1,0,\cdots,0)=\frac1{|y|^{2k+1}}\int_{\substack{|\xi'|=1\\\xi'_1\,=\,0}}\bigg[c_0\frac{\partial^{2j}}{\partial {\xi'_1}^{2j}}\frac1Q(\xi')+c_\Delta\frac{\partial^{2j-2}}{\partial {\xi'_1}^{2j-2}}\frac1Q(\xi')+\frac{\partial^{2j+2}}{\partial {\xi'_1}^{2j+2}}\frac1Q(\xi')\bigg]d\mathcal{H}^{n-2}(\xi'),
		\end{equation}
		where
		\begin{equation}\label{c_0}
		c_0(n,m,k,j)=2k(2k-1)+(n-1)(1-2k-2j) +2(4j+1)(m+j)
		\end{equation}
		and
		\begin{equation}\label{c_Delta}
		c_\Delta(n,m,j)=4j(2j-1)(m+j-1)(2m+2j+1-n).
		\end{equation}
		Therefore, we obtain
		\begin{equation}\label{J_iteration}
		\Delta J_{k,j}(y)=c_\Delta(n,m,j)J_{k+1,j-1}+c_0(n,m,k,j)J_{k+1,j}+J_{k+1,j+1}.
		\end{equation}
	\end{prop}

	\begin{proof} In order to simplify the notation, as $k,j$ are fixed, we write $J(y)$ instead of $J_{k,j}(y)$.\\
		\textbf{Step 1.} Let $y=(y_1,0,\cdots,0)$, $y_1>0$. First of all, writing $J$ in polar coordinates
		\begin{equation*}
		J(r,\eta)=\frac1{r^{2k-1}}\int_{\substack{|\xi|=1\\\eta\cdot\xi\,=\,0}}\nabla^{2j}\frac1Q(\xi)\left(\eta^{\otimes2j}\right)d\mathcal H^{n-2}(\xi),
		\end{equation*}
		by \eqref{F_rad} one infers
		\begin{equation*}
		\frac{\partial}{\partial y_1} J(y_1,0,\cdots,0)=-\frac{2k-1}{|y|^{2k}}\int_{\substack{|\xi|=1\\y\cdot\xi\,=\,0}}\nabla^{2j}\frac1Q(\xi)\left(\frac y{|y|}^{\otimes2j}\right)d\mathcal H^{n-2}(\xi),
		\end{equation*}
		and
		\begin{equation}\label{d2J_d11}
		\frac{\partial^2}{\partial y_1^2} J(y_1,0,\cdots,0)=\frac{2k(2k-1)}{|y|^{2k+1}}\int_{\substack{|\xi|=1\\y\cdot\xi\,=\,0}}\nabla^{2j}\frac1Q(\xi)\left(\frac y{|y|}^{\otimes2j}\right)d\mathcal H^{n-2}(\xi).
		\end{equation}
		As in the proof of Proposition \ref{Prop_1st_iteration}, defining $H(\varphi):=J(B_\varphi y)$ with $B_\varphi$ as in \eqref{B_phi}, we have
		\begin{equation*}
		\begin{split}
		H(\varphi)&=\frac1{|y|^{2k-1}}\int_{\substack{|\xi|=1\\(B_\varphi y)\cdot\xi\,=\,0}}\nabla^{2j}\dfrac1Q(\xi)\left(\frac {B_\varphi y}{|y|}^{\otimes 2j}\right)d\mathcal{H}^{n-2}(\xi)\\
		&\stackrel{\xi=B_\varphi\xi'}{=}\frac1{|y|^{2k-1}}\int_{\substack{|\xi'|=1\\\xi'_1\,=\,0}}\nabla^{2j}\dfrac1Q(B_\varphi\xi')\left(\frac {B_\varphi y}{|y|}^{\otimes 2j}\right)d\mathcal{H}^{n-2}(\xi').
		\end{split}
		\end{equation*}
		We may rewrite the argument as
		\begin{equation*}
		\nabla^{2j}\dfrac1Q(B_\varphi\xi')\left(\frac {B_\varphi y}{|y|}^{\otimes 2j}\right)=\sum_{h=0}^{2j}\binom{2j}{h}\partial_{1^{2j-h}3^h}^{2j}\frac1Q(B_\varphi\xi')\cos^{2j-h}\varphi\,\sin^h\varphi,
		\end{equation*}
		with the shorter notation
		$$\partial_{i^k}^k:=\frac{\partial^k}{\partial x_i^k}.$$
		A differentiation with respect to $\varphi$ yields
		\begin{equation*}
		\begin{split}
		\frac{\partial H}{\partial\varphi}\bigg|_{\varphi=0}&=\frac1{|y|^{2k-1}}\int_{\substack{|\xi'|=1\\\xi'_1\,=\,0}}\sum_{h=0}^{2j}\binom{2j}{h}\bigg\{\frac d{d\varphi}\bigg|_{\varphi=0}\left(\partial_{1^{2j-h}3^h}^{2j}\frac1Q(B_\varphi\xi')\right)\left(\cos^{2j-h}\varphi\bigg|_{\varphi=0}\right)\left(\sin^h\varphi\bigg|_{\varphi=0}\right)\\
		&\quad +\,\partial_{1^{2j-h}3^h}^{2j}\frac1Q(\xi')\bigg[\left(\sin^h\varphi\bigg|_{\varphi=0}\right)\left(\frac d{d\varphi}\bigg|_{\varphi=0}\cos^{2j-h}\varphi\right)\\
		&\quad +\left(\cos^{2j-h}\varphi\bigg|_{\varphi=0}\right)\left(\frac d{d\varphi}\bigg|_{\varphi=0}\sin^h\varphi\right)\bigg]\bigg\}d\mathcal{H}^{n-2}(\xi').
		\end{split}
		\end{equation*}
		Since the terms which remain must have only cosines, everything vanishes except for the term with $h=0$ in the first sum and the one with $h=1$ in the third. Therefore,
		\begin{equation}\label{Step1J:1}
		\begin{split}
		\frac{\partial H}{\partial\varphi}\bigg|_{\varphi=0}=\frac1{|y|^{2k-1}}\int_{\substack{|\xi'|=1\\\xi'_1\,=\,0}}\left\{\frac d{d\varphi}\bigg|_{\varphi=0}\left(\partial_{1^{2j}}^{2j}\frac1Q(B_\varphi\xi')\right)+2j\,\partial_{1^{2j-1}3}^{2j}\frac1Q(\xi')\right\}d\mathcal{H}^{n-2}(\xi').
		\end{split}
		\end{equation}
		Moreover, we compute on $\{\xi'_1=0\}$
		\begin{equation}\label{Step1J:1_1}
		\begin{split}
		\frac d{d\varphi}\bigg|_{\varphi=0}\left(\partial_{1^{2j}}^{2j}\frac1Q(B_\varphi\xi')\right) &=\sum_{h=1}^n\partial_{1^{2j}h}^{2j+1}\frac1Q(\xi')\frac{\partial(B_\varphi\xi')_h}{\partial\varphi}\bigg|_{\varphi=0}\\
		&=\partial_{1^{2j+1}}^{2j+1}\frac1Q(\xi')\left(-\sin\varphi\xi_1'-\cos\varphi\xi_3'\right)\bigg|_{\varphi=0}+\partial_{1^{2j}3}^{2j+1}\frac1Q(\xi')\left(\cos\varphi\xi_1'-\sin\varphi\xi_3'\right)\bigg|_{\varphi=0}\\
		&=-\xi_3'\,\partial_{1^{2j+1}}^{2j+1}\frac1Q(\xi').
		\end{split}
		\end{equation}
		Inserting in \eqref{Step1J:1}, we infer
		\begin{equation*}
		\begin{split}
		\frac{\partial J}{\partial y_3}\bigg|_{(y_1,0,\cdots,0)}=\frac1{|y|}\frac{\partial H}{\partial\varphi}\bigg|_{\varphi=0}=\frac1{|y|^{2k}}\int_{\substack{|\xi'|=1\\\xi'_1\,=\,0}}\left(-\xi_3'\,\partial_{1^{2j+1}}^{2j+1}\frac1Q(\xi')+2j\,\partial_{1^{2j-1}3}^{2j}\,\frac1Q(\xi')\right)d\mathcal{H}^{n-2}(\xi').
		\end{split}
		\end{equation*}
		Of course, an analogous formula holds for any $h\in\{2,\cdots, n\}$, namely
		\begin{equation*}
		\begin{split}
		\frac{\partial J}{\partial y_h}\bigg|_{(y_1,0,\cdots,0)}=\frac1{|y|^{2k}}\int_{\substack{|\xi'|=1\\\xi'_1\,=\,0}}\left(-\xi_h'\,\partial_{1^{2j+1}}^{2j+1}\frac1Q(\xi')+2j\,\partial_{1^{2j-1}h}^{2j}\,\frac1Q(\xi')\right)d\mathcal{H}^{n-2}(\xi').
		\end{split}
		\end{equation*}
		\vskip0.2truecm
		\textbf{Step 2.} Let us now consider $y\in\R^n\setminus\{0\}$, so $y=|y|b_1$, where $|b_1|=1$, and let $B=\big(b_1\,|\,\cdots\,|\,b_n\big)\in SO(n)$. Defining $\tilde J:=J\circ B$, one has
		\begin{equation*}
		\begin{split}
		\nabla J(y)&=\nabla J(B(|y|,0,\cdots,0)^T)=\nabla\tilde J(|y|,0,\cdots,0)B^T\\
		&=\frac1{|y|^{2k}}\int_{\substack{|\xi'|=1\\\xi'_1\,=\,0}}
		\begin{pmatrix}-(2k-1)\partial_{1^{2j}}^{2j}\frac1Q(B\xi')\\2j\,\partial_{1^{2j-1}2}^{2j}\frac1Q(B\xi')-\xi_2'\partial_{1^{2j+1}}^{2j+1}\frac1Q(B\xi')\\\cdots\\2j\,\partial_{1^{2j-1}n}^{2j}\frac1Q(B\xi')-\xi_n'\partial_{1^{2j+1}}^{2j+1}\frac1Q(B\xi')\end{pmatrix}^Td\mathcal{H}^{n-2}(\xi')\cdot B^T\\
		&=\frac1{|y|^{2k}}\int_{\substack{|\xi'|=1\\\xi'_1\,=\,0}}2j\,\nabla\left(\partial_{1^{2j-1}}^{2j-1}\frac1Q(B\xi')\right)\cdot B^T-(B\xi')^T\cdot\partial_{1^{2j+1}}^{2j+1}\frac1Q(B\xi')\\
		&\quad -(2k-1+2j)\begin{pmatrix}\partial_{1^{2j}}^{2j}\frac1Q(B\xi')\\0\\\cdots\\0\end{pmatrix}^TB^Td\mathcal{H}^{n-2}(\xi').
		\end{split}
		\end{equation*}
		\noindent Returning therefore to the variable $\xi=B\xi'$, we get
		\begin{equation}\label{Step2J}
		\begin{split}
		\nabla J(y)&=\frac1{|y|^{2k}}\bigg[2j\,\int_{\substack{|\xi|=1\\y\cdot\xi\,=\,0}}\nabla^{2j}\frac1Q(\xi)\left(\frac y{|y|}^{\otimes 2j-1},\cdot\right)d\mathcal{H}^{n-2}(\xi) -\int_{\substack{|\xi|=1\\y\cdot\xi\,=\,0}}\nabla^{2j+1}\frac1Q(\xi)\left(\frac y{|y|}^{\otimes 2j+1}\right)\xi^T\,\, d\mathcal{H}^{n-2}(\xi)\\
		&\quad -(2k-1+2j)\int_{\substack{|\xi|=1\\y\cdot\xi\,=\,0}}\nabla^{2j}\frac1Q(\xi)\left(\frac y{|y|}^{\otimes 2j}\right)\frac {y^T}{|y|}d\mathcal{H}^{n-2}(\xi)\bigg],
		\end{split}
		\end{equation}
		that is, \eqref{dJ_y}.
		\vskip0.2truecm
		\textbf{Step 3.} Let us again consider $y=(y_1,0,\cdots,0)$, $y_1>0$, and compute $\Delta J(y)$. Defining $\tilde H(\varphi):=\partial_{y_3}J(B_\varphi y)$, according to the splitting in \eqref{Step2J}, we have
		\begin{equation}\label{Step3:H0H1H2}
		\tilde H(\varphi)=:\frac1{|y|^{2k}}\big(2j\,\tilde H_0(\varphi)-\tilde H_1(\varphi)-(2k-1+2j)\tilde H_2(\varphi)\big).
		\end{equation}
		Let us differentiate with respect to $\varphi$ term by term. Firstly,
		\begin{equation*}
		\begin{split}
		\tilde H_0(\varphi)&=\int_{\substack{|\xi|=1\\(B_\varphi y)\cdot\xi\,=\,0}}\nabla^{2j}\frac1Q(\xi)\left(\frac{B_\varphi  y}{|y|}^{\otimes{2j-1}}, e_3\right)\,d\mathcal{H}^{n-2}(\xi)\\
		&\stackrel{\xi=B_\varphi\xi'}{=}\int_{\substack{|\xi'|=1\\\xi_1'\,=\,0}}\nabla^{2j}\frac1Q(B_\varphi\xi)\left(\begin{pmatrix}
		\cos\varphi\\0\\\sin\varphi\\0\\\cdots\\0
		\end{pmatrix}^{\otimes{2j-1}}, e_3\right)\,d\mathcal{H}^{n-2}(\xi')\\
		&=\int_{\substack{|\xi'|=1\\\xi_1'\,=\,0}}\;\sum_{h_1,\cdots,h_{2j-1}=1}^n\left(\nabla^{2j}\frac1Q(B_\varphi\xi)\right)_{3,h_1,\cdots,h_{2j-1}}\begin{pmatrix}
		\cos\varphi\\0\\\sin\varphi\\0\\\cdots\\0
		\end{pmatrix}_{h_1}\cdots\begin{pmatrix}
		\cos\varphi\\0\\\sin\varphi\\0\\\cdots\\0
		\end{pmatrix}_{h_{2j-1}}d\mathcal{H}^{n-2}(\xi')\\
		&=\int_{\substack{|\xi'|=1\\\xi_1'\,=\,0}}\;\sum_{h=0}^{2j-1}\binom{2j-1}{h}\,\partial_{1^{2j-1-h}3^{h+1}}^{2j}\frac1Q(B_\varphi\xi')\,\cos^{2j-1-h}\varphi\,\,\sin^h\varphi\,\,d\mathcal{H}^{n-2}(\xi').
		\end{split}
		\end{equation*}
		Therefore we obtain
		\begin{equation*}
		\begin{split}
		\frac d{d\varphi}\bigg|_{\varphi=0}\tilde H_0(\varphi)&=\int_{\substack{|\xi'|=1\\\xi_1'\,=\,0}}\;\sum_{h=0}^{2j-1}\binom{2j-1}{h}\bigg\{\frac d{d\varphi}\bigg|_{\varphi=0}\left(\partial_{1^{2j-1-h}3^{h+1}}^{2j}\frac1Q(B_\varphi\xi')\right)\left(\cos^{2j-1-h}\varphi\bigg|_{\varphi=0}\right)\left(\sin^h\varphi\bigg|_{\varphi=0}\right)\\
		&\quad +\,\partial_{1^{2j-1-h}3^{h+1}}^{2j}\frac1Q(\xi')\bigg[\left(\sin^h\varphi\bigg|_{\varphi=0}\right)\left(\frac d{d\varphi}\bigg|_{\varphi=0}\cos^{2j-1-h}\varphi\right)\\
		&\quad +\left(\cos^{2j-1-h}\varphi\bigg|_{\varphi=0}\right)\left(\frac d{d\varphi}\bigg|_{\varphi=0}\sin^h\varphi\right)\bigg]\bigg\}\,d\mathcal{H}^{n-2}(\xi').
		\end{split}
		\end{equation*}
		As in Step 1, the terms which remain are the first with $h=0$ and the third with $h=1$, so
		\begin{equation*}
		\begin{split}
		\frac d{d\varphi}\bigg|_{\varphi=0}\tilde H_0(\varphi)&=\int_{\substack{|\xi'|=1\\\xi_1'\,=\,0}}\;\left[\frac d{d\varphi}\bigg|_{\varphi=0}\left(\partial_{1^{2j-1}3}^{2j}\frac1Q(B_\varphi\xi')\right)+(2j-1)\partial_{1^{2j-2}3^2}^{2j}\frac1Q(\xi')\right]d\mathcal{H}^{n-2}(\xi').
		\end{split}
		\end{equation*}
		Differentiating the first term as in \eqref{Step1J:1_1} on $\{\xi'_1=0\}$,
		\begin{equation*}
		\begin{split}
		\frac d{d\varphi}\bigg|_{\varphi=0}\left(\partial_{1^{2j-1}3}^{2j}\frac1Q(B_\varphi\xi')\right)&=\sum_{h=1}^n\,\,\partial_{1^{2j-1}3h}^{2j+1}\frac1Q(\xi')\,\left(\frac{\partial(B_\varphi\xi')_h}{\partial\varphi}\right)\bigg|_{\varphi=0}=\,-\xi_3'\,\partial_{1^{2j}3}^{2j+1}\frac1Q(\xi'),
		\end{split}
		\end{equation*}
		we obtain hence
		\begin{equation}\label{Step3:dH_0}
		\begin{split}
		\frac d{d\varphi}\bigg|_{\varphi=0}\tilde H_0(\varphi)&=\int_{\substack{|\xi'|=1\\\xi_1'\,=\,0}}\left[(2j-1)\,\partial_{1^{2j-2}3^2}^{2j}\frac1Q(\xi')-\xi_3'\,\partial_{1^{2j}3}^{2j+1}\frac1Q(\xi')\right]d\mathcal{H}^{n-2}(\xi').
		\end{split}
		\end{equation}
		Let us now address to the second term in \eqref{Step3:H0H1H2}:
		\begin{equation*}
		\begin{split}
		\tilde H_1(\varphi)&=\int_{\substack{|\xi|=1\\(B_\varphi y)\xi\,=\,0}}\xi_3\,\nabla^{2j+1}\frac1Q(\xi)\left(\frac{B_\varphi y}{|y|}^{\otimes\,2j+1}\right)d\mathcal{H}^{n-2}(\xi)\\
		&\stackrel{\xi=B_\varphi\xi'}{=}\int_{\substack{|\xi'|=1\\\xi_1'\,=\,0}}(B_\varphi\xi')_3\,\nabla^{2j+1}\frac1Q(B_\varphi\xi')\begin{pmatrix}
		\cos\varphi\\0\\\sin\varphi\\0\\\cdots\\0
		\end{pmatrix}^{\otimes\,2j+1}d\mathcal{H}^{n-2}(\xi')\\
		&=\int_{\substack{|\xi'|=1\\\xi_1'\,=\,0}}(\xi_1'\sin\varphi+\xi_3'\cos\varphi)\sum_{h_1,\cdots,h_{2j+1}=1}^n\partial_{h_1\cdots h_{2j+1}}^{2j+1}\frac1Q(B_\varphi\xi')\begin{pmatrix}
		\cos\varphi\\0\\\sin\varphi\\0\\\cdots\\0
		\end{pmatrix}_{h_1}\cdots\begin{pmatrix}
		\cos\varphi\\0\\\sin\varphi\\0\\\cdots\\0
		\end{pmatrix}_{h_{2j+1}}d\mathcal{H}^{n-2}(\xi')\\
		&=\int_{\substack{|\xi'|=1\\\xi_1'\,=\,0}}(\xi_1'\sin\varphi+\xi_3'\cos\varphi)\left(\sum_{h=0}^{2j+1}\binom{2j+1}{h}\partial_{1^{2j+1-h}3^h}^{2j+1}\frac1Q(B_\varphi\xi')\cos^{2j+1-h}\varphi\,\sin^h\varphi\right)d\mathcal{H}^{n-2}(\xi').
		\end{split}
		\end{equation*}
		Hence, differentiating with respect to $\varphi$, with similar computations as for $\tilde H_0$, we obtain:
		\begin{equation*}
		\begin{split}
		\frac d{d\varphi}\bigg|_{\varphi=0}\tilde H_1(\varphi)
		&=\int_{\substack{|\xi'|=1\\\xi_1'\,=\,0}}(\xi'_1\cos\varphi-\xi'_3\sin\varphi)\bigg|_{\varphi=0}\left(\sum_{h=0}^{2j+1}\binom{2j+1}{h}\partial^{2j+1}_{1^{2j+1-h}3^h}\frac1Q(B_\varphi\xi')\cos^{2j+1-h}\varphi\,\sin^h\varphi\right)\bigg|_{\varphi=0}\\
		&\quad +\,\,\xi'_3\,\sum_{h=0}^{2j+1}\binom{2j+1}{h}\frac d{d\varphi}\bigg|_{\varphi=0}\left(\partial^{2j+1}_{1^{2j+1-h}3^h}\frac1Q(B_\varphi\xi')\right)\left(\cos^{2j+1-h}\varphi\bigg|_{\varphi=0}\right)\left(\sin^h\varphi\bigg|_{\varphi=0}\right)\\
		&\quad +\,\,\xi'_3\,\sum_{h=0}^{2j+1}\binom{2j+1}{h}\partial^{2j+1}_{1^{2j+1-h}3^h}\frac1Q(\xi')\bigg[\left(\sin^h\varphi\bigg|_{\varphi=0}\right)\left(\frac d{d\varphi}\bigg|_{\varphi=0}\cos^{2j+1-h}\varphi\right)\\
		& \quad +\left(\cos^{2j+1-h}\varphi\bigg|_{\varphi=0}\right)\left(\frac d{d\varphi}\bigg|_{\varphi=0}\sin^h\varphi\right)\bigg]\\
		&=0+\int_{\substack{|\xi'|=1\\\xi_1'\,=\,0}}\,\xi'_3\, \left[\frac d{d\varphi}\bigg|_{\varphi=0} \partial_{1^{2j+1}}^{2j+1}\frac1Q(B_\varphi\xi')+(2j+1)\partial_{1^{2j}3}^{2j+1}\frac1Q(\xi')\right]d\mathcal{H}^{n-2}(\xi').
		\end{split}
		\end{equation*}
		With similar computations as in \eqref{Step1J:1_1}, we infer
		\begin{equation}\label{Step3:dH_1}
		\frac d{d\varphi}\bigg|_{\varphi=0}\tilde H_1(\varphi)=\int_{\substack{|\xi'|=1\\\xi_1'\,=\,0}}\,\left[(2j+1)\,\xi'_3\,\partial_{1^{2j}3}^{2j+1}\frac1Q(\xi')-(\xi'_3)^2\,\partial_{1^{2j+2}}^{2j+2}\frac1Q(\xi')\right]d\mathcal{H}^{n-2}(\xi').
		\end{equation}
		Finally, we have to consider the third term in \eqref{Step3:H0H1H2}:
		\begin{equation*}
		\begin{split}
		\tilde H_2(\varphi)&=\int_{\substack{|\xi|=1\\(B_\varphi y)\xi\,=\,0}}\,\nabla^{2j}\frac1Q(\xi)\left(\frac{B_\varphi y}{|y|}^{\otimes 2j}\right)\left(\frac{B_\varphi y}{|y|}\right)_3d\mathcal{H}^{n-2}(\xi)\\
		&\stackrel{\xi=B_\varphi\xi'}{=}\int_{\substack{|\xi'|=1\\\xi_1'\,=\,0}}\,\,\sum_{h_1,\cdots,h_{2j}=1}^{n}\partial^{2j}_{h_1\cdots h_{2j}}\frac1Q(B_\varphi\xi')\begin{pmatrix}
		\cos\varphi\\0\\\sin\varphi\\0\\\cdots\\0
		\end{pmatrix}_{h_1}\cdots\,\,\begin{pmatrix}
		\cos\varphi\\0\\\sin\varphi\\0\\\cdots\\0
		\end{pmatrix}_{h_{2j}}\sin\varphi\;d\mathcal{H}^{n-2}(\xi')\\
		&=\int_{\substack{|\xi'|=1\\\xi_1'\,=\,0}}\,\,\sum_{h=0}^{2j}\binom{2j}{h}\partial_{1^{2j-h}3^h}^{2j}\frac1Q(B_\varphi\xi')\cos^{2j-h}\varphi\,\sin^{h+1}\varphi\,\,d\mathcal{H}^{n-2}(\xi').
		\end{split}
		\end{equation*}
		Hence,
		\begin{equation*}
		\begin{split}
		\frac d{d\varphi}\bigg|_{\varphi=0}\tilde H_2(\varphi)&=\int_{\substack{|\xi'|=1\\\xi_1'\,=\,0}}\,\,\sum_{h=0}^{2j}\binom{2j}{h}\bigg[\frac d{d\varphi}\bigg|_{\varphi=0}\left(\partial_{1^{2j-h}3^h}^{2j}\frac1Q(B_\varphi\xi')\right)\left(\cos^{2j-h}\varphi\bigg|_{\varphi=0}\right)\left(\sin^{h+1}\varphi\bigg|_{\varphi=0}\right)\\
		&\quad +\partial_{1^{2j-h}3^h}^{2j}\frac1Q(\xi')\bigg(\left(\sin^{h+1}\varphi\bigg|_{\varphi=0}\right)\left(\frac d{d\varphi}\bigg|_{\varphi=0}\cos^{2j-h}\varphi\right)\\
		&\quad +\left( \cos^{2j-h}\varphi\bigg|_{\varphi=0}\right)\left(\frac d{d\varphi}\bigg|_{\varphi=0}\sin^{h+1}\varphi\right)\bigg)\bigg]d\mathcal{H}^{n-2}(\xi').
		\end{split}
		\end{equation*}
		The first two terms vanish for any choice of $h$, while the last one remains only for $h=0$, so
		\begin{equation}\label{Step3:dH_2}
		\begin{split}
		\frac d{d\varphi}\bigg|_{\varphi=0}\tilde H_2(\varphi)&=\int_{\substack{|\xi'|=1\\\xi_1'\,=\,0}}\partial_{1^{2j}}^{2j}\frac1Q(\xi')\,d\mathcal{H}^{n-2}(\xi').
		\end{split}
		\end{equation}
		Hence, recalling the splitting \eqref{Step3:H0H1H2}, by \eqref{Step3:dH_0}-\eqref{Step3:dH_2} we obtain (omitting from now on in each integral its differential $d\mathcal{H}^{n-2}(\xi')$):
		\begin{equation*}
		\begin{split}
		\frac{\partial^2 J}{\partial y_3^2}\bigg|_{(y_1,0,\cdots,0)}&=\frac1{|y|}\frac d{d\varphi}\bigg|_{\varphi=0}\tilde H(\varphi)\\
		&=\frac1{|y|^{2k+1}}\left[2j\,\frac d{d\varphi}\bigg|_{\varphi=0}\tilde H_0(\varphi)-\frac d{d\varphi}\bigg|_{\varphi=0}\tilde H_1(\varphi)-(2k-1+2j)\,\frac d{d\varphi}\bigg|_{\varphi=0}\tilde H_2(\varphi)\right]\\
		&=\frac1{|y|^{2k+1}}\bigg[-(2k-1+2j)\int_{\substack{|\xi'|=1\\\xi_1'\,=\,0}}\partial_{1^{2j}}^{2j}\frac1Q(\xi')\,+\,2j(2j-1)\int_{\substack{|\xi'|=1\\\xi_1'\,=\,0}}\partial_{1^{2j-2}3^2}^{2j}\frac1Q(\xi')\\
		&\quad -(4j+1)\int_{\substack{|\xi'|=1\\\xi_1'\,=\,0}}\,\xi'_3\,\,\partial_{1^{2j}3}^{2j+1}\frac1Q(\xi')\,+\,\int_{\substack{|\xi'|=1\\\xi_1'\,=\,0}}\,(\xi'_3)^2\,\,\partial_{1^{2j+2}}^{2j+2}\frac1Q(\xi')\bigg].
		\end{split}
		\end{equation*}
		Therefore, the same being valid for any variable $y_h$ with $h\in\{2,\cdots,n\}$, and recalling \eqref{d2J_d11} if $h=1$, we may compute the Laplacian of $J$:
		\begin{equation}\label{Step3:H0+H1+H2}
		\begin{split}
		\Delta J\bigg|_{(y_1,0,\cdots,0)}&=\frac{\partial^2 J}{\partial y_1^2}\bigg|_{(y_1,0,\cdots,0)}+\sum_{h=2}^n\frac{\partial^2 J}{\partial y_h^2}\bigg|_{(y_1,0,\cdots,0)}\\
		&=\frac1{|y|^{2k+1}}\bigg\{\,2k(2k-1)\int_{\substack{|\xi'|=1\\\xi'_1\,=\,0}}\partial_{1^{2j}}^{2j}\frac1Q(\xi')+(n-1)(1-2k-2j)\int_{\substack{|\xi'|=1\\\xi'_1\,=\,0}}\partial_{1^{2j}}^{2j}\frac1Q(\xi')\\
		&\quad +\,2j(2j-1)\int_{\substack{|\xi'|=1\\\xi'_1\,=\,0}}\sum_{h=2}^n\,\partial_{1^{2j-2}h^2}^{2j}\frac1Q(\xi')-(4j+1)\int_{\substack{|\xi'|=1\\\xi'_1\,=\,0}}\sum_{h=1}^n\,\xi'_h\,\partial_h\left(\partial_{1^{2j}}^{2j}\frac1Q\right)(\xi')\\
		&\quad +\int_{\substack{|\xi'|=1\\\xi'_1\,=\,0}}\bigg(\underbrace{\sum_{h=1}^n(\xi'_h)^2}_{|\xi'|^2=1}\bigg)\partial_{1^{2j+2}}^{2j+2}\frac1Q(\xi')\,\bigg\}\\
		&=\frac1{|y|^{2k+1}}\bigg\{\,\big[2k(2k-1)+(n-1)(1-2k-2j)]\int_{\substack{|\xi'|=1\\\xi'_1\,=\,0}}\partial_{1^{2j}}^{2j}\frac1Q(\xi') \\
		&\quad +2j(2j-1)\int_{\substack{|\xi'|=1\\\xi'_1\,=\,0}}\Delta'\left(\partial_{1^{2j-2}}^{2j-2}\frac1Q\right)(\xi')
		-(4j+1)\int_{\substack{|\xi'|=1\\\xi'_1\,=\,0}}\partial_\nu\left(\partial_{1^{2j}}^{2j}\frac1Q\right)(\xi')\\
		&\quad +\int_{\substack{|\xi'|=1\\\xi'_1\,=\,0}}\partial_{1^{2j+2}}^{2j+2}\frac1Q(\xi')\,\bigg\}.
		\end{split}
		\end{equation}
		Here, we denote
		$$
		\Delta'=\partial^2_{22}+\ldots + \partial^2_{nn}.
		$$
		By homogeneity of the symbol, one has (see Lemma \ref{Homog_der} below)
		\begin{equation}\label{normal_derivatives}
		\partial_\nu\left(\partial_{1^{2j}}^{2j}\frac1Q\right)=-2(m+j)\left(\partial_{1^{2j}}^{2j}\frac1Q\right).
		\end{equation}
		Moreover, in order to handle the term $\Delta\partial_{1^{2j-2}}^{2j-2}\frac1Q$, we may apply the following well-known identity
		\begin{equation}\label{Laplace_on_manifolds}
		\Delta' u=\Delta_S u+H_S\partial_\nu u+\partial^2_{\nu\nu}u
		\end{equation}
		with $u=\partial_{1^{2j-2}}^{2j-2}\frac1Q$ and $S:=\{|\xi'|=1\,,\,\xi'_1=0\}$ the manifold on which we are integrating, and
		where $H_S$ stands for the mean curvature of $S$, so we have $H_S=(n-2)$. Noticing that the normal derivatives
		in \eqref{Laplace_on_manifolds} may be handled as in \eqref{normal_derivatives},
		it remains the term with the tangential part of the Laplacian. However, it vanishes when integrated on $S$. Hence,
		\begin{equation}\label{Laplacian}
		\begin{split}
		\int_{\substack{|\xi'|=1\\\xi'_1\,=\,0}}&\Delta' \left(\partial_{1^{2j-2}}^{2j-2}\frac1Q\right)(\xi')
		=(n-2)\int_{\substack{|\xi'|=1\\\xi'_1\,=\,0}}\partial_\nu\left(\partial_{1^{2j-2}}^{2j-2}\frac1Q\right)+\int_{\substack{|\xi'|=1\\\xi'_1\,=\,0}}\partial_{\nu\nu}^2\left(\partial_{1^{2j-2}}^{2j-2}\frac1Q\right)\\
		&\stackrel{\eqref{normal_derivatives}}{=}\big[(n-2)\big(-2(m+j-1)\big)+2(m+j-1)\big( 2m+2j-1\big)\big]
		\int_{\substack{|\xi'|=1\\\xi'_1\,=\,0}}\partial_{1^{2j-2}}^{2j-2}\frac1Q(\xi')\\
		&=2(m+j-1)(2m+2j+1-n)\int_{\substack{|\xi'|=1\\\xi'_1\,=\,0}}\partial_{1^{2j-2}}^{2j-2}\frac1Q(\xi')
		\end{split}
		\end{equation}
		Inserting \eqref{normal_derivatives} and \eqref{Laplacian} in \eqref{Step3:H0+H1+H2} and summing the constants, we finally end up with \eqref{J_iteration_1} and thus with our formula \eqref{J_iteration}.
	\end{proof}

	\begin{lem}\label{Homog_der}
		Let $Q(\cdot)$ be positive and $p$-homogeneous. Then, one has for any multi-index $\alpha\in\mathbb{N}_0^n$ and any $x\in\mathbb{R}^n\setminus \{0\}$:
		\begin{equation*}
		x\cdot \nabla \left(D^\alpha\frac1Q\right)(x)=-(p+|\alpha|) \left(D^\alpha\frac1Q\right)(x).
		\end{equation*}
	\end{lem}
	\begin{proof}
		By assumption we have for $r>0$ that
		$$
		\frac1Q (rx)=r^{-p} \frac1Q (x).
		$$
		Differentiation with respect to $x$ yields:
		$$
		r^{|\alpha|} \left(D^\alpha\frac1Q\right)(rx)=r^{-p}  \left(D^\alpha\frac1Q\right)(x)
		\quad \Rightarrow\quad\left(D^\alpha\frac1Q\right)(rx)=r^{-p - |\alpha|}  \left(D^\alpha\frac1Q\right)(x).
		$$
		Differentiating now with respect to $r$ gives:
		$$
		x\cdot \nabla \left(D^\alpha\frac1Q\right)(rx)=(-p - |\alpha|) r^{-p - |\alpha|-1}  \left(D^\alpha\frac1Q\right)(x).
		$$
		The claim follows by putting $r=1$.
	\end{proof}

	\vskip0.1truecm

	\begin{proof}[Proof of Theorem \ref{conjecture_j}]
		Notice that for $k=1$, $\Delta F$ is already in the form \eqref{k_iteration} by Proposition \ref{Prop_1st_iteration}. We thus proceed by induction and let us suppose that $\Delta^kF$ has the form \eqref{k_iteration} for some $k\in\N$, $k>1$, with coefficients as in \eqref{d_kj}-\eqref{c_kj}, namely
		\begin{equation*}
		\Delta^kF=\sum_{j=1}^kd_{k,j}(n,m)J_{k,k-j}.
		\end{equation*}
		Applying the recursive formula \eqref{J_iteration}, we thus have
		\begin{equation*}
		\Delta^{k+1}F=\sum_{j=1}^{k+1} d_{k+1,j}(n,m) J_{k+1,k+1-j}
		\end{equation*}
		with
		\begin{equation*}
		d_{k+1,j}(n,m)=d_{k,j}(n,m)+d_{k,j-1}(n,m)c_0(n,m,k,k-j+1)+d_{k,j-2}(n,m)c_\Delta(n,m,k-j+2),
		\end{equation*}
		for $j=3,\dots,k+1$. We have $d_{k,k+1}$ according to \eqref{c_kj} and put $d_{k,0}:=d_{k,-1}:=0$. Hence, for $j=1$ and $j=2$ we have the recurrence relations
		$$d_{k+1,2}=d_{k,2}+d_{k,1}\,c_0(n,m,k,k-1),$$
		$$d_{k+1,1}=d_{k,1},$$
		and for those the formulae \eqref{d_kj} and \eqref{c_kj} are easily checked.

		 We show the shape of $d_{k+1,j}(n,m)$ for $j=3,\ldots,k+1$, the
		cases $j=1,2$ being analogous but simpler.
		In order to prove that $d_{k+1,j}(n,m)$
		has the shape as in \eqref{k_iteration} with $k+1$, which will prove Theorem \ref{conjecture_j},
		we show that the following term is equal to $0$:
		\begin{align*}
			\lefteqn{ d_{k,j}(n,m)+d_{k,j-1}(n,m)c_0(n,m,k,k-j+1)+d_{k,j-2}(n,m)c_\Delta(n,m,k-j+2)}\\
			&- (-1)^{j-1} c_{k+1,j}\prod_{\ell=2}^j (n-2m-2k+2\ell-5)\\
			=&(-1)^{j-1}\prod_{\ell=2}^{j-2} (n-2m-2k+2\ell-3)\\
			&\cdot \Bigg\{
			c_{k,j}(n-2m-2k+2j-5)(n-2m-2k+2j-3)\\
			&\quad-c_{k,j-1}(n-2m-2k+2j-5)\\
			&\qquad \cdot\Big(2k(2k-1)+(n-1)(2j-4k-1)+2(4k-4j+5)(m+k+1-j)\Big)\\
			&\quad+4c_{k,j-2}(k+2-j)(2k+3-2j)(m+k+1-j)(-n+2m+2k-2j+5)\\
			&\quad- c_{k+1,j}(n-2m-2k-1)(n-2m-2k+2j-5)
			\Bigg\}\\
			=&(-1)^{j-1}\prod_{\ell=2}^{j-1} (n-2m-2k+2\ell-3)\\
			&\cdot \Bigg\{
			c_{k,j}(n-2m-2k+2j-3)\\
			&\quad-c_{k,j-1}
			\cdot\Big(2k(2k-1)+(n-1)(2j-4k-1)+2(4k-4j+5)(m+k+1-j)\Big)\\
			&\quad-4c_{k,j-2}(k+2-j)(2k+3-2j)(m+k+1-j)\\
			&\quad- c_{k+1,j}(n-2m-2k-1)
			\Bigg\}\\
			=&(-1)^{j-1}\left( \prod_{\ell=2}^{j-1} (n-2m-2k+2\ell-3)\right)\left( \prod_{\ell=2}^{j}\frac{1}{\ell-1}\right)
			\left( \prod_{\ell=2}^{j-2}(k-\ell+1)(2k-2\ell+1)  \right)\\
			&\cdot \Bigg\{
			(k-j+2)(2k-2j+3)(k-j+1)(2k-2j+1)(n-2m-2k+2j-3)\\
			&\quad-(j-1)(k-j+2)(2k-2j+3)\\
			&\qquad\cdot \Big(2k(2k-1)+(n-1)(2j-4k-1)+2(4k-4j+5)(m+k+1-j)\Big)\\
			&\quad-4(j-1)(j-2)(k+2-j)(2k+3-2j)(m+k+1-j)\\
			&\quad- k(2k-1)(k-j+2)(2k-2j+3)(n-2m-2k-1)
			\Bigg\}\\
			\end{align*}
			\begin{align*}
			=&(-1)^{j-1}\left( \prod_{\ell=2}^{j-1} (n-2m-2k+2\ell-3)\right)\left( \prod_{\ell=2}^{j}\frac{1}{\ell-1}\right)
			\left( \prod_{\ell=2}^{j-1}(k-\ell+1)(2k-2\ell+1)  \right)\\
			&\cdot \Bigg\{
			(k-j+1)(2k-2j+1)(n-2m-2k+2j-3)\\
			&\quad-(j-1)
			\cdot \Big(2k(2k-1)+(n-1)(2j-4k-1)+2(4k-4j+5)(m+k+1-j)\Big)\\
			&\quad-4(j-1)(j-2)(m+k+1-j)
			- k(2k-1)(n-2m-2k-1)
			\Bigg\}\\
			=&(-1)^{j-1}\left( \prod_{\ell=2}^{j-1} (n-2m-2k+2\ell-3)\right)\left( \prod_{\ell=2}^{j}\frac{1}{\ell-1}\right)
			\left( \prod_{\ell=2}^{j-1}(k-\ell+1)(2k-2\ell+1)  \right)\\
			&\cdot \Bigg\{
			(k-j+1)(2k-2j+1)(n-2m-2k+2j-3)\\
			&\quad -k(2k-1) (n-2m-2k+2j-3) -(n-1)(2j-4k-1)(j-1)\\
			&\quad -2(j-1)(m+k-j+1)(4k-2j+1)
			\Bigg\}\\
			=&(-1)^{j-1}\left( \prod_{\ell=2}^{j} (n-2m-2k+2\ell-3)\right)\left( \prod_{\ell=2}^{j}\frac{1}{\ell-1}\right)
			\left( \prod_{\ell=2}^{j-1}(k-\ell+1)(2k-2\ell+1)  \right)\\
			&\cdot \Bigg\{(k-j+1)(2k-2j+1)-k(2k-1)+(j-1)(4k-2j+1)
			\Bigg\}
			=0.
		\end{align*}
		This concludes the proof of Theorem \ref{conjecture_j}.
	\end{proof}

	\subsection{The even dimension $\boldsymbol{n=2m+2}$: explicit fundamental solutions}
	\label{Section_EVEN_proof}
	     Here we provide the proof of Theorem \ref{TheoremEVEN}.
	     The starting point is John's formula \eqref{formulaJohnEVEN}, according to which we have
	\begin{equation}\label{formulaJohnEVEN_2}
		\begin{split}
		K_L(0,y) &=- \frac1{(2\pi)^n} (-\Delta_y)^{(n-2m)/2}\int_{|\xi|=1} \frac{\log |y\cdot \xi|}{Q(\xi)} \,d\mathcal{H}^{n-1}(\xi)\\
		&=- \frac2{(2\pi)^n} (-\Delta_y)^{(n-2m)/2}\int_{|\xi|=1 \atop y\cdot \xi>0} \frac{\log (y\cdot \xi)}{Q(\xi)} \,d\mathcal{H}^{n-1}(\xi).
		\end{split}
	\end{equation}
	For $n=2m$ we immediately see that the fundamental solution of $L$ satisfies
	$$
	K_L(0,y)\to +\infty \mbox{\ as\ } y\to 0,
	$$
	which is similar to the case $n=2m+1$ above.

	\bigskip\noindent
	Hence, in order to obtain an explicit expression for $K_L$, we thus have to compute the Laplacian of
	\begin{equation}\label{G}
	G(y):=\int_{|\xi|=1 \atop y\cdot \xi>0} \frac{\log (y\cdot \xi)}{Q(\xi)} d\mathcal{H}^{n-1}(\xi).
	\end{equation}
	Due to the logarithmic term, the calculations for even dimensions cannot be simplified similarly to the previous section. An application of Stokes' theorem would change $\log(\xi_1)$ into $\frac1{\xi_1}$, a non-integrable singularity. For this reason we restrict ourselves to the case $n=2m+2$.
	
		The strategy is similar to the one applied in the proofs of Propositions \ref{Prop_1st_iteration} and \ref{Prop0.3}.

		\noindent\textbf{Step 1.} Let $y=(y_1,0,\cdots,0)$ with $y_1>0$. Splitting $G$ as
		\begin{equation*}
		\begin{split}
		G(y)&=\int_{|\xi|=1 \atop y\cdot \xi>0}\frac{\log(y_1\xi_1)}{Q(\xi)}d\mathcal H^{n-1}(\xi)\\
		&=\log(y_1)\int_{|\xi|=1 \atop y\cdot \xi>0}\frac1{Q(\xi)}d\mathcal H^{n-1}(\xi)+\int_{|\xi|=1 \atop y\cdot \xi>0}\frac{\log(\xi_1)}{Q(\xi)}d\mathcal H^{n-1}(\xi),
		\end{split}
		\end{equation*}
		we first infer
		\begin{equation}\label{dG_1}
		\frac{\partial G}{\partial y_1}(y)=\frac1{y_1}\int_{|\xi|=1 \atop \xi_1>0}\frac1{Q(\xi)}d\mathcal H^{n-1}(\xi)=\frac1{2y_1}\int_{|\xi|=1}\frac1{Q(\xi)}d\mathcal H^{n-1}(\xi)
		\end{equation}
		as $Q(\xi)=Q(-\xi)$ by homogeneity of the symbol. Moreover,
		\begin{equation}\label{dG_11}
		\frac{\partial^2 G}{\partial y_1^2}(y)=-\frac1{2y_1^2}\int_{|\xi|=1}\frac1{Q(\xi)}d\mathcal H^{n-1}(\xi).
		\end{equation}
		In order to compute the other first derivatives of $G$, let us split it as
		\begin{equation*}
		\begin{split}
		G(y)&=\log|y|\int_{|\xi|=1 \atop y\cdot \xi>0}\frac1{Q(\xi)}d\mathcal H^{n-1}(\xi)+\int_{|\xi|=1 \atop y\cdot \xi>0}\frac{\log\big(\tfrac y{|y|}\cdot\xi\big)}{Q(\xi)}d\mathcal H^{n-1}(\xi)\\
		&=:\log|y|\,\tilde G_1(y)+\tilde G_2(y).
		\end{split}
		\end{equation*}
		Hence,
		\begin{equation}\label{G=G_1,2}
		\frac{\partial G}{\partial y_3}(y)=\frac{y_3}{|y|^2}\tilde G_1(y)+\log|y|\,\frac{\partial\tilde G_1}{\partial y_3}(y)+\frac{\partial\tilde G_2}{\partial y_3}(y).
		\end{equation}
		We use the matrix $B_\varphi$ defined in \eqref{B_phi} and define
		$$\tilde H_1(\varphi)=\tilde G_1(B_\varphi y)\qquad\mbox{and}\qquad\tilde H_2(\varphi)=\tilde G_2(B_\varphi y).$$
		Concerning the first term, exploiting  $Q(\xi)=Q(-\xi)$, we
		find for any $y$ 
		\begin{equation}\label{tildeG1}
		\tilde G_1(y)=\frac12\int_{|\xi|=1}\frac1{Q(\xi)}\,d\mathcal H^{n-1}(\xi)\qquad\Rightarrow\qquad\frac{\partial\tilde G_1}{\partial y_3}(y)=0.
		\end{equation}
		Concerning the second term, reasoning as in \eqref{dF_d3-dH_dphi}, we get
		\begin{equation}\label{tildeG2}
		\begin{split}
		y_1\frac{\partial\tilde G_2}{\partial y_3}(y)&=\frac{\partial \tilde H_2}{\partial \varphi}(\varphi)\bigg|_{\varphi=0}=\frac{\partial}{\partial \varphi}\bigg|_{\varphi=0}\int_{|\xi|=1 \atop \xi_1>0}\frac{\log\big(\tfrac y{|y|}\cdot\xi\big)}{Q(B_\varphi\xi)}d\mathcal H^{n-1}(\xi)\\
		&=\int_{|\xi|=1 \atop \xi_1>0}\log(\xi_1)\bigg(-\xi_3\partial_1\frac1Q(\xi)+\xi_1\partial_3\frac1Q(\xi)\bigg)d\mathcal H^{n-1}(\xi).
		\end{split}
		\end{equation}
		Therefore, by \eqref{G=G_1,2}-\eqref{tildeG2} we conclude
		\begin{equation*}
		\frac{\partial G}{\partial y_3}(y)=\frac{y_3}{|y|^2}\int_{|\xi|=1 \atop y\cdot \xi>0}\frac1{Q(\xi)}d\mathcal H^{n-1}(\xi)+\frac1{|y|}\int_{|\xi|=1 \atop \xi_1>0}\log(\xi_1)\bigg(-\xi_3\partial_1\frac1Q(\xi)+\xi_1\partial_3\frac1Q(\xi)\bigg)d\mathcal H^{n-1}(\xi)
		\end{equation*}
		and, analogously, for any $k\in\{2,\cdots, n\}$, one has
		\begin{equation*}
		\begin{split}
		\frac{\partial G}{\partial y_k}(y_1,0\cdots,0)&=\frac{y_k}{2|y|^2}\int_{|\xi|=1}\frac1{Q(\xi)}d\mathcal H^{n-1}(\xi)\\
		&+\frac1{|y|}\int_{|\xi|=1 \atop \xi_1>0}\log\bigg(\xi\cdot\frac y{|y|}\bigg)\bigg(-\xi_k\bigg(\nabla\frac1Q(\xi)\cdot\frac y{|y|}\bigg)+\bigg(\xi\cdot\frac y{|y|}\bigg)\partial_k\frac1Q(\xi)\bigg)d\mathcal H^{n-1}(\xi).
		\end{split}
		\end{equation*}
		Because this formula  is consistent with \eqref{dG_1}, we can write it in a more compact way as
		\begin{equation}\label{nablaG}
		\begin{split}
		\nabla G(y)&=\frac{y^T}{2|y|^2}\int_{|\xi|=1}\frac1{Q(\xi)}d\mathcal H^{n-1}(\xi)\\
		&\quad +\frac1{|y|}\int_{|\xi|=1 \atop y\cdot\xi>0}\log\bigg(\xi\cdot\frac y{|y|}\bigg)\bigg(-\xi^T\bigg(\nabla\frac1Q(\xi)\cdot\frac y{|y|}\bigg)+\bigg(\xi\cdot\frac y{|y|}\bigg)\nabla\frac1Q(\xi)\bigg)d\mathcal H^{n-1}(\xi).
		\end{split}
		\end{equation}
		\vskip0.2truecm

		\noindent \textbf{Step 2.} Let now $y\in\R^n\setminus\{0\}$ so that $y=|y|b_1$ with $|b_1|=1$, and let $B=\big(b_1\,|\,\cdots\,|\,b_n\big)\in SO(n)$.
		Defining
		$$\tilde G(y):=(G\circ B)(y)=\int_{|\xi|=1 \atop y\cdot\xi>0}\frac{\log(y\cdot\xi)}{Q(B(\xi))}d\mathcal H^{n-1}(\xi),$$
		we have
		$$\nabla G(y)=\nabla G\big(B(|y|,0,\cdots,0)^T\big)=\nabla\tilde G\big((|y|,0,\cdots,0)^T\big)B^T.$$
		Therefore,
		\begin{equation*}
		\begin{split}
		\nabla\tilde G\big((|y|,0,\cdots,0)^T\big)&=\frac1{2|y|}e_1^T\int_{|\xi|=1}\frac1{Q(B(\xi))}d\mathcal H^{n-1}(\xi)\\
		&\quad +\frac1{|y|}\int_{|\xi|=1 \atop y\cdot\xi>0}\log(\xi_1)\bigg(-\xi^T\big(\nabla\frac1Q(B(\xi))\cdot B e_1\big)+\xi_1\,\nabla\frac1Q(B(\xi))\cdot B\bigg)d\mathcal H^{n-1}(\xi)\\
		&=\frac1{2|y|}e_1^T\int_{|\xi|=1}\frac1{Q(\xi)}d\mathcal H^{n-1}(\xi)\\
		&\quad +\frac1{|y|}\int_{|\xi|=1 \atop (B^T\xi)_1>0}\log((B^T\xi)_1)\bigg(-\xi^TB\big(\nabla\frac1Q(\xi)\cdot \frac y{|y|}\big)+(B^T\xi)_1\,\nabla\frac1Q(\xi)\cdot B\bigg)d\mathcal H^{n-1}(\xi).
		\end{split}
		\end{equation*}
		Observing that $(B^T\xi)_1=\frac{y\cdot\xi}{|y|}$ and multiplying by $B^T$ shows that \eqref{nablaG} holds also for any $y\in\R^n\setminus\{0\}$.
		\vskip0.2truecm
		\noindent \textbf{Step 3.} We  consider again $y=(y_1,0,\cdots,0)$ with $y_1>0$ and compute $\Delta G(y)$. Similarly as in Step 1, we get
		\begin{equation*}
		\begin{split}
		\partial_3^2\,G(y_1,0,\cdots,0)&=\frac1{2|y|^2}\int_{|\xi|=1}\frac1{Q(\xi)}d\mathcal H^{n-1}(\xi)+\frac1{|y|}\frac{\partial}{\partial\varphi}\bigg|_{\varphi=0}\frac1{|B_\varphi y|}\int_{|\xi|=1 \atop B_\varphi y\cdot\xi>0}\log\bigg(\frac{B_\varphi y\cdot\xi}{|B_\varphi y|}\bigg)\\
		&\quad \cdot\bigg(-\xi_3\bigg(\nabla\frac1Q(\xi)\cdot\frac{B_\varphi y}{|B_\varphi y|}\bigg)+\bigg(\frac{\xi\cdot B_\varphi y}{|B_\varphi y|}\,\partial_3\frac1Q(\xi)\bigg)\bigg)d\mathcal H^{n-1}(\xi).
		\end{split}
		\end{equation*}
		A change of variables and \eqref{B_phi} imply
		\begin{equation*}
		\begin{split}
		&\partial_3^2\,G(y_1,0,\cdots,0)-\frac1{2|y|^2}\int_{|\xi|=1}\frac1{Q(\xi)}d\mathcal H^{n-1}(\xi)\\
		&=\frac1{y_1}\frac1{|y|}\frac{\partial}{\partial\varphi}\bigg|_{\varphi=0}\int_{|\xi|=1 \atop \xi_1>0}\log(\xi_1)\bigg(-(\xi_1\sin\varphi+\xi_3\cos\varphi)\nabla\frac1{Q(B_\varphi\xi)}\cdot\begin{pmatrix}\cos\varphi\\0\\\sin\varphi\\0\\\cdots\\0\end{pmatrix}+\xi_1\,\partial_3\frac1{Q(B_\varphi\xi)}\bigg)d\mathcal H^{n-1}(\xi)\\
		&=\frac1{y_1^2}\int_{|\xi|=1 \atop \xi_1>0}\log(\xi_1)\bigg(-\xi_1\,\partial_1\frac1Q-\xi_3\,\partial_3\frac1Q+\xi_3^2\,\partial_1^2\frac1Q+\xi_1^2\,\partial_3^2\frac1Q-2\xi_1\xi_3\,\partial^2_{13}\frac1Q\bigg)d\mathcal H^{n-1}(\xi).
		\end{split}
		\end{equation*}
		As the same holds for any $k\in\{2,\cdots,n\}$, we infer
		\begin{equation*}
		\begin{split}
		\partial_3^2\,G(y_1,0,\cdots,0)&=\frac1{2|y|^2}\int_{|\xi|=1}\frac1{Q(\xi)}d\mathcal H^{n-1}(\xi)+\frac1{|y|^2}\int_{|\xi|=1 \atop \xi_1>0}\log\bigg(\xi\cdot\frac y{|y|}\bigg)\\
		&\quad \cdot\bigg(-\xi_1\,\partial_1\frac1Q-\xi_k\,\partial_k\frac1Q+\xi_k^2\,\partial_1^2\frac1Q+\xi_1^2\,\partial_k^2\frac1Q-2\xi_1\xi_k\,\partial^2_{1k}\frac1Q\bigg)d\mathcal H^{n-1}(\xi).
		\end{split}
		\end{equation*}
		This together with \eqref{dG_11} yields (omitting from now on the differentials)
		\begin{equation*}
		\begin{split}
		\Delta G(y_1,0,\cdots,0)&=\frac{n-2}{2|y|^2}\int_{|\xi|=1}\frac1Q+\frac1{|y|^2}\int_{|\xi|=1 \atop \xi_1>0}\log\bigg(\xi\cdot\frac y{|y|}\bigg)\bigg[-(n-1)\xi_1\,\partial_1\frac1Q\\
		&\quad -\sum_{k=2}^n\xi_k\partial_k\frac1Q+\bigg(\sum_{k=2}^n\xi_k^2\bigg)\partial_1^2\frac1Q+\xi_1^2\bigg(\sum_{k=2}^n\partial_k^2\frac1Q\bigg)-2\xi_1\bigg(\sum_{k=2}^n\xi_k\partial_k\bigg(\partial_1\frac1Q\bigg)\bigg)\bigg]\\
		&=\frac{n-2}{2|y|^2}\int_{|\xi|=1}\frac1Q+\frac1{|y|^2}\int_{|\xi|=1 \atop y\cdot\xi>0}\log\bigg(\xi\cdot\frac y{|y|}\bigg)\bigg[-(n-2)\xi_1\,\partial_1\frac1Q-\xi\cdot\nabla\frac1Q\\
		&\quad +\partial_1^2\frac1Q+\xi_1^2\Delta\frac1Q-2\xi_1\,\xi\cdot\nabla\bigg(\partial_1\frac1Q\bigg)\bigg].
		\end{split}
		\end{equation*}
		\vskip0.2truecm
		\noindent\textbf{Step 4.} Let now $y\in\R^n\setminus\{0\}$. Reasoning as in Step 2 and using the same notation as there we obtain
		\begin{equation}\label{EVEN_laststep}
		\begin{split}
		\Delta G(y)&=\Delta\tilde G(B^Ty)=\Delta\tilde G((|y|,0,\cdots,0)^T)=
		\frac{n-2}{2|y|^2}\int_{|\xi|=1}\frac1Q+\frac1{|y|^2}\int_{|\xi|=1 \atop \xi_1>0}\log(\xi\cdot e_1)\\
		&\quad \cdot\bigg[-(n-2)(\xi\cdot e_1)\bigg(e_1\cdot\bigg(\nabla\frac1Q(B\xi)\cdot B\bigg)-\xi\cdot\bigg(\nabla\frac1Q(B\xi)\cdot B\bigg)+e_1^T\cdot B^TD\nabla^2\frac1Q(B\xi)B\cdot e_1\\
		&\quad\quad  +(\xi\cdot e_1)^2\Delta\frac1Q(B\xi)-2(\xi\cdot e_1)\bigg(\xi^T\cdot B^T\nabla^2\frac1Q(B\xi)B\cdot e_1\bigg)\bigg].
		\end{split}
		\end{equation}
		Recalling that $B\cdot e_1=\frac y{|y|}$, \eqref{EVEN_laststep} implies
		\begin{equation*}
		\begin{split}
		\Delta G(y) &= \frac{n-2}{2} \frac{1}{|y|^2}\int_{|\xi|=1}\frac1{Q(\xi)}\,d\mathcal H^{n-1}(\xi)
		+\frac{1}{|y|^2}\int_{|\xi|=1 \atop y\cdot \xi>0} \log\bigg(\xi\cdot \frac{y}{|y|}\bigg)\\
		&\quad \cdot \bigg[-(n-2)\bigg(\xi\cdot \frac{y}{|y|}\bigg)\bigg(\nabla \frac{1}{Q(\xi)}\cdot \frac{y}{|y|}\bigg)
		-\xi \cdot \nabla \frac{1}{Q(\xi)}+ \frac{y^T}{|y|} \cdot\nabla^2\frac{1}{Q(\xi)}\cdot \frac{y}{|y|}\\
		&\quad \quad +\bigg(\xi\cdot \frac{y}{|y|}\bigg)^2\Delta \frac{1}{Q(\xi)}
		-2 \bigg(\xi\cdot\frac y{|y|}\bigg)\bigg(\xi^T \cdot \nabla^2\frac1{Q(\xi)}\cdot \frac{y}{|y|}\bigg)
		\bigg]\,d\mathcal H^{n-1}(\xi).
		\end{split}
		\end{equation*}
		Finally, due to the $2m$-homogeneity of $Q$ by means of Lemma~\ref{Homog_der} we have
		$$
		\forall \xi \in \mathbb{S}^{n-1}:\quad \xi \cdot \nabla \frac{1}{Q(\xi)}=-2m\frac{1}{Q(\xi)},
		\quad \xi^T\cdot\nabla^2\frac{1}{Q(\xi)}\cdot \frac{y}{|y|}=-(2m+1)\left(\nabla \frac{1}{Q(\xi)}\cdot\frac{y}{|y|} \right) .
		$$
		Hence the previous formula simplifies to
		\begin{align}\label{DeltaG}
		\Delta G(y) =& \frac{n-2}{2} \frac{1}{|y|^2}\int_{|\xi|=1}\frac{1}{Q(\xi)}\,d\mathcal{H}^{n-1}(\xi)
		+\frac{1}{|y|^2}\int_{|\xi|=1 \atop y\cdot \xi>0} \log\left(\xi\cdot \frac{y}{|y|}\right)\nonumber \\
		&\quad \cdot  \left( (4m+4-n) \left(\xi\cdot \frac{y}{|y|}\right)\left(\nabla \frac{1}{Q(\xi)}\cdot \frac{y}{|y|}\right)
		+2m\frac{1}{Q(\xi)}\right.\\
		&\left.\quad\quad + \frac{y^T}{|y|} \cdot \nabla^2\frac{1}{Q(\xi)}\cdot \frac{y}{|y|} +\left(\xi\cdot \frac{y}{|y|}\right)^2\Delta \frac{1}{Q(\xi)}
		\right) \,d\mathcal{H}^{n-1}(\xi)\nonumber
		\end{align}
		and, recalling \eqref{pole0}, the proof is concluded.

	
	\section{Further examples: Sign change of suitable fundamental solutions for $\boldsymbol{n\ge 2m+2}$}\label{Section_EVEN}
	
	Proposition~\ref{prop:app:4.2} shows that in any space dimension $n\ge 2m+2$
	there exists an elliptic operator $L$ of the form (\ref{eq:ell_op})
	in $\mathbb{R}^{n}$ with corresponding fundamental solution $K_L$,
	where one finds a vector $y\in \mathbb{R}^{n}\setminus \{0\}$ such that $K_{L}(0,y)<0$, provided we are able to construct such an operator of
	order $2m$ in space dimension $n=2m+2$.
	Together with Theorem~\ref{thm:pos_dir} this will show that $K_L$ is sign changing
	near the origin, i.e. near its singularity. In view of Theorem~\ref{thm:pos_dir}
	this will prove Theorem~\ref{thm:negative_directions_any_dimension}.
	\vskip0.2truecm
	The starting point to obtain such a result is formula (\ref{formulaJohnEVEN}),
	according to which we have in dimension $n=2m+2$
	\begin{equation*}
	K_{L}(0,y)= \frac1{(2\pi )^n} \Delta G(y),
	\end{equation*}
	where $G$ is defined in \eqref{G} and $\Delta G(y)$ is calculated
	in \eqref{DeltaG}.
	
	%
	%

	We consider the symbol $$Q_\alpha(\xi',\xi_n)=|\xi'|^{2m}-\alpha|\xi'|^{2m-2}\xi_n^2+\xi_n^{2m},$$
	which for $m=2$ reduces to the symbol of the operator $L$ as in \eqref{El}. First, we find a threshold parameter $\alpha_m^*>0$ so that $Q_\alpha$ is elliptic for $\alpha<\alphastar$. Then, for such symbols we compute $\Delta G$ in a suitable point. Finally, exploiting the form of the ``limiting'' symbol $Q_{\alphastar}$, we will find that for $\alpha<\alphastar$ but close to it the sign of $\Delta G$ in such a point is negative. Notice that for such $\alpha$ the sublevels of $Q_\alpha$ are non-convex. Together with the observation that $\text{sgn}(K_L(0,y))=\text{sgn}(\Delta G(y))$ for any $y\in\R^n$ - an immediate consequence of \eqref{formulaJohnEVEN_2} - this proves the existence of operators of order $2m$ in $\R^{2m+2}$ whose fundamental solution attains negative values in some directions. 
	
	\begin{lem}\label{Qalpha_elliptic}
		$Q_\alpha$ is a symbol of an elliptic operator provided $\alpha<\alphastar:=m(m-1)^{\frac1m-1}$.
	\end{lem}
	\begin{proof}
		Since ellipticity for $\alpha\le 0$ is obvious, we consider only $\alpha>0$. Notice that $Q_\alpha(\xi',0)=|\xi'|^{2m}>0$ and that $Q_\alpha(\xi',\xi_n)=Q_\alpha(\xi',|\xi_n|)$, so we may assume that $\xi_n>0$.
		
		Let us write $\xi_n=\sqrt s|\xi'|$ for $s>0$, so that $Q_\alpha(\xi',\sqrt s|\xi'|)=|\xi'|^{2m}f(s)$, where $f(s):=s^m-\alpha s+1$. Then $f'(s)>0$ provided $s>\big(\frac\alpha m\big)^\frac 1{m-1}$. This implies that $Q_\alpha$ is elliptic if and only if
		$$f\left(\left(\frac\alpha m\right)^\frac 1{m-1}\right)=\alpha^{\frac m{m-1}}\left(m^{-\frac m{m-1}}-m^{-\frac 1{m-1}}\right)+1>0,$$
		namely for $\alpha<m(m-1)^{\frac1m-1}=\alphastar$.
	\end{proof}
	\begin{remark}\label{Remark_alphastar}
		Notice that, as a consequence of the proof of Lemma \ref{Qalpha_elliptic}   the symbol $Q_{\alphastar}$ is degenerate elliptic and that $f$ vanishes of order 2 at $\left(\frac{\alphastar}m\right)^{\frac 1{m-1}}=(m-1)^{-\frac 1m}=:\gamma_m$. Indeed, $f''(\gamma_m)>0$.
		
		In this section we need to consider $\alpha<\alphastar$ but close 
		to $\alphastar$. The sublevel sets of such $Q_{\alpha}$ are nonconvex. However, as pointed out in \cite[p. 100]{Davies} and Prop.~\ref{prop:4thorder_signchange} above one may also have sign changing fundamental solutions for $\alpha<0$ in dimensions $n\ge 2m+4$, where the sublevel sets of  $Q_{\alpha}$ are convex.
		See Figure \ref{Fig_symbols}.
	\end{remark}

	\begin{figure}[h!]
		\centering
		\begin{subfigure}[b]{0.3\linewidth}
			\includegraphics[width=\linewidth]{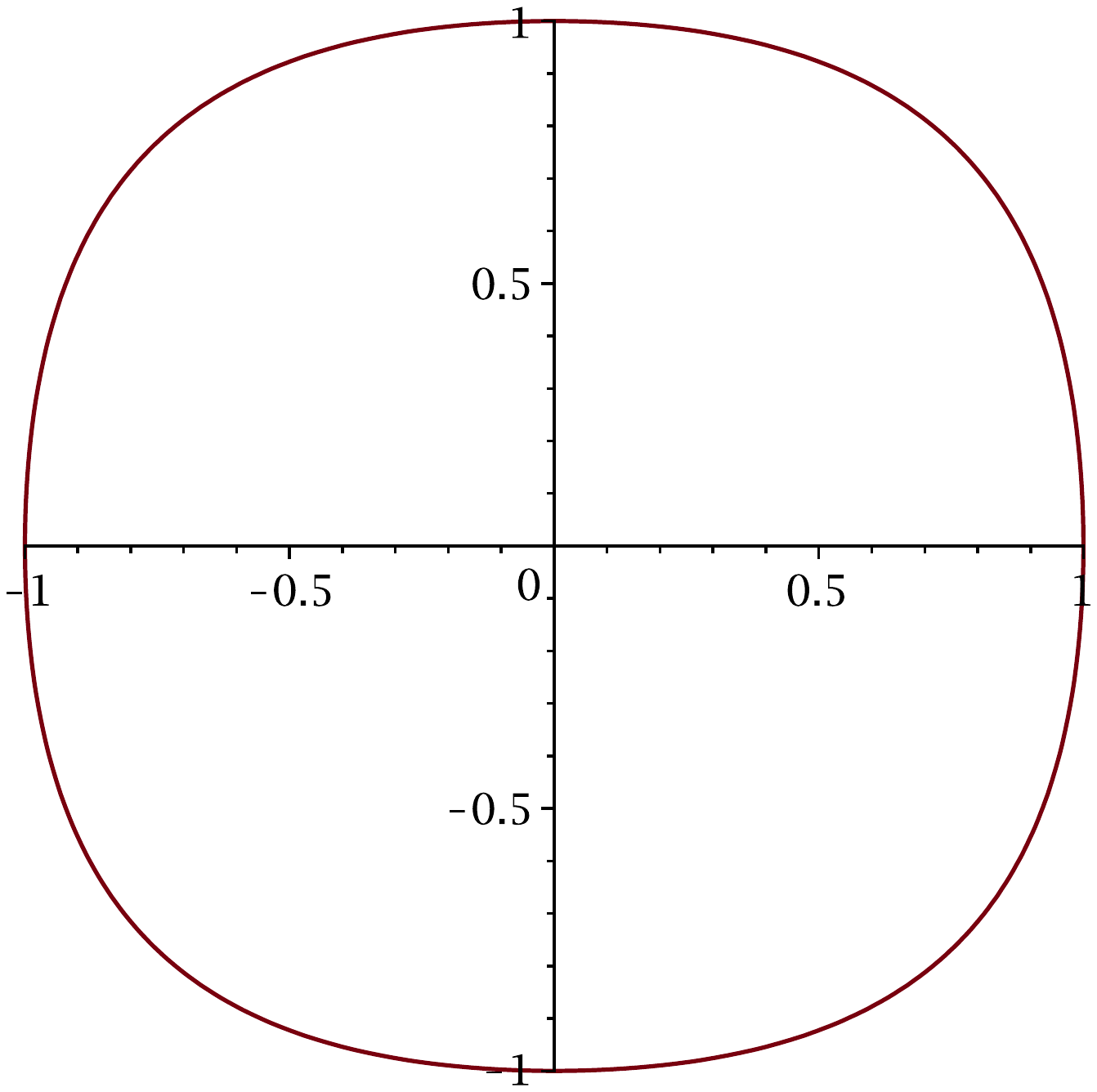}
			\vspace{-3cm}
			\caption{{$m=2$, $\alpha=-1$}}
		\end{subfigure}
		\qquad\qquad\quad
		\begin{subfigure}[b]{0.3\linewidth}
			\includegraphics[width=\linewidth]{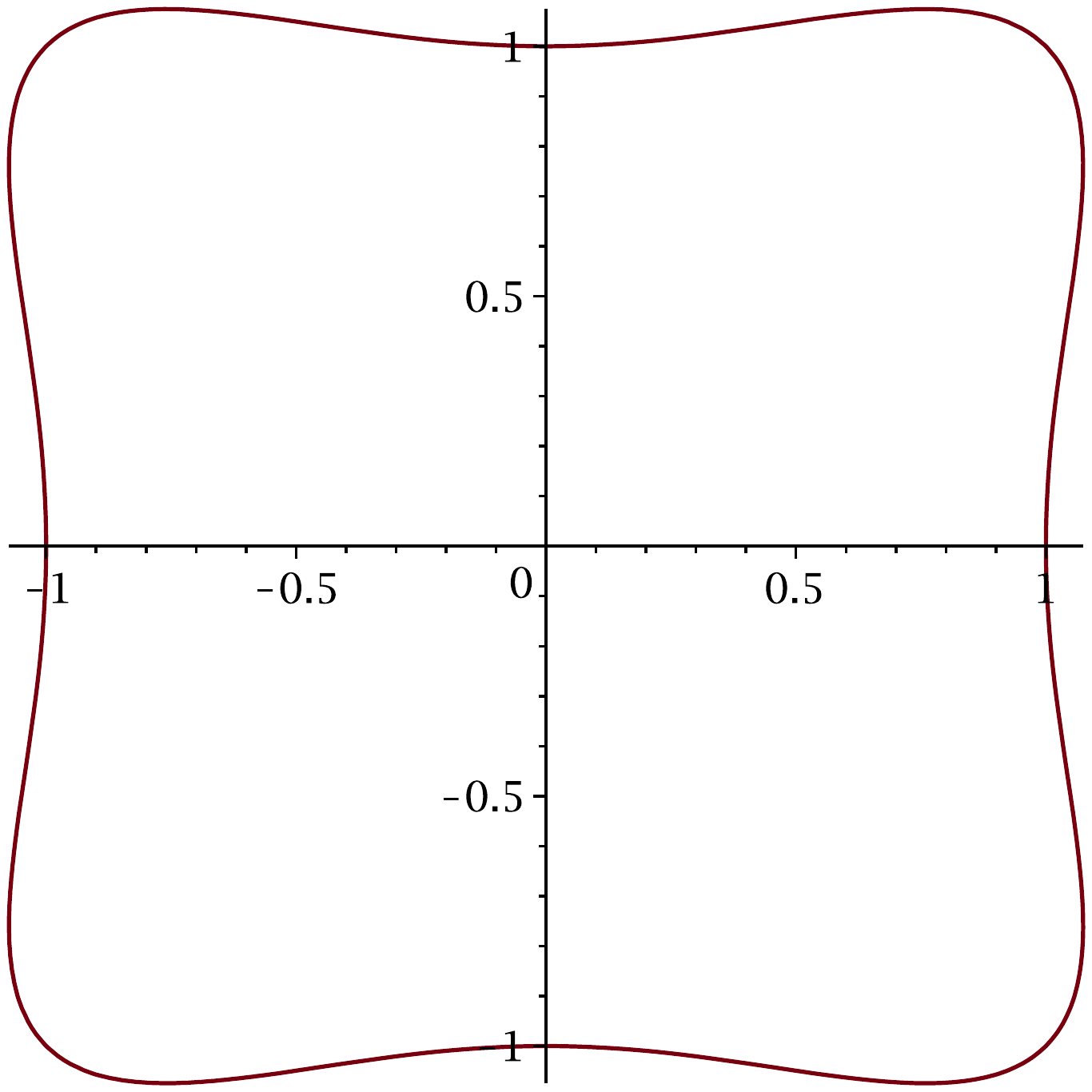}
			\vspace{-3cm}
			\caption{{$m=2$, $\alpha=1$}}
		\end{subfigure}
		\caption{{The shape of sublevel sets of $Q_{\alpha}$ for $\alpha<0$ (left) and $\alpha>0$ (right).}}
		\label{Fig_symbols}
	\end{figure}

	\paragraph{Computation of $\Delta G_\alpha(e_n)$.}
	By the peculiar form of $Q_{\alpha}$ we choose  $y=e_n$ and we write $G_\alpha$ instead of $G$ to stress the dependence on the parameter $\alpha$. By \eqref{DeltaG}, recalling that all integrals are on $S:=\mathbb S^{n-1}$, so $\xi_n=\sqrt{1-|\xi'|^2}$,
	we get
	\begin{equation}\label{DeltaG_S}
	\begin{split}
	\Delta G_\alpha(e_n) &= \int_{|\xi'|<1}\frac{2m}{Q_\alpha(\xi',\sqrt{1-|\xi'|^2})}\,\frac{d\xi'}{\sqrt{1-|\xi'|^2}}
	+\frac12\int_{|\xi'|<1}\frac{\log\left(1-|\xi'|^2\right)}{\sqrt{1-|\xi'|^2}}\\
	&\quad \cdot \left((2m+2)\sqrt{1-|\xi'|^2}\, \partial_n \frac1{Q_\alpha}(\xi',\sqrt{1-|\xi'|^2})
	+\frac{2m}{Q_\alpha(\xi',\sqrt{1-|\xi'|^2})}\right.\\
	&\left.\quad\quad +(2-|\xi'|^2) \partial_n^2\frac1{Q_\alpha}(\xi',\sqrt{1-|\xi'|^2})+\left(1-|\xi'|^2\right)\Delta'\frac1{Q_\alpha}(\xi',\sqrt{1-|\xi'|^2})
	\right)\, d\xi',
	\end{split}
	\end{equation}
	where $\Delta':=\sum_{k=1}^{n-1}\partial_k^2$. Exploiting the form of $Q_\alpha$ and writing $Q_\alpha(\xi'):=Q_\alpha(\xi',\sqrt{1-|\xi'|^2})$, one has
	\begin{equation}\label{dQ_1}
	\begin{split}
	\partial_n\frac1{Q_\alpha}\bigg|_S&=-\frac{\partial_nQ_\alpha}{Q_\alpha^2}\bigg|_S=-\frac{2m\xi_n^{2m-1}-2\alpha\xi_n|\xi'|^{2m-2}}{Q_\alpha(\xi)^2}\,\bigg|_S\\
	&=\frac{-2m(1-|\xi'|^2)^{\frac{2m-1}2}+2\alpha|\xi'|^{2m-2}(1-|\xi'|^2)^{1/2}}{Q_\alpha(\xi')^2},
	\end{split}
	\end{equation}
	\begin{equation}\label{dQ_11}
	\begin{split}
	\partial_n^2\frac1{Q_\alpha}\bigg|_S&=\bigg(-\frac{\partial_n^2Q}{Q_\alpha^2}+\frac{2(\partial_nQ_\alpha)^2}{Q_\alpha^3}\bigg)\bigg|_S\\
	&=\frac{-2m(2m-1)\xi_n^{2m-2}+2\alpha|\xi'|^{2m-2}}{Q_\alpha(\xi)^2}\,\bigg|_S+\frac{2(2m\xi_n^{2m-1}-2\alpha\xi_n|\xi'|^{2m-2})^2}{Q_\alpha(\xi)^3}\,\bigg|_S\\
	&=\frac{-2m(2m-1)(1-|\xi'|^2)^{m-1}+2\alpha|\xi'|^{2m-2}}{Q_\alpha(\xi')^2}\\
	&\quad+2\cdot\frac{4m^2(1-|\xi'|^2)^{2m-1}+4\alpha^2(1-|\xi'|^2)|\xi'|^{4m-4}-8\alpha m|\xi'|^{2m-2}(1-|\xi'|^2)^m}{Q_\alpha(\xi')^3}.
	\end{split}
	\end{equation}
	Analogously we find for $k=1,\ldots,n-1$
	\begin{equation*}\label{dQ_k}
	\begin{split}
	\partial_k\frac1{Q_\alpha}\bigg|_S&=-\frac{\partial_kQ_\alpha}{Q_\alpha^2}\bigg|_S
	=\frac{(2m-2)\alpha\xi_n^2|\xi'|^{2m-4}\xi_k-2m|\xi'|^{2m-2}\xi_k}{Q_\alpha(\xi)^2}\,\bigg|_S\\
	&=\frac{(2m-2)\alpha(1-|\xi'|^2)|\xi'|^{2m-4}\xi_k-2m|\xi'|^{2m-2}\xi_k}{Q_\alpha(\xi')^2},
	\end{split}
	\end{equation*}
	\begin{equation*}\label{dQ_kk}
	\begin{split}
	\partial_k^2\frac1{Q_\alpha}\bigg|_S&=\bigg(-\frac{\partial_k^2Q_\alpha}{Q_\alpha^2}+\frac{2(\partial_kQ_\alpha)^2}{Q_\alpha^3}\bigg)\bigg|_S\\
	&=\frac{(2m-2)\alpha\xi_n^2\left((2m-4)|\xi'|^{2m-6}\xi_k^2+|\xi'|^{2m-4}\right)-2m\left(|\xi'|^{2m-2}+(2m-2)|\xi'|^{2m-4}\xi_k^2\right)}{Q_\alpha(\xi)^2}\,\bigg|_S\\
	&
	\quad+2\cdot\frac{(2m-2)^2\alpha^2\xi_n^4|\xi'|^{4m-8}\xi_k^2+4m^2|\xi'|^{4m-4}\xi_k^2-4m(2m-2)\alpha\xi_n^2|\xi'|^{4m-6}\xi_k^2}{Q_\alpha(\xi)^3}\,\bigg|_S\\
	&=\frac{(2m-2)\alpha(1-|\xi'|^2)\left((2m-4)|\xi'|^{2m-6}\xi_k^2+|\xi'|^{2m-4}\right)-2m\left(|\xi'|^{2m-2}+(2m-2)|\xi'|^{2m-4}\xi_k^2\right)}{Q_\alpha(\xi')^2}\\
	&\quad+2\cdot\frac{(2m-2)^2\alpha^2(1-|\xi'|^2)^2|\xi'|^{4m-8}\xi_k^2+4m^2|\xi'|^{4m-4}\xi_k^2-4m(2m-2)\alpha(1-|\xi'|^2)|\xi'|^{4m-6}\xi_k^2}{Q_\alpha(\xi')^3}.
	\end{split}
	\end{equation*}
	Therefore,
	\begin{equation}\label{Delta'Q}
	\begin{split}
	\Delta'\frac1{Q_\alpha}\bigg|_S&=\sum_{k=1}^{n-1} \partial_k^2\frac1{Q_\alpha}\bigg|_S=\frac{(2m-2)(4m-3)\alpha(1-|\xi'|^2)|\xi'|^{2m-4}-2m(4m-1)|\xi'|^{2m-2}}{Q_\alpha(\xi')^2}\\
	&\quad+\frac{2|\xi'|^{4m-8}}{Q_\alpha(\xi')^3}\big((2m-2)^2\alpha^2(1-|\xi'|^2)^2|\xi'|^2+4m^2|\xi'|^6-4m(2m-2)\alpha|\xi'|^4(1-|\xi'|^2)\big).
	\end{split}
	\end{equation}
	We insert \eqref{dQ_1}-\eqref{Delta'Q} into \eqref{DeltaG_S} and write it in polar coordinates. Let $\sigma_n$ denote as before the $(n-1)$-dimensional volume of the unit sphere in $\mathbb{R}^n$. We obtain
	\begin{equation}\label{longintegral}
	\begin{split}
	\frac{\Delta G_\alpha(e_n)}{\sigma_{n-1}}&=\int_0^1\frac{2mr^{2m}}{\sqrt{1-r^2}Q_\alpha(r)}dr+\int_0^1\frac{\log(1-r^2)r^{2m}}{\sqrt{1-r^2}}\bigg[(2m+2)\bigg(\frac{-m(1-r^2)^m+\alpha (1-r^2)r^{2m-2}}{Q_\alpha(r)^2}\bigg)\\
	&\quad+\frac m{Q_\alpha(r)}-(2-r^2)\frac{m(2m-1)(1-r^2)^{m-1}-\alpha r^{2m-2}}{Q_\alpha(r)^2}\\
	&\quad+4(2-r^2)\frac{m^2(1-r^2)^{2m-1}+\alpha^2(1-r^2)r^{4m-4}-2\alpha mr^{2m-2}(1-r^2)^m}{Q_\alpha(r)^3}\\
	&\quad+\frac{4(1-r^2)r^{4m-6}}{Q_\alpha(r)^3}\big((m-1)^2\alpha^2(1-r^2)^2+m^2r^4-2m(m-1)\alpha r^2(1-r^2)\big)\\
	&\quad+(1-r^2)\frac{(m-1)(4m-3)\alpha(1-r^2)r^{2m-4}-m(4m-1)r^{2m-2}}{Q_\alpha(r)^2}\bigg]dr\\
	&=\int_0^1\frac{2mr^{2m}}{\sqrt{1-r^2}Q_\alpha(r)}dr+\int_0^1\frac{\log(1-r^2)r^{2m}}{\sqrt{1-r^2}}\bigg[\frac m{Q_\alpha(r)}+\frac{N_2(r^2)}{Q_\alpha(r)^2}+\frac{4(1-r^2)N_3(r^2)}{Q_\alpha(r)^3}\bigg]\, dr,
	\end{split}
	\end{equation}
	where
	\begin{equation}\label{N_2}
	\begin{split}
	N_2(t)=&-(2m+2)(1-t)\big(m(1-t)^{m-1}-\alpha t^{m-1}\big)+(2-t)(-m(2m-1)(1-t)^{m-1}+\alpha t^{m-1})\\
	&+(1-t)t^{m-2}\big((m-1)(4m-3)\alpha(1-t)-m(4m-1)t\big)
	\end{split}
	\end{equation}
	and, after some algebra,
	\begin{equation}\label{N_3}
	\begin{split}
	N_3(t)&=(2-t)\big(m(1-t)^{m-1}-\alpha t^{m-1}\big)^2+t^{2m-3}\big((m-1)\alpha(1-t)-mt\big)^2\\
	&=:(2-t)R_1(t)^2+t^{2m-3}R_2(t)^2.
	\end{split}
	\end{equation}
	
	The goal is to understand the behaviour of $\Delta G_\alpha(e_n)$ as $\alpha\uparrow\alphastar$. By Remark \ref{Remark_alphastar}, we know that $Q_{\alphastar}(r):=Q_{\alphastar}(r,\sqrt{1-r^2})=\tilde Q(r^2)\big(1-(1+\gamma_m)r^2\big)^2$ where $\tilde Q$ is a positive polynomial of degree $m-2$. Actually, it is easy to show that (cf. Lemma \ref{Qalpha_elliptic})
	$$f(s)\big|_{\alpha=\alphastar}=s^m-\alphastar s+1=\left(\sum_{k=0}^{m-2}\frac{m-1-k}{m-1}\gamma_m^{-k-2}s^k\right)(s-\gamma_m)^2.$$
	Therefore, recalling the value of $\gamma_m=(m-1)^\frac1m$ and substituting $s=\frac{1-r^2}{r^2}$, we get
	$$Q_{\alphastar}(r)=\left(\sum_{k=0}^{m-2}\frac{m-1-k}{(m-1)^{\frac{m-k-2}m}}\,(1-r^2)^k\,r^{2m-4-2k}\right)\big(1-(1+\gamma_m)r^2\big)^2.$$
	
	Notice that the singularity that $Q_{\alphastar}$ would produce at $r_0:=(1+\gamma_m)^{-1/2}$ is not integrable. Moreover, we shall see that, although the numerator of the second integral in \eqref{longintegral} vanishes precisely at the same point, it is not strong enough to compensate such a singularity.
	
	\paragraph{Computation of $\Delta G_{\alphastar}(e_n)$.}
	Let $t=r^2$, $t_0:=r_0^2=\frac1{1+\gamma_m}$ and define
	\begin{equation}\label{N}
	N(t):=m\Qstar(t)^2+N_2(t)\Qstar(t)+4(1-t)N_3(t),
	\end{equation}
	where $\Qstar(t):=Q_{\alphastar}\left(\sqrt t\right)$.
	
	\paragraph{Step 1: $N(t_0)=0$.} Because $\Qstar(t_0)=0$, we just need to show that $N_3(t_0)=0$. First,
	\begin{equation}\label{R_1}
	R_1(t_0)=m\left(\frac{\gamma_m}{1+\gamma_m}\right)^{m-1}-\alphastar\left(\frac1{1+\gamma_m}\right)^{m-1}=\frac{m\gamma_m^{m-1}-\alphastar}{(1+\gamma_m)^{m-1}}=0
	\end{equation}
	by the definitions of $\alphastar$ and $\gamma_m$. Moreover,
	\begin{equation}\label{R_2}
	R_2(t_0)=(m-1)\alphastar\frac{\gamma_m}{1+\gamma_m}-m\frac1{1+\gamma_m}=\frac m{1+\gamma_m}\left((m-1)^{\frac1m}\gamma_m-1\right)=0.
	\end{equation}
	
	\paragraph{Step 2: $N'(t_0)=0$.} Since $t_0$ is a zero of $\Qstar$ of order $2$,
	\begin{equation*}
	\begin{split}
	N'(t_0)&=2m\Qstar(t_0)\Qstar'(t_0)+N_2'(t_0)\Qstar(t_0)+N_2(t_0)\Qstar'(t_0)+4(1-t_0)N_3'(t_0)-4N_3(t_0)\\
	&=4(1-t_0)N_3'(t_0).
	\end{split}
	\end{equation*}
	By \eqref{N_3} and \eqref{R_1}-\eqref{R_2} it is clear that $N_3'(t_0)=0$.
	
	\paragraph{Step 3: $N''(t_0)\not=0$.} Similarly as before, by \eqref{N} and $\Qstar(t_0)=\Qstar'(t_0)=N_3(t_0)=N_3'(t_0)=0$ we get
	\begin{equation}\label{ddN}
	N''(t_0)=N_2(t_0)\Qstar''(t_0)+4(1-t_0)N_3''(t_0).
	\end{equation}
	Because of \eqref{N_3} and \eqref{R_1}-\eqref{R_2}, we have
	\begin{equation}\label{ddN3}
	\begin{split}
	N_3''(t_0)&=2(2-t_0)\left(R_1'(t_0)\right)^2+2t_0^{2m-3}\left(R_2'(t_0)\right)^2\\
	&=2\frac{2\gamma_m+1}{1+\gamma_m}(m-1)^2\left(\frac{m\gamma_m^{m-2}+m(m-1)^{\frac1m-1}}{(1+\gamma_m)^{m-2}}\right)^2+2\frac{\left(m(m-1)^\frac1m+m\right)^2}{(1+\gamma_m)^{2m-3}}\\
	&=\frac{2m^2}{(1+\gamma_m)^{2m-3}}\left(\left(2(m-1)^{-\frac1m}+1\right)(m-1)^\frac2m\left((m-1)^{\frac1m}+1\right)^2+\left((m-1)^{\frac1m}+1\right)^2\right)\\
	&=\frac{2m^2}{(1+\gamma_m)^{2m-3}}\left((m-1)^{\frac1m}+1\right)^4.
	\end{split}
	\end{equation}
	Next, we may rewrite
	\begin{equation}\label{N_2bis}
	N_2(t)=-(2m+2)(1-t)R_1(t)+(2-t)M_1(t)+(1-t)t^{m-2}M_2(t),
	\end{equation}
	where
	$$M_1(t)=-m(2m-1)(1-t)^{m-1}+\alphastar t^{m-1}$$
	and
	$$M_2(t)=(m-1)(4m-3)\alphastar(1-t)-m(4m-1)t.$$
	Evaluating on $t_0$, we get
	\begin{equation*}
	M_1(t_0)=\frac m{(1+\gamma_m)^{m-1}}\left(-(2m-1)(m-1)^{\frac1m-1}+(m-1)^{\frac1m-1}\right)=\frac{-2m(m-1)^\frac1m}{(1+\gamma_m)^{m-1}},
	\end{equation*}
	\begin{equation*}
	M_2(t_0)=(m-1)(4m-3)m(m-1)^{\frac1m-1}\frac{\gamma_m}{1+\gamma_m}-m(4m-1)\frac1{1+\gamma_m}=\frac{-2m}{1+\gamma_m}.
	\end{equation*}
	Since $R_1(t_0)=0$, from \eqref{N_2bis} we infer
	\begin{equation}\label{N_2t0}
	\begin{split}
	N_2(t_0)&=\frac{2\gamma_m+1}{1+\gamma_m}\cdot\frac{-2m(m-1)^\frac1m}{(1+\gamma_m)^{m-1}}+\frac{\gamma_m}{(1+\gamma_m)^{m-1}}\cdot\frac{-2m}{1+\gamma_m}\\
	&=\frac{-2m}{(1+\gamma_m)^m}\left[\left(2(m-1)^{-\frac1m}+1\right)(m-1)^\frac1m+(m-1)^{-\frac1m}\right]\\
	&=\frac{-2m(m-1)^{-\frac1m}}{(1+\gamma_m)^m}\left((m-1)^\frac1m+1\right)^2.
	\end{split}
	\end{equation}
	Next, $Q_{\alphastar}(r,\sqrt{1-r^2})|_{r=\sqrt{t}}=\Qstar(t)=t^m-\alphastar(1-t)t^{m-1}+(1-t)^m$, therefore
	\begin{equation}\label{ddQ}
	\begin{split}
	\Qstar''(t_0)&=m(m-1)(1-t_0)^{m-2}-\alphastar(m-1)(m-2)(1-t_0)t_0^{m-3}+(2\alphastar+m)(m-1)t_0^{m-2}\\
	&=\frac{m(m-1)}{(1+\gamma_m)^{m-2}}\left[(m-1)^{\frac2m-1}-(m-1)^{\frac1m-1}(m-2)(m-1)^{-\frac1m}+2(m-1)^{\frac1m-1}+1\right]\\
	&=\frac m{(1+\gamma_m)^{m-2}}\left((m-1)^\frac1m+1\right)^2.
	\end{split}
	\end{equation}
	Hence, according to \eqref{ddN}, \eqref{ddN3}, \eqref{N_2t0}, and \eqref{ddQ} we finally obtain
	\begin{equation*}
	N''(t_0)=\frac{6m^2(m-1)^{-\frac1m}}{(1+\gamma_m)^{2m-2}}\left((m-1)^\frac1m+1\right)^4>0.
	\end{equation*}
	As a consequence of Steps 1-3, we have thus proved that $t_0$ is a zero of $N$ of order 2 and, moreover, that it is also a minimum, thus $N(t)>0$ for $t$ close to $t_0$. This yields $\Delta G_{\alphastar}(e_n)=-\infty$: indeed, the second integral in \eqref{longintegral} has a non-integrable singularity at $r_0=\sqrt{t_0}\in(0,1)$ of the kind $(r-r_0)^{-4}$ which prevails on the one in the first integral, of the kind $(r-r_0)^{-2}$. The negative sign comes from $\log(1-r^2)<0$ as $r\in(0,1)$.
	
	\begin{proof}[Proof of Theorem \ref{TheoremEVEN_Examples}]
		By pointwise convergence $\Delta G_\alpha(e_n)\to\Delta G_{\alphastar}(e_n)$ as $\alpha\uparrow\alphastar$, we may apply Fatou's lemma and obtain
		$$\limsup_{\alpha\uparrow\alphastar}\Delta G_\alpha(e_n)\leq\Delta G_{\alphastar}(e_n)=-\infty$$
		and conclude that for $\alpha<\alphastar$ and \textit{close} to $\alphastar$ one has $\Delta G_\alpha(e_n)<0$. This behaviour is well observable in Figure \ref{Fig_EVEN}, where the graph of $\alpha\mapsto\Delta G_\alpha(e_n)$ is displayed for $m=2$ (here $\alpha_2^*=2$).
		\begin{figure}
			\centering
			\includegraphics[width=0.35\linewidth]{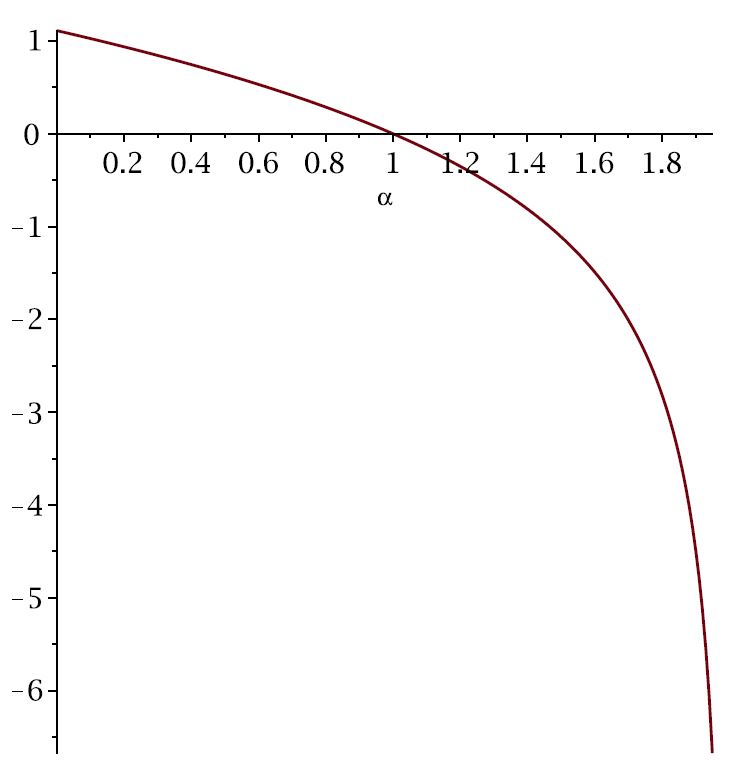}
			\caption{The graph of $\Delta G_\alpha(e_n)$ for $m=2$ with $\alpha\in[0,1.95]$.}
			\label{Fig_EVEN}
		\end{figure}
		The proof is completed recalling that $\text{sgn}(K_\alpha(0,e_n))=\text{sgn}(\Delta G_\alpha(e_n))$, $K_\alpha$ being the fundamental solution of the operator whose symbol is $Q_\alpha$.
	\end{proof}


\end{document}